\documentclass[12pt]{amsart}
\usepackage[headings]{fullpage}
\usepackage{amssymb,epic,eepic,epsfig,amsbsy,amsmath,amscd,color,svg}
\usepackage[all]{xy}
\usepackage{graphicx}
\usepackage{texdraw}
\usepackage{url}
\usepackage{bbm}
\usepackage{mathrsfs}
\usepackage{scalerel}
\newcommand\no[1]{ }

\usepackage{float}   
\usepackage{mathrsfs}
\usepackage[bookmarks=true,%
    colorlinks=true,%
    linkcolor=blue,%
    citecolor=blue,%
    filecolor=blue,
    menucolor=blue,%
    urlcolor=blue,%
    breaklinks=true]{hyperref}

\usepackage[T1]{fontenc}
\def\printname#1{
        \if\draft y
                \smash{\makebox[0pt]{\hspace{-0.5in}
                        \raisebox{8pt}{\tt\tiny #1}}}
        \fi }

\theoremstyle{plain}
\newtheorem{theorem}{Theorem}[section]
\newtheorem{thm}{Theorem}
\newtheorem{lemma}[theorem]{Lemma}
\newtheorem{corollary}[theorem]{Corollary}
\newtheorem{proposition}[theorem]{Proposition}

\newtheorem{definition}{Definition}
\theoremstyle{definition}
\newtheorem{remark}[theorem]{Remark}
\newtheorem{example}[theorem]{Example}

\def\BC{\mathbb C}
\def\BN{\mathbb N}
\def\BZ{\mathbb Z}
\def\BR{\mathbb R}

\def\BQ{\mathbb Q}
\DeclareMathOperator*{\foo}{\scalerel*{+}{\sum}}

\def\CI{\mathcal I}

\def\Cob{\mathrm{Cob}}

\def\M{{\mathcal M}}

\def\2Mor{\mathrm{2Mor}}
\def\Ssr{\overline{\cS}(\fS)}
\def\bcS{\overline{\cS}}

\def\OSL{{\mathcal O}_{q^2}(\mathrm{SL}(2))}
\newcommand\eDD[1]{{}_{#1}\Delta}

\def\la{\langle}
\def\ra{\rangle}

\def\rk{\mathrm{rk}}
\def\cM{\mathcal M}

\def\cS{\mathscr S}
\def\ot{\otimes}

\def\cE{\mathcal E}
\def\cF{\mathcal F}

\def\bk{\mathbf k}

\def\cP{\mathcal P}
\def\tD{\tilde D}

\def\Id{\mathrm{Id}}
\def\fS{\mathfrak S}
\def\bfS{\overline{\fS}}

\def\D{\Delta}

\def\cY{\mathcal Y}
\def\ev{{\mathrm{ev}}}

\def\embed{\hookrightarrow}

\def\bbS{\overline \Sigma}

\def\cY{\mathcal Y}

\def\ev{{\mathrm{ev}}}

\def\embed{\hookrightarrow}

\newcommand{\red}[1]{{#1}}

\def\ooS{\mathring {\cS}}

\def\Ss{\cSs}

\def\Tr{\mathrm{Tr}}

\DeclareMathOperator{\tr}{\mathrm tr}

\def\al{\alpha}
\def\ve{\varepsilon}
\def\be { \begin{equation} }
\def\ee { \end{equation} }

\def\bD{\underline{ \Delta} }

\def\btheta{\widetilde \theta  }

\def\B{\fB}
\def\P{\mathcal P}

\def\bT{\mathbb T}

\def\fl{\mathrm{fl}}

\def\nc{\newcommand}

\nc\FIGc[3]{\begin{figure}[htpb]
    \includegraphics[height=#3]{#1-eps-converted-to.pdf}
    \caption{#2}
    \label{fig:#1}
    \end{figure}}
    
    \nc\FIGceps[3]{\begin{figure}[htpb]
    \includegraphics[height=#3]{#1.eps}
    \caption{#2}
    \label{fig:#1}
    \end{figure}}
    
\newcommand\incl[2]{{\includegraphics[height=#1]{#2-eps-converted-to.pdf}}}

\def\leftve{\raisebox{-8pt}{\incl{.8 cm}{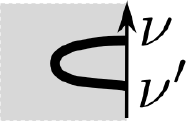}}}
\def\leftup{\raisebox{-8pt}{\incl{.8 cm}{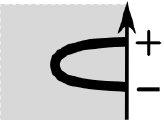}}}

\def\leftupPP{\raisebox{-8pt}{\incl{.8 cm}{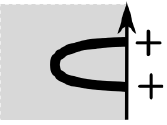}}}
\def\leftupNN{\raisebox{-8pt}{\incl{.8 cm}{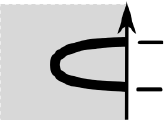}}}

\def\leftved{\raisebox{-8pt}{\incl{.8 cm}{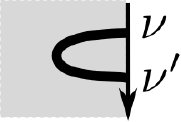}}}

\def\leftvedl{\raisebox{-8pt}{\incl{.8 cm}{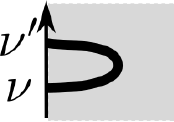}}}
\def\emptyr{\raisebox{-8pt}{\incl{.8 cm}{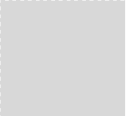}}}
\def\emptys{\raisebox{-8pt}{\incl{.8 cm}{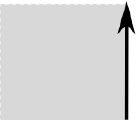}}}

\def\cross{  \raisebox{-8pt}{\incl{.8 cm}{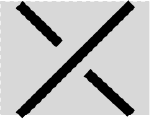}} }
\def\resoP{  \raisebox{-8pt}{\incl{.8 cm}{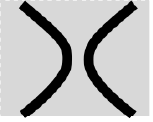}} }
\def\resoN{  \raisebox{-8pt}{\incl{.8 cm}{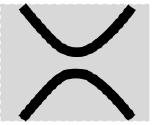}} }
\def\kinkp{  \raisebox{-8pt}{\incl{.8 cm}{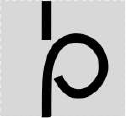}} }
\def\kinkn{  \raisebox{-8pt}{\incl{.8 cm}{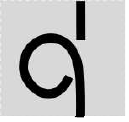}} }
\def\kinkzero{  \raisebox{-8pt}{\incl{.8 cm}{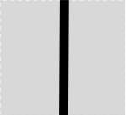}} }
\def\reordoneall{  \raisebox{-8pt}{\incl{.8 cm}{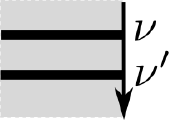}} }

\def\reordonepn{  \raisebox{-8pt}{\incl{.8 cm}{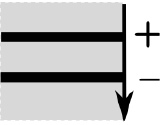}} }
\def\reordtwopn{  \raisebox{-8pt}{\incl{.8 cm}{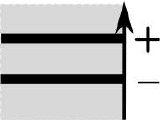}} }
\def\reordoneallp{  \raisebox{-8pt}{\incl{.8 cm}{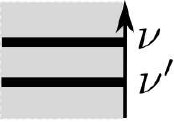}} }
\def\reordnp{  \raisebox{-8pt}{\incl{.8 cm}{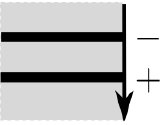}} }
\def\leftmp{  \raisebox{-8pt}{\incl{.8 cm}{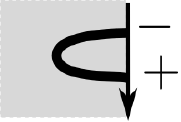}} }
\def\reordpn{  \raisebox{-8pt}{\incl{.8 cm}{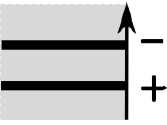}} }
\def\reordone{  \raisebox{-8pt}{\incl{.8 cm}{reord1}} }
\def\reordtwo{  \raisebox{-8pt}{\incl{.8 cm}{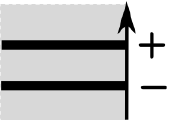}} }
\def\reordthree{  \raisebox{-8pt}{\incl{.8 cm}{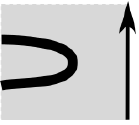}} }
\def\trivloop{  \raisebox{-8pt}{\incl{.8 cm}{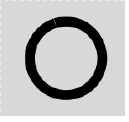}} }

\def\reordonez{  \raisebox{-8pt}{\incl{.8 cm}{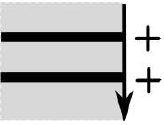}} }
\def\reordonea{  \raisebox{-8pt}{\incl{.8 cm}{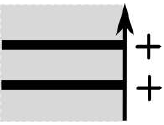}} }

\def\reordsixz{  \raisebox{-8pt}{\incl{.8 cm}{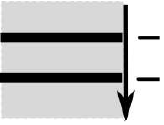}} }
\def\reordsixa{  \raisebox{-8pt}{\incl{.8 cm}{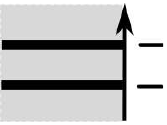}} }

\def\corone{ \raisebox{-12pt}{\incl{1.3 cm}{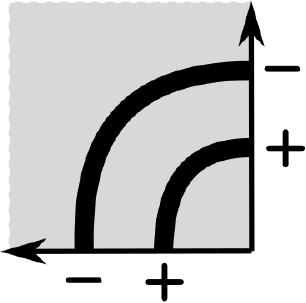}}}
\def\cortwo{ \raisebox{-12pt}{\incl{1.3 cm}{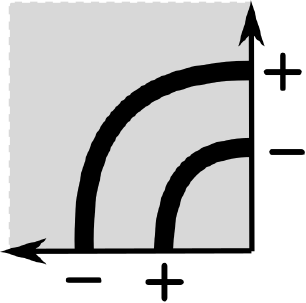}}}
\def\corthree{ \raisebox{-12pt}{\incl{1.3 cm}{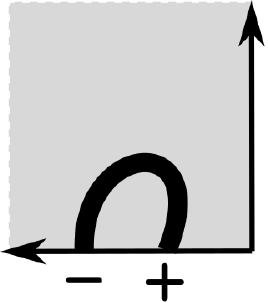}}}
\def\coronep{ \raisebox{-12pt}{\incl{1.3 cm}{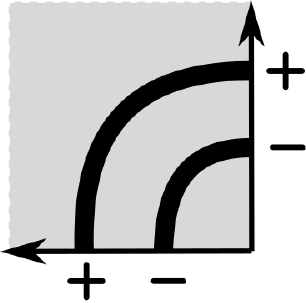}}}
\def\cortwop{ \raisebox{-12pt}{\incl{1.3 cm}{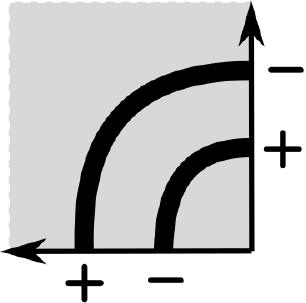}}}
\def\corthreep{ \raisebox{-12pt}{\incl{1.3 cm}{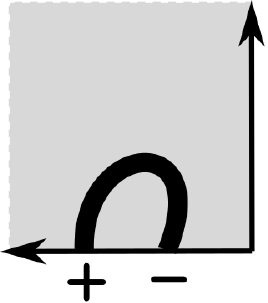}}}

\def\reordonegra{\raisebox{-12pt}{\incl{1.3 cm}{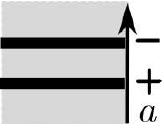}}}

\def\NSGraph{\mathsf{TopMultiCyc}}

\def\NSForest{\mathsf{TopForest}}

\def\Edge{\mathsf{Edges}}
\def\cut{\mathsf{cut}}
\def\Qq{{\BQ(q^{1/2})}}

\begin{document}

\title{Stated skein algebras of surfaces}
\author[Francesco Costantino]{Francesco Costantino}
\address{Institut de Math\'ematiques de Toulouse, 118 route de Narbonne,
F-31062 Toulouse, France}
\email{francesco.costantino@math.univ-toulouse.fr}

\author[Thang  T. Q. L\^e]{Thang  T. Q. L\^e}
\address{School of Mathematics, 686 Cherry Street,
 Georgia Tech, Atlanta, GA 30332, USA}
\email{letu@math.gatech.edu}

\date{\today}

\thanks{ Partially supported by ANR-11-LABX-0040-CIMI within the program ANR-11-IDEX-0002-02. \\
2010 {\em Mathematics Classification:} Primary 57N10. Secondary 57M25.}

\begin{abstract}{We study the algebraic and geometric properties of stated skein algebras of surfaces with punctured boundary. We prove that the skein algebra of the bigon is isomorphic to the quantum group ${\mathcal O}_{q^2}(\mathrm{SL}(2))$ \red{that provides} a topological interpretation for its structure morphisms. 
We also show that its stated skein algebra lifts in a suitable sense the Reshetikhin-Turaev functor and in particular we recover the dual $R$-matrix for ${\mathcal O}_{q^2}(\mathrm{SL}(2))$ in a topological way. 
We deduce that the skein algebra of a surface with $n$ boundary components is a \red{comodule-algebra} over ${\mathcal O}_{q^2}(\mathrm{SL}(2))^{\otimes{n}}$ and prove that cutting along an ideal arc corresponds to Hochshild cohomology of bicomodules. We give a topological interpretation of braided tensor product of stated skein algebras of surfaces as ``glueing on a triangle''; then we recover topologically some bialgebras in the category of ${\mathcal O}_{q^2}(\mathrm{SL}(2))$-comodules, among which the ``transmutation'' of ${\mathcal O}_{q^2}(\mathrm{SL}(2))$.
We also provide an operadic interpretation of stated skein algebras as an example of a ``geometric non symmetric modular operad''.   
In the last part of the paper we define a reduced version of stated skein algebras and prove that it allows to recover Bonahon-Wong's quantum trace map and interpret skein algebras in the classical limit when $q\to 1$ as regular functions over a suitable version of moduli spaces of twisted bundles.  }\end{abstract}

\maketitle

\def\pbbS{\partial \bbS}
\def\bSP{(\bbS,\cP)}
\def\cA{\mathcal A}
\def\cB{\mathcal B}
\def\Si{\Sigma}
\def\fB{\mathfrak B}
\def\pr{\mathrm{pr}}
\def\cO{\mathcal O}

\def\poS{\partial_0\Sigma}
\def\cSs{\cS}
\def\RcSs{\cS^{rel}_{\mathrm s}}
\def\zRcSs{\cS^{rel}_{\mathrm s,0}}

\def\cR{\mathcal R}
\def\basics{basic skein}
\def\YD{\cY(\D)}
\def\tYD{\tilde \cY(\D)}
\def\YtD{\cY^{(2)}(\D)}
\def\tYtD{\tilde \cY^{(2)}(\D)}
\def\bve{{\boldsymbol{\ve}}}
\def\bm{{\mathbf m}}
\def\hYeD{\tilde \cY^\ev(\D)}
\def\YeD{\cY^\ev(\D)}
\def\pfS{\partial \fS}
\def\upfS{\partial \underline{\fS}}

\def\pbfS{\partial \bfS}
\def\bove{\boldsymbol{\ve}}
\def\bomu{{\vec \mu}}
\def\bonu{{\vec \nu}}
\def\Gr{\mathrm{Gr}}
\def\bl{\mathbf{l}}
\def\onto{\twoheadrightarrow}
\def\Stink{\mathrm{St}^\uparrow(\bk)}
\def\Stinkp{\mathrm{St}^\uparrow(\bk')}
\def\sincr{s^\uparrow}
\def\St{\mathrm{St}}

\def\tF{\tilde {\cF}}
\def\tE{\tilde {\cE}}
\def\tfT{\tilde {\fT}}
\def\bcSs{\overline{\cS}_s}
\def\cN{\mathcal N}
\def\MN{(M,\cN)}
\def\pM{\partial M}
 \def\cSsp{\cS_{+}}
   \def\TrD{\Tr_\D}
   \def\hTrD{\widehat{\Tr}_\D}
   \def\hcSs{\widehat{\cS}_s}
   \def\SS{\cS(\fS)}
\tableofcontents

\section{Introduction}
This paper is devoted to study the notion of stated skein algebra of surfaces introduced by the second author in \cite{Le:TDEC} in order to reinterpret in skein theoretical terms the construction of the quantum trace by Bonahon and Wong \cite{BW} as well as incorporating Muller's version of skein algebra \cite{Muller}. 
Although the definition of the stated skein module applies to $3$-manifolds, this paper is entirely devoted to the case of surfaces: a forthcoming paper will describe how this fits in the framework of an extended topological field theory in dimensions $1,2,3$. 
Indeed the case of surfaces is sufficiently rich in algebraic and geometrical terms to deserve a separate treatment and we will now outline the results of this paper.

 \subsection{Skein algebras} Let $\cR= \BZ[q^{\pm 1/2}]$ be the ring of Laurent polynomials in a variable $q^{1/2}$.
 Suppose $\fS$ is the result of removing a finite number of points, called {\em punctures}, from a compact oriented 2-dimensional manifold with possibly non-empty boundary. 
 The ordinary skein algebra $\ooS(\fS)$, introduced by Przytycki~\cite{Prz} and Turaev~\cite{Turaev}, is defined to be the $\cR$-module generated by isotopy classes of framed unoriented links in $\fS \times (0,1)$  modulo the Kauffman relations \cite{Kauffman}
\begin{align}
\label{eq.skein0} \cross \ &= \ q\resoP + q^{-1} \resoN\\
\label{eq.loop0}  \trivloop\  &=\  (-q^2 -q^{-2})\emptyr.
 \end{align}
The product of two links $\al_1$ and $\al_2$ is the result of stacking $\al_1$ above $\al_2$.  The skein algebra  has played an important role in low-dimensional topology and quantum topology and it serves as a bridge between classical topology and quantum topology. The skein module has connections to the $\mathrm{SL}_2(\BC)$-character variety \cite{Bullock,PS1}, the quantum group of $\mathrm{SL}_2(\BC)$, the Witten-Reshetikhin-Turaev topological quantum field theory \cite{BHMV}, the quantum Teichm\"uller spaces \cite{CF,Kashaev,BW,Le:QT}, and the quantum cluster algebra theory \cite{Muller}.
 
 In the definition of the skein algebra $\ooS(\fS)$ the boundary $\pfS$ does not play any role, and we have $\ooS(\fS) = \ooS(\mathring {\fS})$, where $\mathring {\fS}$ is the interior of $\fS$. In an attempt to introduce excision into the study of the skein algebra, the second author  \cite{Le:TDEC} introduce the notion of {\em stated skein algebra}, denoted in this paper by  $\SS$, whose definition involves tangles properly embedded into $\fS \times (0,1)$. These tangles can have end-points only  on {\em boundary edges} of $\fS$, which are open intervals connected components of the boundary.
 For details see Section \ref{sec.def}.

A key result about stated skein algebras is that they behave well under cutting along an ideal arc. Here an {\em ideal arc} is a proper embedding $c: (0,1) \embed \fS$ (so that its end points are the punctures). 
Cutting $\fS$ along $c$ one gets a 2-manifold $\fS'$ whose boundary contains two open intervals $a$ and $b$ so that one can recover $\fS$ from $\fS'$ by gluing $a$ and $b$ together, see Figure \ref{fig:cut}.

\FIGc{cut}{Cutting $\fS$ along ideal arc $c$ to get $\fS'$, which might be disconnected}{3cm}
Then \cite[Theorem 1]{Le:TDEC} (see 
the splitting Theorem \ref{thm.1a} below)  says that  there is a natural injection of algebras 
\be  \theta_c: \cS(\fS)\hookrightarrow \cS(\fS'),  \label{eq.rho0}
\ee
given by a simple state sum. The extension from $\ooS(\fS)$ to $\SS$ is unique (or canonical) if one wants the splitting theorem and a consistency requirement to hold.
 
 \def\SB{\cS(\cB)}
 \def\USL{U_{q^2}(\mathfrak{sl}_2)}
 The paper is a systematic study of the stated skein algebra $\SS$. Let us now list the main results of the paper.
 
 \subsection{Bigon and quantum $\mathrm{SL}_2(\BC)$ coordinate ring} The quantized enveloping algebra $\USL$ and its Hopf dual $\OSL$, known as the quantum coordinate ring of the Lie group $\mathrm{SL}_2(\BC)$, play an important role in many branches of mathematics, see \cite{Kassel,MajidFoundation}. These algebras are usually defined by rather complicated presentations which are hard to comprehend.
 
A first consequence of the splitting theorem is that the quantum coordinate ring $\OSL$ can be described by simple geometric terms, namely, it is naturally isomorphic to the stated skein algebra of the bigon $\cB$, which is the standard disk without two points on its boundary, see Figure \ref{fig:bigon10}.

\FIGc{bigon10}{Left: bigon. Right: splitting the bigon along the dashed ideal arc}{3cm}
By splitting the bigon along an ideal arc $c$ (which is the dashed arc in Figure \ref{fig:bigon10}) we get a homomorphism $\Delta:=\theta_c$,
$$ \Delta: \SB \to \SB\otimes \SB,$$
which turns out to be compatible with the product and makes $\SB$ a bialgebra. Moreover, we will define using topological terms the counit, antipode and co-$R$-matrix which turns $\SB$ into a \red{``dual quasitriangular'' (see  \cite{MajidFoundation} Section 2.2) a.k.a. ``cobraided'' (see \cite{Kassel}, Section VIII.5)} Hopf algebra, and will prove the following.

\begin{thm}[Theorems \ref{teo:SL2} and \ref{teo:cobraid}] The \red{dual quasitriangular} Hopf algebra $\SB$ is isomorphic in a natural way to the quantum coordinate ring $\OSL$.
\label{thm.1}
\end{thm}

This result allows to use skein theoretical techniques to study $\OSL$. We will show that many complicated algebraic objects and facts  concerning the quantum groups $\OSL$ and $\USL$ have  simple transparent picture interpretations. For example, the above mentioned co-$R$-matrix has a very simple geometric picture description, see Theorem~\ref{teo:cobraid}. 
Another example is given by the reconstruction of Kashiwara's crystal basis, see Proposition~\ref{prop:kashiwara}. One can even ``import'' in $\OSL$ natural skein theoretical objects: in Subsection \ref{sub:joneswenzl} we define and provide some properties of the Jones-Wenzl idempotents in $\OSL$. 

\subsection{Lift of the Reshetikhin-Turaev invariant}
Suppose $T$ is a tangle diagram in the bigon whose boundary $\partial T$ is in $\partial \cB$ and the boundary points are labeled by signs $\pm$. The Reshetikhin-Turaev  operator invariant theory \cite{RT} assigns to $T$ a scalar $Z(T)\in \BQ(q^{1/2})$, see Section \ref{sec.WT}. On the other hand, such a labeled tangle $T$  defines an element in our skein algebra $\SB$. We have the following result which shows that our ``invariant'', which is $T$ considered as an element of $\cS(\cB)$, is a lift of the Reshetikhin-Turaev invariant.
\begin{thm} [Theorem \ref{teo:cat}] One has $\epsilon(T) = Z(T)$, where $\epsilon: \SB\to \BQ[q^{\pm 1/2}]$ is the counit.
\end{thm}  
It would be interesting to understand this lift of the Reshetikhin-Turaev invariant in terms of categorification.

\subsection{Skein algebras as comodule over $\OSL$. Hochshild cohomology} One important consequence of the identification of the bigon algebra with $\OSL$  is that for every boundary edge $e$ of a surface $\fS$, the skein algebra $\SS$ has  a  right $\OSL$-comodule structure
$$ \Delta_e : \SS \to \SS \ot \SB.$$
This map $\Delta_e$ is the splitting homomorphism \eqref{eq.rho0} applied to the an ideal arc parallel to $e$ which cuts off an ideal bigon from $\fS$ whose right edge is $e$, see Figure \ref{fig:copro}. 
Similarly identifying the left edge of $\cB$ to $e$  we get a left $\OSL$-comodule structure on $\SS$.

\FIGc{copro}{Geometric definition of the coaction: splitting the bigon along the dashed ideal arc}{3cm}

Using the comodule structure one can refine the splitting theorem by identifying the image of the splitting homomorphism, as follows. Let us cut $\fS$ along an ideal arc $c$ to get $\fS'$ as in Figure \ref{fig:cut}. Then $\cS(\fS')$ has a right $\OSL$-module structure coming from edge $a$ and a left $\OSL$-module structure coming from edge $b$. Thus $\cS(\fS')$ is a $\OSL$-bicomodule, and hence there is defined the Hochshild cohomology $HH^0(\cS(\fS'))$, for details see Section \ref{sec.comodule}.

\begin{thm}[Theorem \ref{teo:cotensor}] \label{thm.Hoch}
Under the splitting homomorphism the skein algebra $\SS$ embeds isomorphically into the Hochshild cohomology $HH^0(\cS(\fS'))$. In particular, when $c$ cuts $\fS$ into two surfaces $\fS_1$ and $\fS_2$, the splitting homomorphism  maps $\SS$ isomorphically onto the cotensor product of $\cS(\fS_1)$ and $\cS(\fS_2)$.   
\end{thm}

  \subsection{Skein algebra $\SS$ as module over $\USL$} 
Since the Hopf algebra $\USL$ is the Hopf dual of $\OSL$, then after tensoring with $\Qq$ each right $\OSL$-comodule is automatically a left $\USL$-module. Thus each boundary edge $e$ of $\fS$ gives $\SS$ a left $\USL$-module structure. Note that finite-dimensional $\USL$-modules are well-understood as they are quantum deformations of   modules over the Lie algebra $\mathfrak{sl}_2(\BC)$.

\begin{thm} [Part of Theorem \ref{teo:module}] Over the field $\BQ(q^{1/2})$, \red{for every boundary edge} the $\USL$-module $\SS$ is integrable, i.e. it is the direct sum of finite-dimensional irreducible $\USL$-modules.
\end{thm}
Actually Theorem \ref{teo:module} is much stronger: it provides an explicit decomposition and contains much more information about the decompositions as it deals also with the decomposition over Lusztig's integral version of $\USL$.

Using this result we also prove a dual version of Theorem \ref{thm.Hoch} which, with the notation of the theorem, shows that $HH_0(\Qq\otimes_{\cR}\cSs(\fS'))=\Qq\otimes_{\cR}\cSs(\fS)$ (see Theorem \ref{r.Hoch2}). 

\subsection{Braided tensor product} The co-$R$-matrix makes the category of $\OSL$-co\-mo\-du\-les a braided category  and  in general given two algebras in that category (which are then $\OSL$-co\-mo\-dule algebras)  their tensor product can be endowed with the structure of an algebra  by using appropriately the braiding: this is  the {\em braided tensor product} of the algebras, see \cite{MajidFoundation}. 
 \red{In Section \ref{sec:braidedtensor}, we generalise this notion to that of ``self-braided tensor product'' which applies to a comodule algebra having two commuting comodule structures.

Suppose $\underline{\fS}$ is obtained by identifying two edges of a (possibly disconnected) surface $\fS$ with two distinct edges of an ideal triangle as in Figure \ref{fig:brtens-new}. 
\FIGc{brtens-new}{Gluing $\fS$ to two distinct edges of an ideal triangle to get $\underline{\fS}$.} {3cm}
Then $\cS(\fS)$  has two natural commuting structures of $\OSL$-comodule algebra, and we have the following:
\begin{thm}[Theorem \ref{teo:braidedtensor}] As a $\OSL$-comodule algebra $\cSs(\underline{\fS})$ is canonically isomorphic to the self-braided tensor product of $\cS(\fS)$. In particular, if $\fS=\fS'\sqcup \fS''$ and the two edges belong to $\fS'$ and $\fS''$ respectively then $\cSs(\underline{\fS})$ is canonically isomorphic to the braided tensor product of $\cS(\fS_1)$ and  $\cS(\fS_2)$.
\end{thm}
}

Through this theorem we easily compute the skein algebra of all ``polygons'', ``punctured bigons'', and ``punctured monogons'' in Subsection \ref{sub:monogons}. It is remarkable that the skein algebras of the latter turn out to be bialgebra objects in the category of $\OSL$-comodules and that their structure morphism have natural topological interpretation. In particular the punctured monogon yields the ``transmutation'' of $\OSL$. 

\subsection{Modular operad} 

The splitting homomorphism allows to put the theory of stated skein algebras of surfaces in the framework  of operad theory. We define  the notion of geometric non-symmetric modular operad in Section \ref{sub:cooperad} and prove the following.

\begin{thm}[Precise statement given by  Theorem \ref{teo:NSmodoperad}] The stated skein algebra of surfaces gives rise to a non-symmetric modular operad in a category of bimodules over $\USL$.
\end{thm}

To be more specific while leaving the details for Section \ref{sub:cooperad}, let us recall that, according to Markl (\cite{Mar}) a ``Non-symmetric modular operad in a monoidal category $Cat$'' is a monoidal functor $NSO:\mathsf{MultiCyc}\to Cat$, where $\mathsf{MultiCyc}$ is a suitable category of ``multicyclic sets''.  
In Section  \ref{sub:cooperad} we re-cast Markl's definition, by defining a category $\NSGraph$ whose objects are punctured surfaces $\fS$ and whose morphisms are finite sets of ideal arcs (describing a way of cutting the surfaces). From this point of view, we then show in Theorem \ref{teo:NSmodoperad} that stated skein algebras provide a symmetric monoidal functor from this category into a suitable category of modules and bimodules over copies of $\USL$, thus providing a topological example of a NS modular operad.

\def\bSS{\overline {\cS}(\fS)}
\def\Ibad{{\mathcal I}^{\mathrm{bad}}}
\subsection{Reduced stated skein algebra, quantum torus, and quantum trace map} The stated skein algebra $\SS$ has a quotient $\bSS= \SS/\Ibad$, called the {\em reduced stated skein algebra}, whose algebraic structure is much simpler as it can be embedded into the so called quantum tori. Here $\Ibad$ is the ideal generated by elements, called bad arcs, described in Figure \ref{fig:badarc0} and is explained in Section \ref{sec:reduced}.

\FIGc{badarc0}{A bad arc.} {2cm}

We will show  that the ordinary skein algebra $\ooS(\fS)$ still embeds into $\bSS$ and hence we can use $\bSS$ to study $\ooS(\fS)$. Most importantly, the splitting theorem still holds for $\bSS$.
 
\begin{thm}[Theorem  \ref{teo:rsplit}] If $\fS'$ is the result of cutting $\fS$ along an ideal arc $c$, then the splitting homomorphism $\theta_c$ descends to an algebra embedding
 $$ \bar \theta_c: \bSS \embed \overline {\cS}(\fS').$$
\end{thm} 
The non-trivial fact here is that $\bar\theta _c$ is injective.
  
 Except for a few simple surfaces, we can always cut $\fS$ along ideal arcs so that the result is a collection of ideal triangles $T_1, \dots, T_k$. It follows that there is an embedding
 \be
 \Theta: \bSS \embed \bigotimes_{i=1}^k \overline {\cS} (T_i). \label{eq.rhoall} 
 \ee
 
The important thing with the reduced version is that for an ideal triangle $T$, unlike the full fledged $\cS(T)$, the reduced stated skein algebra $\overline {\cS}(T)$ is a {\em quantum torus} in three variables:

\begin{thm}[Theorem \ref{teo:redtri}] \label{thm.tri}
The reduced stated skein algebra $\overline {\cS}(T)$ of an ideal triangle has presentation
$$ \overline {\cS}(T) = \cR\la \al^{\pm 1}, \beta ^{\pm 1}, \gamma^{\pm 1}\ra  / (\beta\alpha = q\al\beta, \gamma\beta = q \beta\gamma, \al\gamma=q \gamma\al).$$
\end{thm}
Moreover, the reduced stated skein algebra of the bigon is naturally isomorphic to the algebra $\cR[x^{\pm 1}]$ of Laurent polynomial in one variable, see Proposition \ref{r.rbigon}.

Consequently, the map $\Theta$ of \eqref{eq.rhoall} embeds the reduced stated skein algebra $\bSS$ into a quantum torus in $3k$ variables.
Geometrically the  variables $\al,\beta,\gamma$ in Theorem \ref{thm.tri} come from the corner arcs  of the ideal triangle. There is a similar quantum torus $\bT'(T)$ in 3 variables corresponding to the edges of $T$, and a simple change of variables gives us an embedding $\overline {\cS}(T) \embed \bT'(T)$. Combining with $\Theta$ of \eqref{eq.rhoall} we get an algebra embedding 
$$
\tr_q: \bSS \overset {\Theta}\embed \bigotimes_{i=1}^k \bcS(T_i)  \embed \bigotimes_{i=1}^k \bT(T_i).
$$

There is a subalgebra $\cY$ of $\bigotimes_{i=1}^k \bT(T_i)$, known as the Chekhov-Fock algebra associated to the triangulation. 
\red{ The famous quantum trace map of Bonahon and Wong \cite{BW} is an algebra homomorphism $\widehat \tr_q: \widehat\cS(\fS) \to \cY$, where $\widehat \cS(\fS)$ is a coarser version of the stated skein algebra which surjects onto 
 $\cS(\fS)$, see Section \ref{sub:statedskeinalgebra}. 

\begin{thm} [See Theorem \ref{thm.qtrace}] The image of $\tr_q$ is in $\cY$. Thus $\tr_q$ restricts to an algebra embedding  $\tr_q: \bSS \embed \cY$, and   the quantum trace map of Bonahon and Wong is the composition $\widehat \cS(\fS)\onto  \bSS\overset {\tr_q} {\embed} \cY.$
\end{thm}
}

The existence of the quantum trace map (for $\ooS(\fS)$) was conjectured by Chekhov and Fock \cite{CF2}, and was established  by Bonahon and Wong \cite{BW}. It is called the quantum trace map since when $q=1$ it becomes a formula expressing the trace of a curve under the holonomy representation of the hyperbolic metric in terms of the shear coordinates of the Teichm\"uller space.
The second author \cite{Le:QT} gave another proof of the existence of the quantum trace map based on the Muller skein algebra, which is actually a subspace of the state skein algebra $\SS$. The above approach using the reduced stated  skein algebra  offers another proof, and also determines the kernel of the original quantum trace map $\widehat \tr_q$.

\subsection{Classical limit}
The last section explores the natural question of ``what is the classical limit of $\SS$?'' In the case of the standard skein algebra $\ooS(\fS)$ it is known \cite{Bullock,PS1} (see also \cite{CM}) that when the quantum parameter $q$ is $-1$ and the ground ring is $\BC$ then $\ooS(\fS)$ is isomorphic as an algebra to the coordinate ring the $\mathrm{SL}_2(\mathbb{C})$-character variety of $\fS$ and that in general the algebras at $q$ and $-q$ are isomorphic via the choice of a spin structure on $\fS$ (\cite{Barrett}). Note that our stated skein algebra is not commutative when $q=-1$ though it is commutative when $q=1$. 

We introduce the variety $tw(\fS)$ of  ``twisted  $\mathrm{SL}_2(\mathbb{C})$-bundles'' over $\fS$, which, roughly speaking, are flat $\mathrm{SL}_2(\mathbb{C})$-bundles over the unit tangent bundle $U\fS$  of $\fS$ with holonomy $-Id$ around the fibers of $\pi:U\fS\to \fS$ and \red{are equipped} with trivializations along the edges of $\fS$, but which we reformulate in terms of groupoid representations. To deal with the non-oriented nature of the  arcs of the stated skein algebra we have to use a trick smoothing the arcs at their end-points so that one can compose arcs. 
\begin{thm}[Theorem \ref{teo:classicaliso}] \label{thm.classic}
 When $q=1$ and the ground ring is $\BC$ the stated skein algebra $\SS$ is naturally isomorphic to the coordinate ring of $tw(\fS)$.
\end{thm}

 In classical terms, the splitting theorem becomes an instance of a van-Kampen like theorem for groupoid representations.

Theorem \ref{teo:classicaliso} highlights a relation between $\SS$ and the coordinate ring of the character variety of $\fS$. The study of quantizations of such rings has been performed with different techniques (based on Hopf algebras and lattice gauge theory) by  Alekseev-Grosse-Schomerus (\cite{AGS}), Buffenoir-Roche \cite{BR}, Fock-Rosly (\cite{FR}) and, later, via skein theoretical approaches by Bullock-Frohman-Kania-Bartoszynska (\cite{BFK}). 
The relation of our work with these previous ones is still to be clarified, although it seems that one of the main differences between our approach and some of the above cited ones, is that we allow for ``observables with boundary'' and, as explained in the preceding paragraph, this endows the algebras we work with rich algebraic structures which in particular make computations much easier.

\subsection{Related results}

While the authors were completing the present work,  D. Ben-Zvi, A. Brochier and D. Jordan \cite{BBJ} constructed a theory of quantum character variety for general Hopf algebras, based on completely  different techniques. Part of the results of this paper could probably be recasted in that theory, though we don't know the precise relation between the two theories. The substantial difference of the techniques used makes these works complementary. K. Habiro informed us that his  ``quantum fundamental group theory'' also 
gives an alternative approach to the theory of quantum character variety.

When the authors presented their works at conferences, Korinman informed us that he in joint work 
A. Quesney obtained results similar to Theorem \ref{thm.Hoch} and Theorem \ref{thm.classic}, see their recent preprint \cite{KQ}.

{\bf Acknowledgements.} The authors would like to thank F. Bonahon, R. Kashaev, A. Sikora,  V. Turaev, and D. Thurston for helpful discussions.  
Both authors were partially supported by the CIMI (Centre International de Math\'ematiques et d'Informatique) Excellence program during the Thematic Semester on Invariants in low-dimensional geometry and topology held in spring 2017. The second author is partially support by an NSF grant.

The authors presented the results of this paper in the form of talks or mini-courses at many conferences, including  ``Thematic School on Quantum Topology and Geometry'' (University of Toulouse, May 2017), ``Algebraic Structures in Topology and Geometry'' (Riederalp, Switzerland, January 2018), ``TQFTs and categorification'' (Cargese, March 2018), ``Volume Conjecture in Tokyo'' (University of Tokyo, August 2018), ``Categorification and beyond'' (Vienna, January 2019),  Topology mini-courses (university of Geneva, May 2019), ``New developments in quantum topology'' (Berkeley June 2019), ``Expansions, Lie algebras and invariants'' (Montreal July 2019) and would like to thank the organizers for the opportunities to present their work.

\section{Stated skein algebras} \label{sec.def}
We will present the basics of the theory of stated skein algebras: definitions, bases of skein algebras, the splitting homomorphism,  filtrations and gradings. New results involve Proposition \ref{r.inversion} describing the inversion homomorphism and Proposition \ref{r.grbinom} giving the exact value of the splitting homomorphism in the associated graded algebra.
\subsection{Notations}
Throughout the paper let
$\BZ$ be the set of integers, $\BN$ be the set of non-negative integers, $\BC$ be  the set of complex numbers.
The ground ring $\cR$  is a commutative ring with unit 1,  containing a distinguished invertible element $q^{1/2}$.
 For a finite set $X$ we denote by $|X|$ the number of elements of $X$.
 
 \def\bue{\overset \bullet =}
 We will write $x \bue y$ if there is $k\in \BZ$ such that $x = q^{k} y$.

\def\fT{\mathfrak T}
\subsection{Punctured bordered surface}\label{sub:puncturedsurfaces} By a {\em punctured bordered surface} $\fS$ we mean a  surface of the form $\fS = \bfS \setminus \cP$, where $\bfS$ is 
 a compact oriented surface  with
(possibly empty) boundary $\pbfS$, and $\cP$ is a finite non empty set such that every connected component of the boundary $\pbfS$ has at least one point in $\cP$.  We don't require $\fS $ be to connected. It is easy to see that $\bfS$ is uniquely determined by $\fS$. Throughout this section we fix a  punctured bordered surface $\fS$.

An {\em ideal arc} on $\fS$ is an immersion $a: [0,1] \to \bfS$ such that $a(0),a(1)\in \cP$ and the restriction of $a$ onto $(0,1)$ is an embedding into $\fS$. 
Isotopies of ideal arcs are considered in the class of ideal arcs.

A connected component of $\pfS$ is called a {\em boundary edge} of $\fS$ (or simply an edge), which is an ideal arc.

\begin{remark} The fact that each connected component of $\partial \fS$ is an open interval is not a serious restriction as for the purpose of the constructions of this paper a point-less boundary component is treated as a puncture; so that in the end the only excluded surfaces are closed ones without punctures.
\end{remark}

\subsection{Ordinary skein algebra}  Let $M=\fS \times (0,1)$. For a point $(z,t)\in \fS \times (0,1)$ its {\em height} is $t$. A vector at
 $(z,t)$ is called {\em vertical} if it is a positive vector of $z \times (0,1)$. {\em A framing} of a 
 1-dimensional submanifold $\al$ of $M=\fS \times (0,1)$ is 
 a continuous choice of a vector transverse to $\al$ at each point of $\al$.

 A {\em framed link} in $\fS\times (0,1)$ is a closed 1-dimensional unoriented submanifold $\al$  equipped with a framing.  The empty set, by convention, is considered a framed link. 
 
 A link diagram on $\fS$ determines an isotopy class of framed links in $\fS \times (0,1)$, where the framing is vertical everywhere. Every isotopy class of framed links in $\fS \times (0,1)$ is presented by a  link diagram.

The {\em skein module} $\ooS(\fS)$, first introduced  by Przytycki \cite{Prz} and Turaev \cite{Turaev0}, is defined to be the $\cR$-module generated by the isotopy classes of framed unoriented links in $\fS \times (0,1)$ modulo the Kauffman relations
\begin{align}
\label{eq.skein01} \cross \ &= \ q\resoP + q^{-1} \resoN\\
\label{eq.loop01}  \trivloop\  &=\  (-q^2 -q^{-2})\emptyr
 \end{align}
We use the following convention about pictures in these identities, as well as in other identities in this paper. Each  shaded part is a  part of $\fS$, with a link diagram on it.
Relation \eqref{eq.skein01} says that if  link diagrams $D_1, D_2$, and $ D_3$ are identical everywhere except for a small disk in which $D_1,D_2, D_3$ are like in respectively the first, the second, and the third shaded areas, then $[D_1] = q[D_2] + q^{-1} [D_3]$ in the skein module $\ooS(\fS)$. Here $[D_i]$ is the isotopy class of links determined by $D_i$. The other relation is interpreted similarly.

For two framed links $\al_1$ and $\al_2$ the product $\al_1 \al_2$  is defined as the result of stacking $\al_1$ above $\al_2$. That is, first isotope $\al_1$ and $\al_2$ so that $\al_1 \subset \fS \times (1/2,1)$ and $\al_2 \subset \fS \times (0, 1/2)$. Then $\al_1 \al_2= \al_1 \cup \al_2$.
 It is easy to see that this gives rise to a well defined product and hence an $\cR$-algebra structure on $\ooS(\fS)$.

\subsection{Tangles and order}\label{sub:tangle} In order to include the boundary of $\fS$ into the picture, we will replace framed links by more general objects called $\pM$-tangles. Recall that $M= \fS \times (0,1)$ and its boundary is $\pM = \pfS \times (0,1)$.

In this paper, a {\em $\pM$-tangle} is
an unoriented,
 framed, compact,  properly embedded 1-dimensional submanifold $\al \subset M=\fS \times (0,1)$ such that:
 \begin{itemize}
 \item  at every point of $\partial \al=\al \cap \pM$ the framing is {\em vertical}, and
\item  for any boundary edge $b$, the points of $\partial_b(\al):=\partial \al \cap (b \times (0,1))$  have distinct~heights.
 \end{itemize}

For a $\pM$-tangle $\al$ define a partial order on $\partial \al $ by: $x>y$ if there is a boundary edge $b$ such that $x, y \in b \times (0,1)$ and
 $x$ has greater height. If $x>y$ and there is no $z$ such that $x>z>y$, then we say
$x$ and $y$ are {\em consecutive}. 

{\em Isotopies of  $\pM$-tangles} are considered  in the class of $\pM$-tangles. \red{It should be noted that isotopies of $\pM$-tangles preserve the height order.}
The empty set, by convention, is a $\pfS$-tangle which is isotopic only to itself.

As usual, $\pM$-tangles are depicted by their diagrams on $\fS$, as follows.
Every $\pfS$-tangle is isotopic to one with vertical framing.
Suppose a vertically framed $\pM$-tangle $\al$ is in general position with respect to
the standard projection $\pi: \fS \times (0,1) \to \fS$, i.e.  the restriction $\pi |_{\al}:\al \to \fS$ is an immersion with transverse  double points as the only possible singularities and there are no double points on the boundary of $\fS$. 
Then $D=\pi (\al)$, together with 
\begin{itemize}
\item  the over/underpassing information at every double point, and 
\item  the linear order on $\pi(\al) \cap b$ for each boundary edge $b$ induced from the height order
\end{itemize} 
is called a {\em  $\pM$-tangle  diagram}, or simply a tangle diagram on $\fS$.
{\em Isotopies} of  $\pM$-tangle diagrams are  ambient isotopies in $\fS$.

Clearly the $\pM$-tangle diagram of a $\pM$-tangle $\al$ determines the isotopy class of $\al$.
 When there is no confusion, we identify a $\pM$-tangle diagram with its isotopy class of $\pM$-tangles.

\def\ori{{\mathfrak o}}
 Let $\ori$ be an {\em orientation} of $\pfS$,   which on a boundary edge may or may not be equal to the orientation inherited from $\fS$. A $\pM$-tangle diagram $D$ is {\em $\ori$-ordered} if for each boundary edge $b$ the order of $\partial D$ on $b$ is  increasing when one goes along $b$ in  the direction of $\ori$. It is clear that every $\pM$-tangle, after an isotopy, can be presented by an  $\ori$-ordered $\pM$-tangle diagram.
 If $\ori$ is  the orientation coming from $\fS$, the $\ori$-order is called the {\em positive order}.

\def\ori{{\mathfrak o}}
\def\ofS{\mathring {\fS}}

\subsection{Stated skein algebra}\label{sub:statedskeinalgebra}

A {\em state} on a finite set $X$ is a map $s: X \to \{\pm\}$. We write $\# s= |X|$.
A {\em stated $\pM$-tangle }  (respectively {\em a stated $\pM$-tangle diagram}) is a $\pM$-tangle (respectively {\em a  $\pM$-tangle diagram})  equipped with a {\em state} on its set of boundary points.

The {\em  stated skein algebra}  $\cSs(\fS)$ is  the $\cR$-module freely spanned
 by isotopy classes of stated $\pM$-tangles modulo the {\em defining relations}, which are the old  skein relation~\eqref{eq.skein} and the trivial loop relation~\eqref{eq.loop}, and the new boundary relations~\eqref{eq.arcs} and~\eqref{eq.order}:
\begin{align}
\label{eq.skein} \cross \ &= \ q\resoP + q^{-1} \resoN\\
\label{eq.loop}  \trivloop\  &=\  (-q^2 -q^{-2})\emptyr\\
\label{eq.arcs} \leftup\  & =\  q^{-1/2} \emptys\ , \qquad \leftupPP\ =0, \quad \  \leftupNN \ = 0   \\
 \label{eq.order}   \reordone\ &=\  q^2 \reordtwo \ +\  q^{-1/2} \reordthree
 \end{align}
Here each   shaded part is a  part of $\fS$, with a stated $\pM$-tangle diagram on it.
Each arrowed line is part of a boundary edge, and the order on that part is indicated by the arrow and the points on that part are consecutive in the height order. The order of other end points away from the picture can be arbitrary and are not determined by the arrows of the picture. On the right hand side of the first identity of \eqref{eq.arcs}, the arrow does not play any role; it is there only because the left hand side has an arrow.

Again for two $\pM$-tangles $\al_1$ and $\al_2$ the product $\al_1 \al_2$  is defined as the result of stacking $\al_1$ above $\al_2$. The product makes $\cSs(\fS)$ an $\cR$-algebra. In \cite{Le:TDEC} it is proved that if $\cR$ is a domain then $\SS$  does not have non-trivial zero-divisors, \red{a fact known earlier for the case when $\fS$ has no boundary \cite{PS2}.}

If   $\fS_1$ and $\fS_2$ are two punctured bordered  surfaces, then
there is a natural isomorphism
$$
\cSs (\fS_1 \sqcup \fS_2) \cong \cSs (\fS_1) \ot_\cR
\cSs (\fS_2).
$$

Since the interior $\ofS$ of $\fS$ does not have boundary, we have $
\cS(\ofS) = \ooS(\fS).$

\def\SMuller{\cS^{\mathrm{+}}(\fS)}
The subalgebra $\SMuller$ spanned by $\pM$-tangles whose  states  are all $+$ is naturally isomorphic to the skein algebra defined by Muller \cite{Muller}, see \cite{Le:TDEC,LP}. \red{If in the definition of $\cS(\fS)$ we use only two relations \eqref{eq.skein} and \eqref{eq.loop}, we get a coarser version $\widehat{\cS}(\fS)$ which was defined by Bonahon and Wong \cite{BW}.}

\begin{remark} \label{rem.hS}  Relations \eqref{eq.arcs} already appeared in \cite{BW}. Relation \eqref{eq.order} appeared  in \cite{Le:TDEC} where the stated skein algebra was introduced.
\end{remark}

\def\starto{\overset \star \longrightarrow}
\def\leftvep{\raisebox{-8pt}{\incl{.8 cm}{leftvep}}}

\subsection{Consequences of defining relations}

Define $C^\nu_{\nu'}$ for $\nu, \nu'\in \{\pm \}$ by
\be
\label{eq.Cve}
C^+_+ = C^-_- = 0, \quad C^+_-= q^{-1/2}, \quad C^-_+ = -q^{-5/2}.
\ee
In the next lemma we have the values of all the trivial arcs.
\begin{lemma}[Lemma 2.3 of \cite{Le:TDEC}] \label{r.arcs}
In $\cSs(\fS)$ one has
\begin{align}
\label{eq.kink} -q^{-3} \kinkp \ &= \ \kinkzero\  = \ -q^3\kinkn\\
\label{eq.arcs1} \leftve\   &=\  C^\nu _{\nu'}  \emptys\\
\label{eq.arcs2} \leftvedl \ & = \ \leftved\   =\  - q^{3} C^{\nu'} _{\nu}  \emptys
\end{align}
\end{lemma}

The next lemma describes how a skein behaves when the order of two consecutive boundary points is switched.

\begin{lemma}[Height exchange move, Lemma 2.4 of \cite{Le:TDEC}]\label{r.refl}
 (a) One has
\begin{align}
\label{eq.reor1}
\reordonez \  = \ q^{-1} \,  \reordonea , \quad \reordsixz \  &= \ q^{-1} \, \reordsixa , \quad \reordonepn = q \, \reordtwopn \\
\label{eq.reor2}
q^{\frac 32} \reordnp  - q^{-\frac 32}  \reordpn  &= (q^2 -q^{-2}) \reordthree.
\end{align}

(b) Consequently, if $q=1$ or $q=-1$, then for all $\nu, \nu'\in \{\pm \}$,
\be
\label{eq.sign}
 \reordoneall =
 q\,  \reordoneallp.
 \ee

\end{lemma}

\begin{remark} Because of relation  \eqref{eq.sign}, in general  $\cSs(\fS)$ is not commutative when $q=-1$. This should be contrasted with the case of the usual skein algebra   $\ooS(\fS)$, which is
  commutative and is canonically equal to the $\mathrm{SL}_2(\BC)$ character variety of $\pi_1(\fS)$ if $\cR=\BC$ and $q=-1$ (assuming $\fS$ is connected), see \cite{Bullock,PS1}.
\end{remark}

\def\id{\mathrm{id}}
\def\ofS{\mathring{\fS}}
\def\tsigma{\chi}

 \def\cL{\mathcal L}
 \def\tchi{\tilde \chi}

\subsection{Reflection anti-involution}  \label{sec.refl} 
\begin{proposition}[Reflection anti-involution, Proposition 2.7 in \cite{Le:TDEC}]   \label{r.reflection} \red{Suppose $\cR=\BZ[q^{\pm 1/2}]$. There exists a unique $\BZ$-linear map $\chi: \cSs(\fS) \to \cSs(\fS)$, such that

\begin{itemize}
\item  $\chi(q^{1/2})= q^{-1/2}$, 
\item $\chi$ is an anti-automorphism, i.e. for any $x,y \in \cSs(\fS)$,
$$\tsigma(x+y)= \tsigma(x) + \tsigma(y), \quad \tsigma(xy) = \tsigma(y) \tsigma(x),$$
\item if $\al$ is a stated $\pM$-tangle diagram then $\chi(\al)$ is the result of switching all the crossings of $\al$ and reversing the linear order on each boundary edge.
\end{itemize} 
}
 \end{proposition}
Clearly $\tsigma^2=\id$. We call $\tsigma$ the reflection anti-involution.

\def\inv{\mathrm{inv}}
\def\id{\mathrm{id}}
\def\er{e_r}
\def\el{e_l}
\def\binv{\overline{\inv}}

\subsection{Inversion along an edge}  \label{sec.invers}

\red{
\begin{proposition}
\label{r.inversion}  Let  $e$ be a boundary edge of a punctured bordered surface $\fS$ and  $f:\{\pm \} \to \cR$ be a function such that $f(+) f(-)= - q^{-3}$.

There exists a unique $\cR$-linear homomorphism $\inv_{e,f}: \Ss(\fS) \to \Ss(\fS)$ such that if $\al$ is a stated $\partial \fS$-tangle diagram with a state $s$ and with positive order on $e$,
then 
\be 
\inv_{e,f}(\al)=  \left( \prod_{x \in (\al \cap e)} f(s(x))  \right)  \al',
\label{eq.def10}
\ee where $\al'$ is the same $\al$ except that the height order of $\al'$ on $e$ is given by the negative direction of $e$ and the state of $\al'$ on $e$ is obtained from that of $\al$ by switching $\nu\in \{\pm \}$ to $-\nu$ at every boundary points in $\al \cap e$.

If $e'$ is another boundary edge and $f'(+) f'(-)= -q^{-3}$, then 
\be 
\inv_{e,f}  \circ \inv_{e',f' } = \inv_{e'.f'} \circ \inv_{e,f}.
\label{eq.com5}
\ee 
\end{proposition}

\def\tinv{\widetilde{\inv}}
\def\tSs{\widetilde{\Ss}}
\begin{proof}  Let $T$ be the set of isotopy classes of stated, positively ordered $\pfS$-tangle diagrams. Since $T$ spans $\Ss(\fS)$, the  uniqueness of $\inv_{e,f}$ is clear. 

Let $\tSs$ be the $\cR$-module freely span by $T$, and $\tinv_{e,f}: \Ss'\to \Ss(\fS)$ be the 
 $\cR$-linear map defined  by \eqref{eq.def10}. To show that  $\tinv_{e,f}$ descends to a map $\inv_{e,f}: \Ss(\fS)\to \Ss(\fS)$ one needs to prove $\tinv_{e,f}$ is invariant under isotopy in $M:=\fS \times (0,1)$ and under the moves generated by the defining relations \eqref{eq.skein}-\eqref{eq.order}. More precisely, we have to show that  $\tinv_{e,f}(\al) = \tinv_{e,f}(\al)$ for any $\al,\al'\in T$   whenever 
\begin{itemize}
\item[(i)] $\al$ and $\al'$ are isotopic as $\pM$-tangles, or
\item[(ii)] $\al$ and $\al'$ are respectively the left hand side and the right hand side of the defining relations \eqref{eq.skein}-\eqref{eq.order}.
\end{itemize}

It is known that $\al$ and $\al'$ are isotopic as $\pM$-tangles if and only if they are related by a sequence of the 3 framed Reidemeister moves of \cite[Section 1.2]{Oh}. 
The invariance under the 3 framed Reidemeister moves follows from the invariance under the defining relations \eqref{eq.skein} and \eqref{eq.loop}, see \cite{Kauffman}. Clearly $\tinv_{e,f}$ is invariant under the moves generated by the  defining relations \eqref{eq.skein} and \eqref{eq.loop}. There remains relations \eqref{eq.arcs} and \eqref{eq.order}.

Let us consider \eqref{eq.arcs}.  Using the definition, $f(+)f(-)=q^{-3}$, and then \eqref{eq.arcs2}, we have:

$$\inv_{e,f}\left(\leftup\right)=f(+)f(-)\leftmp =-q^{-3}\leftmp =-q^{-3}(-q^3)C_-^+ = q^{-1/2},$$
which proves the first identity of \eqref{eq.arcs}. The other two identities of \eqref{eq.arcs} are trivial.

Let us consider \eqref{eq.order}. By definition and Lemma \ref{r.refl},
\be \inv_{e,f}\left(\reordone\right)=f(+)f(-)\reordonepn=(-q^{-3})  ( q \, \reordtwopn)= -q^{-2} \reordtwopn.
\label{eq.u2}
\ee
\begin{align}\inv_{e,f}\left(q^2\reordtwo+q^{\frac{-1}{2}}\reordthree\right) & =q^2(-q^{-3})\reordnp+q^{-\frac{1}{2}}\reordthree \notag\\
&=-q^{-1} \left( q^{-3} \reordpn+ q^{-\frac{3}{2}}(q^2-q^{-2})\reordthree \right)+q^{-\frac{1}{2}}\reordthree  \notag
\\
&=-q^{-4}\reordpn+q^{-\frac{9}{2}}\reordthree, \label{eq.u3}
\end{align}
and the right hand sides of \eqref{eq.u2} and \eqref{eq.u3} are equal due to \eqref{eq.order}.

Identity \eqref{eq.com5} follows immediately from the definitions.
\end{proof}

There are two important cases for us. Define
\be \inv_e:= \inv_{e,C}, \quad \text{and} \ \binv_e:= \inv_{e,\bar C},
\ee
where \be 
C(+)= \bar C (-) = -q^{-5/2}, \quad C(-) = \bar C(+) = q^{-1/2}.
\ee

Note that $C(\nu) = C^{-\nu}_{ \nu}$ and $\bar C(\nu) = C(-\nu)$ for  $\nu\in \{\pm \}$. 
For a stated tangle diagram $\al$ with a state $s$ on the boundary edge $e$ define 
\be 
 C_e(\al) = \prod_{x\in \al \cap  e} C(s(x))= \prod_{x\in \al \cap  e} C^{-s(x)}_{s(x)}.
 \label{eq.Ce}
 \ee 
If $\al$ has positive order on $e$ then, with $\al'$  defined as in Proposition \ref{r.inversion}, one has
\be\label{eq:invC} 
\inv_e(\al) = C_e(\al) \al',
\ee

\begin{remark} The definition \eqref{eq.def10} works only for stated $\pfS$-tangle diagrams with positive order on $e$. If the order is not positive, the formula will be different.

In general $\inv_{e,f}$ is not an algebra homomorphism.
\end{remark}

}
\subsection{Basis of stated skein module} \label{sec.basis0}

{ A $\pM$-tangle diagram $D$} is {\em simple} if it has neither double point nor
trivial component. Here a closed component of $D$ is {\em trivial} if it bounds a disk in $\fS$,
 and an arc component of $\al$ is {\em trivial} if it can be homotoped relative to its boundary  to a subset of a boundary edge.
 By convention, the empty set is considered as a simple stated $\pM$-tangle diagram with 0 components.

Define an order on $\{ \pm\}$ so that the sign $-$ is less than the sign $+$. If $X$ is a partially ordered set, then a state $s: X \to \{\pm\}$ is {\em increasing} if $s$ is an increasing function, i.e. $f(x)  \le f(y)$ whenever $x \le y$.

\def\inv{\mathrm{nd}}

\def\tB{\tilde B}
 Choose an orientation  $\ori$  of
$\pfS$.
Let  $B(\fS;\ori)$ be the set of
 of all isotopy classes of increasingly stated, $\ori$-ordered simple $\pM$-tangle diagrams. From the defining relations  it is easy to show that the set $B(\fS;\ori)$ spans $\SS$ over $\cR$.

\begin{theorem} [Theorem 2.8 in \cite{Le:TDEC} ]   \label{thm.basis1a}
 Suppose $\fS$ is a punctured bordered surface and $\ori$ is an orientation of $\pfS$.
Then $B(\fS;\ori)$ is an $\cR$-basis of $\cSs(\fS)$.
  \end{theorem}
  \def\dec{\mathrm{dec}}
  
 \begin{remark} Theorem \ref{thm.basis1a} means that the coefficients given in the defining relations \eqref{eq.arcs} and \eqref{eq.order} are consistent in the sense that they do not lead to any more relations among the set $B(\fS;\ori)$.
 \label{rem.cons}
 \end{remark}

\def\oB{\mathring B}

The subset $\oB(\fS;\ori)\subset B(\fS;\ori)$ consisting of $\al\in B(\fS;\ori)$ having no arcs is a basis of the ordinary skein algebra $\ooS(\fS)$. Similarly, the subset $B^{+}(\fS;\ori)\subset B(\fS;\ori)$ consisting of $\al\in B(\fS;\ori)$ having only positive states is a basis of the Muller  skein algebra $\SMuller$, see \cite{Muller,Le:TDEC,LP}. Hence we have the following.
\begin{corollary} Both the ordinary skein algebra $\ooS(\fS)$ and the Muller skein algebra $\SMuller$ are subalgebras of the stated skein algebra $\SS$.  
\end{corollary}

\def\fA{{\mathfrak{A}}}
\subsection{Filtration and grading} \label{sec.filtration} Suppose $a$ is either an ideal arc or a simple closed curve on $\fS$ and $\al$ is a simple $\pM$-tangle diagram on $\fS$. The geometric intersection index $I(a,\al)$ is
$$ I(a,\al) = \min |a \cap \al'|,$$
where the minimum is over all the simple $\pM$-tangle diagrams $\al'$ isotopic to $\al$.

For a collection $\fA=\{ a_1, \dots, a_k\}$, where each $a_i$ is either an ideal arc or a simple closed curve, and $n\in \BN$ let $F^\fA_n(\SS)$ be the $\cR$-submodule of $\SS$ spanned by all stated simple $\pM$-diagrams $\al$ such that $\sum_{i=1}^k I(a_i, \al) \le n$. It is easy to see that the collection $\{F^\fA_n(\SS)\}_{n\in \BN}$ forms a filtration of $\SS$ compatible with the algebra structure, i.e. with $F_n= F^\fA_{n}(\SS)$ one has
$$ F_{n} \subset F_{n+1}, \ \bigcup_{n\in \BN} F_n = \SS, \ F_n F_{n'} \subset F_{n+n'}.$$
One can define the associated graded algebra 
$\Gr^{\fA}(\cSs(\fS))$:
 $$ \Gr^{\fA}(\cSs(\fS))= \bigoplus_{n=0}^\infty  \Gr^{\fA}_n(\cSs(\fS))\   \text{with } \Gr^{\fA}_n(\cSs(\fS))= F_n/F_{n-1}\ \forall n\ge 1 \ \text{and}\  \Gr_0= F_0.$$ 

This type of filtration has been used extensively in the theory of the ordinary skein algebra, see e.g. \cite{Le:QT,FKL,LP,Marche}. 

 The following is a consequence of Theorem \ref{thm.basis1a}:
\begin{proposition}[Proposition 2.12 in \cite{Le:TDEC}]\label{r.basis2}  Let  $\ori$ be an orientation of the boundary of a punctured bordered surface $\fS$, and $\fA=\{a_1,\dots, a_k\}$ be a collection of boundary edges of $\fS$.\\
 (a) The set $\{ \al \in B(\fS;\ori)\mid \sum_{i=1}^k I(\al, a_i)\le n\} $
is an $\cR$-basis of $F^\fA_n(\cSs(\fS))$.\\
(b) The set $\{ \al \in B(\fS;\ori) \mid \sum_{i=1}^k I(\al, a_i)= n\} $ is an $\cR$-basis of $\Gr^\fA_n(\cSs(\fS))$.\\
\end{proposition}

\def\Ed{{\cE_\partial}}
\def\ld{\mathrm{lt}}
For what concerns the grading, for each non-negative integer $m$ and a boundary edge $e$ let  $G^e_m$ be the $\cR$-subspace of $\cSs(\fS)$ spanned by stated $\pM$-tangle diagrams  $\al$ with 
 $\delta_e(\al):= \sum_{u \in (\al \cap e)} s(u) =m,$ where $s$ is the state and 
  we identify $+$ with $+1$ and $-$ with $-1$.  
 
From the defining relations it is clear that $\cSs(\fS)= \bigoplus_{m  \in \BZ} G^e_m$ and $G^e_m G^e_{m'} \subset G^e_{m+m'}$. In other words,  $\cSs(\fS)$ is a graded algebra with the grading $\{ G^e_m\}_{m\in \BZ}$.

Also the following is a consequence of Theorem \ref{thm.basis1a}:
\begin{proposition}\label{r.basis3}
Let $\fS$ be a punctured bordered surface and $\ori$ be an orientation of $\pfS$.
 The set $\{ \al \in B(\fS;\ori) \mid \delta_e(\al)=m\} $ is an $\cR$-basis of $G^e_{m}(\cSs(\fS))$.

\end{proposition}

If $\ori'$ is another orientation of the boundary $\pfS$, the change from basis $B(\fS;\ori)$ to $B(\fS;\ori')$ might be complicated. For the associated space $\Gr^\fA(\SS)$, the change of bases is simpler. 

Recall that $\al \bue \al'$ means $\al = q^m \al'$ for some $m\in \BZ$.

\begin{proposition} \label{r.orichange} Suppose $\al$ is stated tangle diagram on $\fS$ and $I(\al,e)=k$ where $e$ is a boundary edge. Let's alter $\al$ to get $\al'$ by changing the height order on $e$ and the states on $e$ such that $\delta_e(\al) = \delta_e(\al')$. Then  one has
 \be 
\al \bue \al'   \quad \text{ in $\Gr^e_k(\SS)$}\label{eq.orichange}. \ee
\end{proposition}
\begin{proof} 

One can get $\al'$ from $\al$ by a sequence of moves, each is either (i) an exchange of the heights of two consecutive vertices on $e$, or (ii) an exchange of states of two consecutive vertices on $b$. We can assume that $\al'$ is the result of doing a move of type (i) or type (ii).

In case of move (i), the identities \eqref{eq.reor1} and \eqref{eq.reor2} prove \eqref{eq.orichange}.

In case of move (ii), the identity \eqref{eq.order} prove \eqref{eq.orichange}.
\end{proof}

\def\tpr{\widetilde{\pr}}
\def\tal{\tilde \al}
\subsection{Splitting/Gluing punctured bordered surfaces}    \label{sec.31}
Suppose $a$ and $b$ are distinct boundary edges of a punctured bordered surface $\fS'$ which may not be connected. Let $\fS= \fS'/(a=b)$ be the result of gluing $a$ and $b$ together in such a way that the  orientation is compatible. The canonical projection $\pr: \fS' \to \fS$  induces a projection $\tpr: M=\fS' \times (0,1) \to M=\fS\times (0,1)$.
Let $c= \pr(a)=\pr(b)$. It is an interior ideal arc  of $\fS$.

Conversely if $c$ is an ideal arc in the interior of $\fS$, then there exists $\fS', a,b$ as above such that $\fS= \fS'/(a=b)$, with $c$ being the common image of $a$ and $b$. We say that $\fS'$ is the result of splitting $\fS$ along $c$.

 A $\pM$-tangle  $\al\subset M=\fS \times (0,1)$,  is said to be {\em vertically transverse to $c$} if
 \begin{itemize}
 \item $\al$ is transverse  to $c \times (0,1)$,
 \item the points in $\partial_c\, \al:= \al \cap (c \times (0,1))$ have distinct heights, and have vertical framing.
 \end{itemize}
 Suppose $\al$ is a stated $\pM$-tangle vertically transverse  to $c$. Then $\tal:=\tpr^{-1}(\al)$ is a $\pM'$-tangle
 which is stated at every boundary point except for newly created boundary points, which are points in  $\tpr^{-1} (\partial_c\, \al)$.
  {\em A lift} of $\al$ is a stated $\pM'$-tangle $\beta$ which is $\tal$ equipped with states on $\tpr^{-1} (\partial_c\, \al)$ such that if $x,y\in \tpr^{-1} (\partial_c\, \al)$ with $\tpr(x)=\tpr(y)$ then $x$ and $y$ have the same state. If $|\partial_c\, \al|=k$, then  $\al$ has $2^k$ lifts.

\begin{theorem}[Splitting Theorem, Theorem 3.1 in \cite{Le:TDEC}] \label{thm.1a} Suppose $c$ is an ideal arc in the interior of a  punctured bordered surface $\fS$ and $\fS'$ is the result of splitting $\fS$ along $c$.

(a)
There is a unique $\cR$-algebra homomorphism
$ \theta_c: \cSs(\fS) \to \cSs(\fS')$, called the splitting homomorphism along $c$,
such that if $\al$ is a stated $\pM$-tangle  vertically transverse  to $c$, then
\be 
\theta _c(\al)=\sum \beta, \label{eq.slit}
\ee where the sum is over all lifts $\beta$ of $\al$.

(b) In addition, $\theta_c $ is injective.

(c) If $c_1$ and $c_2$ are two non-intersecting ideal arcs in the interior of $\fS$, then
$$ \theta_{c_1} \circ \theta_{c_2} = \theta_{c_2} \circ \theta_{c_1}.$$

\end{theorem}
\begin{remark} The coefficients of the right hand sides of the defining relations \eqref{eq.arcs} and \eqref{eq.order} were chosen so that one has the consistency (see Remark \ref{rem.cons}) and the splitting theorem. It can be shown that if one requires the consistency and the splitting theorem, then the coefficients are unique, up symmetries of a group isomorphic to $\BZ/2 \times \BZ/2$.
\end{remark}

\subsection{Splitting homomorphism and filtration} Fix an orientation $\ori$ of the boundary edges of $\partial \fS$. Let  $\fS'$ be the result of splitting $\fS$ along an ideal arc $c$, with $c$ being lifted to boundary edges $a$ and $b$ of $\fS'$. 
Choose an orientation of $c$ and lift this orientation to $a$ and $b$ which, together with $\ori$, gives an orientation $\ori'$ for $\fS'$.
Assume $D$ is a stated simple $\ori$-ordered $\pM$-tangle diagram which is {\em taut with respect to $c$}, i.e. $|D\cap c|=I(D,c)$. 
For each each $i=0,1,\dots, m:=|D \cap c|$ let $(\tD,s_i)$ be the $\pM'$-tangle diagram where $\tD= \pr^{-1}(D)$,  and the states on both $a$ and $b$ are $\ori'$-increasing and having exactly $i$  minus signs. Then each $(\tD, s_i)$ is in the  basis  of the free $\cR$-module $\Gr^{\{a,b\}}_{2m}(\cSs(\fS'))$ described in Proposition \ref{r.basis2}.

For non-negative integers $n,i$ the quantum binomial coefficient is defined by
$$ \binom ni _q = \frac{\prod_{j=n-i+1}^n (1-q^j)}{\prod_{j=1}^i (1-q^j) }.$$

\begin{proposition}\label{r.grbinom} In $\Gr^{\{a,b\}}_{2m}(\cSs(\fS'))$ one has
\be
\label{eq.split2} 
\theta_c(D) = \sum_{i=0}^m \binom{m}{i}_{q^4} (\tD, s_i).
\ee
\end{proposition}
\def\reordonegram{\raisebox{-13pt}{\incl{1.3 cm}{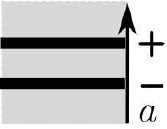}}}
\def\reordonegrbm{\raisebox{-13pt}{\incl{1.3 cm}{reord1grbm}}}
\def\reordonegrb{\raisebox{-13pt}{\incl{1.3 cm}{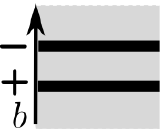}}}
\def\reordonegrbm{\raisebox{-13pt}{\incl{1.3 cm}{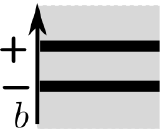}}}
\def\reordnine{\raisebox{-13pt}{\incl{1.3 cm}{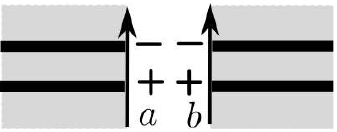}}}
\def\reordninem{\raisebox{-13pt}{\incl{1.3 cm}{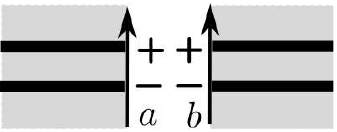}}}
\begin{proof} For $s: D \cap c \to \{ \pm \}$ let $(\tD,s)$ be the stated $\ori'$-ordered $\fS'$-tangle diagram with state $s$ on $a$ and $b$.
By definition,
\be 
\theta(D) = \sum_{s: D \cap c \to \{ \pm \} }  (\tD, s) \quad \in \cSs(\fS').
\ee 
Taking into account the filtration, from relations \eqref{eq.order} and \eqref{eq.reor1}, we see that in $\Gr^{\{a,b\}}_{2m}(\cSs(\fS'))$,
\be \label{eq.swap} \reordonegra  = q^2 \left( \reordonegram\right), \quad  \reordonegrb = q^2 \, \left( \reordonegrbm \right).
\ee
It follows that $\Gr^{\{a,b\}}_{2m}(\cSs(\fS'))$ we have
\be 
\label{eq.switch}
\reordnine \ =\ q^4 \left (  \reordninem 
\right).
\ee
Suppose $s: D \cap c \to \{\pm\}$ has $i$ minus values. 
For $k=1,\dots, i$ let $x_k$ be the number of plus states (of $s$) lying below the $k$-th minus state. By doing many switches, each changing  a pair of  consecutive $(-, +)$ to $(+,-)$, we can transform $s$ into $s_i$. The number of switches is $x_1 + \dots + x_i$. Hence from \eqref{eq.switch} we see that 
$$ (\tD,s) = q^{4(x_1+ \dots x_i)} (\tD, s_i).$$
Taking the sum over all $s: D \cap c \to \{\pm\}$ with  $i$ minus values, we get
 \be  \label{eq.fg1}
\theta(D) = \sum_{i=0}^m \left( \sum_{0\le x_1 \le \dots \le x_i \le m-i} q^{4( x_1 + \dots + x_i)}\right) \, (\tD, s_i) \quad \text{in} \ \Gr^{\{a,b\}}_{2m}(\cSs(\fS')).
\ee
By induction on $i$ one can easily prove that
\be
 \sum_{0\le x_1 \le \dots \le x_i \le n} q^{4( x_1 + \dots + x_i)} =\binom {n+i}i_{q^4},
\ee
from which and \eqref{eq.fg1} we get \eqref{eq.split2}.
\end{proof}

\subsection{The category of punctured bordered surfaces and the functor $\cSs$.} \label{sec.cat1}

A {\em morphism} from  one bordered punctured surface $\fS$ to another one $\fS'$ is an isotopy  class of orientation-preserving embeddings from $\fS$ to $\fS'$. Here we  assume that the embeddings map a boundary edge of $\fS$ into (but not necessarily onto) a boundary edge of $\fS'$.

Very often we identify an embedding $f: \fS\embed \fS'$ with its isotopy class.

\def\inv{\mathrm{inv}}
\def\id{\mathrm{id}}
\def\er{e_r}
\def\el{e_l}

Suppose $f: \fS \to \fS'$ is an embedding representing a morphism from $\fS$ to $\fS'$. Define an $\cR$-linear homomorphism $f_*: \Ss(\fS)\to \Ss(\fS')$ such that if $\al$ is a stated tangle diagram on $\fS$ with positive order then $f_*(\al)$ is given by the  stated tangle diagram $f(\al)$, also with positive order.
 It is clear that $f_*$ is an $\cR$-linear homomorphism, and it does  not change under isotopy of $f$. 

In general $f_*$ is not an $\cR$-algebra homomorphism. However, if every edge of $\fS'$ contains the image of at most one edge of $\fS$, then $f_*$ is an $\cR$-algebra homomorphism.

\begin{example} {\em A bigon } is the standard closed disk in the plane with two points on its boundary removed. Suppose $a\subset \fS$ is an arc whose two end points are in distinct boundary edges $e_1$ and $e_2$, where 
\end{example}

\begin{example} Let $e$ be a boundary edge of $\fS$ and $\fS'= \fS \setminus \{v\}$, where $v \in e$. The embedding $\iota: \fS' \embed \fS$ induces an $\cR$-linear homomorphism $\iota_*: \cS(\fS') \to \SS$ which is surjective but not injective in general.

 Suppose $e'\subset e$ is one of the two boundary edges of $\fS'$ which is part of $e$. There is a diffeomorphism $g: \fS\to \fS' \setminus \{e'\}$ which is unique up to isotopy. Thus, we have a morphism $f: \fS \to \fS'$, which is the composition
 $$ \fS\overset g \longrightarrow  \fS' \setminus \{e'\} \embed \fS'.$$
The morphism $f$ induces an injective (but not surjective) algebra morphism $f_*: \SS \embed \cS(\fS')$. 
\end{example}

\def\bmu{{\vec \mu}}
\def\bnu{{\vec\nu}}
\def\boeta{{\vec \eta}}

\section{Hopf algebra structure of the bigon and $\OSL$} \label{sec:bigon}

We will define using geometric terms a  \red{dual quasitriangular (a.k.a. cobraided)} Hopf algebra structure on the stated skein algebra $\cS(\cB)$ of the bigon $\cB$ and then show that it is naturally isomorphic to the \red{dual quasitriangular} Hopf algebra $\OSL$. We also show simple pictures of the canonical basis of $\OSL$, and discuss the Jones-Wenzl idempotents in $\cS(\cB)$. In this section $\cR=\BZ[q^{\pm 1/2}]$ unless otherwise stated.

\def\monogon{\raisebox{-13pt}{\incl{1.3 cm}{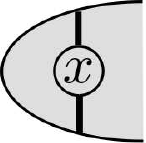}}}
\def\monogona{\raisebox{-13pt}{\incl{1.3 cm}{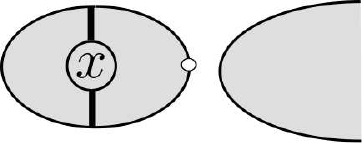}}}
\subsection{Monogon and Bigon}\label{sub:monogon} Let $D$ be the standard disk
$$ D = \{ (x,y) \in \BR^2 \mid x^2 + y^2 \le 1\}$$
and $v_1=(0,-1)$ and $v_2 = (0,1)$ are two points on the circle $\partial D$. The punctured bordered surface $\cM= D\setminus \{v_1\}$ is called the {\em monogon}, and 
 $\cB= D\setminus \{v_1, v_2\}$ is called the {\em bigon}. Let $\el, \er$ be the two  boundary edges of $\cB$ as depicted in Figure \ref{fig:bigon0}. For $\bmu=(\mu_1, \dots, \mu_k)$ and $\boeta=(\eta_1,\dots,\eta_k)$ in $\{\pm\}^k$ let $\al_{\boeta\bmu}\in \cSs(\fB)$ be the element presented by $k$ parallel arcs as in Figure  \ref{fig:bigon0}, with $(\eta_1,\dots,\eta_k)$ being the states on $\el$ in increasing order and $(\mu_1, \dots, \mu_k)$ being the states on $\er$ in increasing order.

\FIGc{bigon0}{Monogon, bigon with its edges $e_l$ and $e_r$, elements $\al_{+-}$ and  $\al_{\boeta\bmu}$. Note that for $\al_{\boeta\bmu}$ the height order is indicated by the arrows.}{2.6cm}
We have $\Ss(\cM)=\cR$. 
\red{Moreover, one has

\be \monogon =  \monogona
\label{eq.mon0}
\ee
Here the circle enclosing $x$ and two vertical lines attaching to it stand for a stated $\pfS$-tangle diagram. The proof follows by using the skein relation \eqref{eq.skein}, then the loop relation \eqref{eq.loop}, and finally the arc relation \eqref{eq.arcs} to reduce $x$ to a scalar.
}

We study the algebra $\Ss(\cB)$ in this section.
\def\rot{\mathrm{rot}}
 Let $\rot: \cB\to \cB$ be the rotation (of the plane containing $\cB$) by $180^o$ about the center of $\cB$, which is a self-diffeomorphism  of $\cB$ and induces an $\cR$-algebra involution
 \be 
 \rot_*: \Ss(\cB) \to \Ss(\cB).
 \label{def.rot}
 \ee

\subsection{Coproduct}
Suppose $e$ is a boundary edge of a punctured bordered surface  $\fS$. Let $\fS'$ be the result of cutting out of $\fS$ a bigon $\cB$ whose right edge $\er$ is identified with $e$.
Since $\fS'$ is canonically isomorphic to $\fS$ in the category of punctured bordered surfaces, we will identify $\Ss(\fS)$ with $\Ss(\fS')$.
 The splitting homomorphism gives  an injective algebra homomorphism 
 $$\Ss(\fS')  \embed \Ss(\fS\sqcup \fB) \equiv  \Ss(\fS) \otimes _\cR \Ss(\cB).$$ Since we identify $\Ss(\fS)$ with $\Ss(\fS')$, this map becomes an $\cR$-algebra homomorphism 
\be \label{eq.comodule}
  \D_e: \Ss(\fS)   \embed \Ss(\fS) \otimes _\cR \Ss(\cB).\ee
  \FIGc{split1}{The coaction $\D_e$.}{2.5cm}

Suppose $x\in B(\fS, \ori)$ is a basis element, where $\ori$ is a given orientation of $\partial \fS$. Assume the state of $x$ on $e$ is $\bmu$, then we have (see Figure \ref{fig:split1}):
\be
\D_e(x) = \sum_{\boeta  \in S_{x\cap e}} x_\boeta \otimes\,  \al_{\boeta\bmu}, \label{eq.coact}
\ee 
where $S_{x \cap e}$ is the set of all states of $x\cap e$ and $x_\boeta$ is $x$ with the state on $e$ switched to $\boeta$.

In particular, when $\fS=\cB$ and $e=\er$, we get an $\cR$-algebra homomorphism $\Delta= \Delta_{\er}$,
\be 
\Delta: \Ss(\cB) \to \Ss(\cB) \otimes_ R \Ss(\cB).
\ee

Theorem \ref{thm.1a}(c) implies that  $\Delta$ is coassociative, i.e.
\be 
(\Delta \otimes \id) \Delta = (\id \otimes \Delta ) \Delta.
\ee

Applying \eqref{eq.coact} to $x=\al_{\nu\mu}$ with $\nu,\mu\in \{\pm\}$, we get
\be 
\label{eq.coproduct1}
\Delta(\al_{\nu\mu} )= \sum _{\eta\in \{\pm \}} \al_{\nu\eta} \otimes \al_{\eta \mu}.
\ee
\subsection{Presentation of $\Ss(\cB)$} A presentation of the algebra $\Ss(\cB)$ was given in \cite{Le:TDEC}. We give here a presentation of $\Ss(\cB)$ in a form which is suitable for us. Recall that $C(\eta)= C^{\bar \eta}_{ \eta}$ for $\eta\in \{\pm \}$ were defined by \eqref{eq.Cve}. We form the following matrix

\be  C:= \begin{pmatrix}
C^+_+ & C^+_- \\ C^-_+ & C^-_-
\end{pmatrix} = \begin{pmatrix}
0 & q^{-1/2} \\ -q^{-5/2} & 0
\end{pmatrix}
\ee

\def\antipode{\raisebox{-18pt}{\incl{1.5 cm}{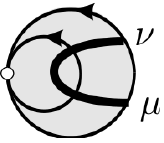}}}
\def\aptwo{\raisebox{-18pt}{\incl{1.5 cm}{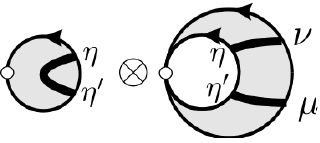}}}

\begin{lemma}\label{lem:relations}
The $\cR$-algebra $\cSs(\B)$ is generated by $\{\alpha_{\nu,\mu} \mid \nu,\mu\in \{\pm\}\}$ with the following relations: 
\begin{align}
C &= A^t C A   \label{eq.rel1a}\\
C&= A C A^t, \label{eq.rel2a}
\end{align}
where $A^t$ is the transpose of $A$ and
$$ A := \begin{pmatrix}
\al_{++} & \al_{+-} \\ \al_{-+} & \al_{--}
\end{pmatrix}.
$$

\end{lemma}
\begin{proof} Explicitly, the relations \eqref{eq.rel1a} and  \eqref{eq.rel2a} are respectively
\begin{align}
C^{\nu}_{\mu}\cdot 1& = \sum_{\eta\in \{ \pm \}} C(\bar\eta) \al_{\eta \nu} \al_{\bar\eta \mu}= C^{+}_-\alpha_{+\nu}\alpha_{-\mu}+C^{-}_+\alpha_{-\nu}\alpha_{+\mu} \qquad \forall \nu,\mu\in \{\pm\}\label{eq.relations1}\\
C^{\nu}_{\mu}\cdot 1&=  \sum_{\eta\in \{ \pm \}} C(\bar \eta) \al_{\nu \eta } \al_{\mu\bar\eta }= C^{+}_-\alpha_{\nu+}\alpha_{\mu-}+C^{-}_+\alpha_{\nu-}\alpha_{\mu+}\qquad \forall \nu,\mu\in \{\pm\}\label{eq.relations2}.
\end{align}

Let $e$ be the only boundary edge of the monogon $\cM$. Because $\Ss(\cM)= \cR$ and $\cR \otimes \Ss(\cB)= \Ss(\cB)$, the $\cR$-algebra map $\D_e: \Ss(\cM) \to \Ss(\cM) \otimes_{\cR} \Ss(\cB)$ is an $\cR$-algebra map $\D_e: \cR \to \Ss(\cB)$. As any $\cR$-linear map, we must have $\D_e(c)= c \cdot 1$, where $1$ is the unit of $\Ss(\cB)$. Apply $\Delta_e$ to the simple arc in the monogon whose endpoints are stated by $\nu$ and $\mu$ and we get a 
 proof of \eqref{eq.relations1} as follows:
$$ C^\nu_\mu= \antipode \overset {\D_e} \longrightarrow \sum_{\eta, \eta'\in \{\pm \}} \aptwo=  \sum_{\eta\in \{ \pm \}} C(\bar\eta) \al_{\eta \nu} \al_{\bar\eta \mu}.$$
Equation \eqref{eq.relations2} is obtained from Equation \eqref{eq.relations1} by applying the map $\rot_*$ of \eqref{def.rot}.

Using  Theorem \ref{thm.basis1a} one sees that the set
\be  B=  \{ \alpha_{++}^h\alpha_{-+}^k\alpha_{--}^l    \mid h,k,l\in \BN\} \cup \{ \alpha_{++}^h\alpha_{+-}^k\alpha_{--}^l \mid h,k,l\in \mathbb{N}, k \ge 1\}
\label{eq.basis}
\ee
is an $\cR$-basis of $\Ss(\cB)$. In particular, $\Ss(\cB)$ is generated by $\al_{\nu\mu}$ with $\nu,\mu\in \{\pm \}$.

Using these relations it is easy to check that any monomial in the $\al_{\nu\mu}$ can be expressed as an $\cR$-linear combinations of   $B$. The proposition follows.
\end{proof}

\subsection{Counit} 
 The embedding $\iota: \cB \embed \cM$ gives rise to an $\cR$-linear map $\iota_*: \Ss(\cB) \to \Ss(\cM)=\cR$. Define $\epsilon: \Ss(\cB) \to \cR$ as the composition $\epsilon= \iota_* \circ \inv_{e_r}$, 
$$ \epsilon: \Ss(\cB) \overset {\inv_{e_r}} \longrightarrow \Ss(\cB) \overset {\iota_*} \longrightarrow \Ss(\cM)=\cR,$$
where $\inv_{e_r}$ is defined in Section \ref{sec.invers}.
Explicitly, if $\al$ is a  stated $\partial \cB$-tangle diagram as in Figure \ref{fig:counit}, then
\be  \epsilon(\al) = C_{e_r}(\al) \, \al',
\label{eq.counit}
\ee
where $\al'$ is described in Figure \ref{fig:counit} and  $C_{e_r}(\al)$ is defined by \eqref{eq.Ce}.

\FIGc{counit}{How to obtain $\al'$ and $\al''$ from $\al$ in the definition of counit and antipode. Height order  is  indicated by the arrows on the boundary edges. Then $\al'$ is the same $\al$, but considered as a tangle diagram in $\cM$ with  its states on the edge $e_r$ switched from $\nu$ to $\bar \nu =-\nu$. And $\al''$ is obtained from $\al$ by a rotation of $180^o$, and switching all the states $\nu$ to $\bar \nu$.
 }{4cm}

Using \eqref{eq.counit} and the values of $C(\eta)$, one can check that 
\be
\label{eq.counit1} 
\epsilon(\al_{\nu\mu}) = \delta_{\nu\mu}: = \begin{cases} 1 \quad &\text{if} \ \nu=\mu \\
0 &\text{if} \ \nu\neq \mu \end{cases}.
\ee

\def\counitthree{\raisebox{-18pt}{\incl{1.5 cm}{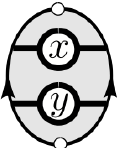}}}
\def\counitfour{\raisebox{-18pt}{\incl{1.5 cm}{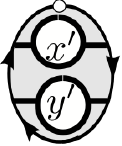}}}
\def\counitfive{\raisebox{-18pt}{\incl{1.5 cm}{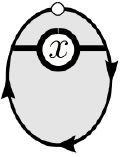}}}
\def\counitx{\raisebox{-18pt}{\incl{1.5 cm}{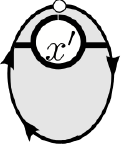}}}
\def\counity{\raisebox{-18pt}{\incl{1.5 cm}{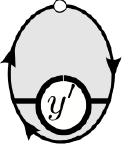}}}

\begin{proposition} The algebra $\Ss(\cB)$ is a bialgebra with  counit $\epsilon$ and coproduct $\D$.

\end{proposition}

\begin{proof} We already saw that $\D$ is an algebra homomorphism and is associative.
It remains to show that $\epsilon$ is an algebra homomorphism, and 
\be
\label{eq.coco}
(\epsilon \otimes \id ) \circ \Delta(x) =x =  (\id \otimes \epsilon  ) \circ \Delta(x).
\ee
Let $x, y\in \Ss(\cB)$ be presented by  tangle diagrams schematically depicted as in Figure \ref{fig:counit2}. 
\FIGc{counit2}{Elements $x,y\in \Ss(\cB)$. Each horizontal strand stands for several horizontal lines which are tangled in the two small disks.}{1.6cm}

The following shows that $\epsilon(xy)= \epsilon(x)\epsilon(y)$, i.e. $\epsilon$ is an algebra homomorphism:
$$ \epsilon(xy)=  \epsilon\left(  \counitthree  \right) = 
C_{\er}(x)\, C_{\er}(y) \counitfour 
= C_{\er}(x)\, C_{\er}(y) \counity \, \counitx \, 
= \epsilon(y) \epsilon(x). $$
Here we use equality \eqref{eq.mon0} in the third identity.

As both $\D$ and $\epsilon$ are algebra homomorphisms, one only needs to check \eqref{eq.coco} for the generators $x=\al_{\nu\mu}$ with $\nu,\mu\in \{\pm \}$. Using \eqref{eq.coproduct1} and \eqref{eq.counit1}, we have
$$ (\epsilon \otimes \id ) \circ \Delta(\al_{\nu\mu}) =  (\epsilon \otimes \id )  \sum_\eta \al_{\nu\eta} \otimes \al_{\eta\mu} = \al_{\nu\mu},$$
which proves the first identity of \eqref{eq.coco}. The other identity is proved similarly.
\end{proof}

\subsection{Antipode} \label{sec:antipode}
 Define 
$ S: \Ss(\cB) \to \Ss(\cB)$, \text{by} $S:= \rot_* \circ ( \inv_{\er} \circ (\binv _{\el})^{-1}),$ where $\inv$ and $\binv$ are defined in Subsection \ref{sec.invers}.
Explicitly, if $\al$ is a  stated $\partial \cB$-tangle diagram as in Figure \ref{fig:counit}, then
\be  S(\al) = \frac{C_{\er}(\al) }{ C_{\el}(\al)} \, \al'',
\label{eq.antipode}
\ee
where $\al''$ is described in Figure \ref{fig:counit}. In particular, we have
\begin{equation}
S(\al_{\nu\mu}) = \frac{C( \mu)}{C( \nu)} \al_{\bar \mu \bar \nu} \label{eq.ap1}.
\end{equation}
Explicitly,
\be 
S(\al_{++}) = \al_{--}, \ S(\al_{--})= \al_{++}, \ S(\al_{+-}) = -q^{2}\al_{+-}, \ S(\al_{-+}) = -q^{-2} \al_{-+}.
\ee

\begin{proposition}
The map $S$ is an antipode of the bialgebra $\Ss(\cB)$, making $\Ss(\cB)$ a Hopf algebra.
\end{proposition}
\begin{proof}
From the definition \eqref{eq.antipode} one  sees that $S$ is an anti-homomorphism, i.e.
$$S(xy)= S(y) S(x).$$
It remains to check the following property of an antipode:
\be
\label{eq.ap}
\sum  S(x') x'' = \epsilon(x) 1 = \sum  x' S(x'') ,
\ee
where we use the Sweedler's notation for the coproduct $\D x = \sum x' \otimes x''$.
Since $S$ is an anti-homomorphism and $\epsilon$ is an algebra homomorphism, it is enough to check \eqref{eq.ap} for generators $x= \al_{\nu\mu}$. In that case, using  \eqref{eq.coproduct1} we have
$ \D(x)=  \D(\al_{\nu\mu})=\sum_\eta \al_{\nu\eta} \otimes \al_{\eta\mu}$, and
\begin{align*}
 \sum S(x') x''&= \sum_{\eta\in \{\pm \}}  S(\al_{ \nu \eta})  \al_{\eta\mu}\\
& = \sum_{\eta\in \{\pm \}} \frac{C(\eta)}{C(\nu)} \al_{\bar \eta \bar \nu}  \al_{\eta\mu} \quad \text{by  \eqref{eq.ap1}}\\
&= \frac{C^{\bar \nu}_ \mu \cdot 1 }{C(\nu)} \quad \text{by  \eqref{eq.relations1}}\\
&= \delta_{\nu\mu} \cdot 1 = \epsilon(\al_{\nu\mu}) \cdot 1 \quad \text{by definition of $C(\nu)$ and  \eqref{eq.counit1}},
\end{align*}
which proves the first identity of \eqref{eq.ap1}. The second identity of \eqref{eq.ap1} is proved similarly.
\end{proof}

\subsection{Quantum algebra $\OSL$}
Let us recall the definition of the  quantum coordinate ring $\OSL$ of $\mathrm{SL}_2(\BC)$, which is the Hopf dual of the quantum group
$\USL$. See \cite{MajidFoundation}.
\begin{definition}[$\OSL$]
The Hopf algebra $\OSL$ is the $\cR$-algebra generated by $a,b,c,d$ with 
relations 
\begin{align}  ca& =q^2ac, \ db=q^2bd, \ ba=q^2ab,\ dc=q^2cd, \label{eq.rel1} \\
 bc& =cb,  \ ad-q^{-2}bc=1  \ {\rm and} \  da - q^2cb=1 \label{eq.rel2}.
\end{align}
Its coproduct structure is given by
 $$\Delta(a)=a\otimes a+b\otimes c,\qquad \Delta(b)=a\otimes b+b\otimes d,
\qquad \Delta(c)=c\otimes a+d\otimes c, \qquad\Delta(d)=c\otimes b+d\otimes d.$$ 
Its counit is defined as $\epsilon(a)=\epsilon(d)=1,\epsilon(b)=\epsilon(c)=0$
 and its antipode is defined by $S(a)=d,S(d)=a,S(b)=-q^2 b,S(c)=-q^{-2} c$.
\end{definition}
\begin{theorem}\label{teo:SL2}
There exists a Hopf algebra isomorphism $\phi: \Ss(\cB) \to \OSL$
given on the generators by 
\be \phi(\alpha_{+,+})=a,\qquad \phi(\alpha_{+,-})=b,\qquad\phi(\alpha_{-,+})=c,\qquad\phi(\alpha_{-,-})=d.
\label{eq.def1}
\ee
Furthermore, 
\red{
under the identification of $\Ss(\cB)$ with $ \OSL$ via the isomorphism $\phi$, the involution $\rot_*: \Ss(\cB)\to \Ss(\cB)$,  given by the rotation of $180^{\circ}$ around the center of the bigon (see \eqref{def.rot}), becomes the $\cR$-algebra involution $r:\OSL\to \OSL$ given by
 $r(a)=a,r(b)=c,r(c)=b, r(d)=d$. Moreover, $r$ is a co-algebra antimorphism, i.e. 
 $$(r\otimes r) \circ \Delta^{op}=\Delta\circ r.$$  
 }
\end{theorem}
\begin{proof}\red{By Lemma \ref{lem:relations}, the $\cR$-algebra $\Ss(\cB)$ is generated by $\al_{\pm\pm}$ with relations \eqref{eq.rel1a} and \eqref{eq.rel2a}. Under the assignment $\phi$ given 
on generators $\al_{\pm\pm}$
by \eqref{eq.def1}, 
the matrix relations \eqref{eq.rel1a} and \eqref{eq.rel2a} become respectively
\begin{align}  ca&= q^2 ac ,\  db= q^2 bd,\  ad- q^{-2} cb=1,\  da -q^2 bc=1 \label{eq.rel1as} \\
ba&= q^2 ab,\  dc= q^2 cd,\  ad- q^{-2} bc=1,\ da - q^2 cb=1 \label{eq.rel1bs}
\end{align}
Each of these identities is a consequence of the relations in \eqref{eq.rel1} and \eqref{eq.rel2}. Conversely, all the relations in \eqref{eq.rel1} and \eqref{eq.rel2}, except for $bc=cb$, are among the identities  \eqref{eq.rel1as} and \eqref{eq.rel1bs}. The remaining relation $bc=cb$ is obtained by taking the difference between the last identity of \eqref{eq.rel1as} and the last identity of \eqref{eq.rel1bs}. Hence $\phi$ is an $\cR$-algebra isomorphism.
}

To check that $\phi$ is a Hopf algebra isomorphism it is sufficient to check this on the level of generators where it is straightforward. 

The last statement is a direct verification.
\end{proof}

\subsection{Geometric depiction of co-$R$-matrix, a lift of the co-$R$-matrix} The Hopf algebra $\OSL$ is \red{\em ``dual quasitriangular'' (see  \cite{MajidFoundation} Section 2.2) or ``cobraided'' (see e.g. \cite{Kassel}, Section VIII.5)} , i.e. it has a co-$R$-matrix with the help of which one can make the category of $\OSL$-modules a braided category. Formally, a {\em co-$R$-matrix}  is a bilinear form 
$$ \rho : \OSL \otimes \OSL \to \cR$$ such that there exists another bilinear form $\bar \rho: 
\OSL \otimes \OSL \to \cR$ (the ``inverse" of $\rho$)
satisfying for any $x,y,z\in U$,
\begin{align}
\sum \rho (x' \ot y') \bar \rho (x'' \ot y'')&= \sum \bar \rho (x' \ot y') \rho (x'' \ot y'')= \epsilon (x) \epsilon(y) \\
\sum \rho (x''\ot y'') y'x'  &= \sum \rho (x'\ot y') x'' y''     \label{eq.rflip}\\
\rho (xy \ot z)& = \sum \rho (x' \ot z') \rho  (y'' \ot z'') \epsilon(x'') \epsilon (y')  
\label{eq.cobraid3}\\
\rho (x \ot y z)& = \sum \rho  (x' \ot z') \rho (x'' \ot y'') \epsilon(z'') \epsilon (y') 
\label{eq.cobraid4}
\end{align}
Here we use Sweedler's notation for the coproduct.
Relations \eqref{eq.cobraid3} and \eqref{eq.cobraid4} show that $\rho$ is totally determined by  its values at a set of generators of the algebra $\OSL$, and the values $\rho$ at a set of generators are given by (see \cite{Kassel})
\be 
\rho \begin{pmatrix}
a\ot a & b \ot b & a\ot b & b \ot a \\
c\ot c & d \ot d & c \ot d & d \ot c\\
a\ot c & b \ot d & a \ot d & b \ot c\\
c\ot a & d \ot b & c \ot b & d \ot a
\end{pmatrix} = 
\begin{pmatrix}
q & 0 & 0 & 0 \\
0 & q & 0 & 0\\
0 & 0 & q^{-1} & q - q^{-3}\\
0 & 0 & 0 & q^{-1}
\end{pmatrix}   \label{eq.cobraid0}
\ee

\def\letterx{  \raisebox{-9pt}{\incl{.9 cm}{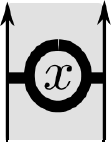}} } 
\def\lettery{  \raisebox{-9pt}{\incl{.9 cm}{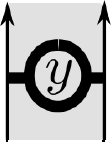}} }
\def\letterxy{  \raisebox{-9pt}{\incl{.9 cm}{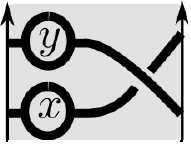}} }
\def\xybar{  \raisebox{-9pt}{\incl{.9 cm}{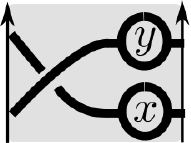}} }
\def\yxbar{  \raisebox{-9pt}{\incl{.9 cm}{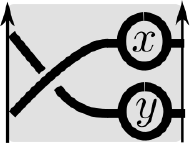}} }
\def\letterxyz{  \raisebox{-10pt}{\incl{1 cm}{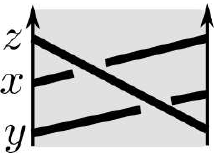}} } 
\def\xyzp{  \raisebox{-10pt}{\incl{1 cm}{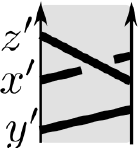}} } 
\def\xyzpp{  \raisebox{-10pt}{\incl{1 cm}{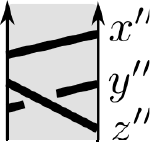}} } 
\def\letterbc{  \raisebox{-10pt}{\incl{1 cm}{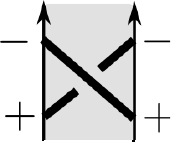}} }
\def\letterbcp{  \raisebox{-10pt}{\incl{1 cm}{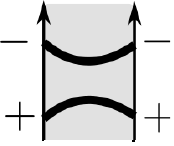}} }
\def\letterbcpp{  \raisebox{-10pt}{\incl{1 cm}{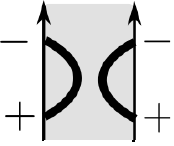}} }
\def\pcB{\partial \cB}
 
 \begin{theorem}\label{teo:cobraid} Under the identification of $\cS(\cB)$ and $\OSL$ via the isomorphism $\phi$,
 the co-$R$-matrix $\rho$ and its inverse $\bar \rho$  have the following geometric description
 \begin{align}
 \rho \left (\letterx  \ot \lettery \right) & = \epsilon \left(  \letterxy \right)  \label{eq.cobraid} 
 \\
  \bar \rho \left (\letterx  \ot \lettery \right) & = \epsilon \left(  \xybar \right). \label{eq.cobraid2}
\end{align}   
 \red{Here a circle enclosing $x$ and two lines adjacent to the circle stand for a stated $\pcB$-tangle diagram, also denoted by $x$. The left hand side of \eqref{eq.cobraid} stands for $\rho(x\ot y)$. }
 \end{theorem}
 \begin{proof}
 Let $\rho'$ be the map defined by the r.h.s. of \eqref{eq.cobraid}: we will show that $\rho'=\rho$. For this it is enough to show that $\rho'$ satisfies \eqref{eq.cobraid3}, \eqref{eq.cobraid4}, and the initial values identity \eqref{eq.cobraid0}, all with $\rho$ replaced by $\rho'$. 
We have, \red{where a line labeled by, say $x$, stands for the stated $\pcB$-tangle diagram $x$},
$$\rho'(xy \ot z)= \epsilon \left(\letterxyz   \right)$$
Split the bigon by the vertical middle ideal arc, then use the fact that $\epsilon(u) =\sum \epsilon(u') \epsilon(u'')$ (in any Hopf algebra) where $\Delta(u)= \sum u' \ot u''$, we have
\begin{align*}
&= \sum \epsilon \left(\xyzp   \right) \epsilon \left(\xyzpp   \right)   \\
&= \sum \rho  (x' \ot z') \rho  (y'' \ot z'') \epsilon(x'') \epsilon (y'). 
\end{align*}
This proves \eqref{eq.cobraid3} for $\rho'$. The proof of \eqref{eq.cobraid4} is similar.

To check \eqref{eq.cobraid0} we have to check 16 identities, all of which are easy. For example, the most difficult one is the identity of the $(3,4)$ entries:
\begin{align*}
\rho'(b\ot c) &= \epsilon \left( \letterbc \right) = q \, \epsilon \left( \letterbcp \right) + q^{-1} \, \epsilon \left( \letterbcpp \right)  \\
&= q - q^{-3} \qquad \text{by \eqref{eq.counit1} and \eqref{eq.arcs2}}. \ 
\end{align*}
 This proves \eqref{eq.cobraid0} for  the $(3,4)$ entries. Identity \eqref{eq.cobraid0} for  other entries are similar.
 \end{proof}
 
\begin{remark}\label{rem:rhoprime} \red{The bilinear form $ \rho':  \OSL \otimes \OSL \to \cR$ defined by
\be 
 \rho'(x \otimes y) := \bar \rho(y \otimes x) =  \epsilon \left(  \yxbar \right) \label{eq.coR2}
\ee
gives a new co-$R$-matrix for $\OSL$, which is the mirror reflection of $\rho$.}
\end{remark}

\subsection{The Jones-Wenzl idempotents as elements of the bigon algebra}\label{sub:joneswenzl}
\red{
In this subsection we will work over the ring $\cR^{loc}$ obtained by localizing $\cR$ over the multiplicative set generated by $\{[n]=\frac{q^{2n}-q^{-2n}}{q^2-q^{-2}}, n\geq 1\}$.
Recall that the $n^{th}$ Temperely-Lieb algebra $T_n$ (see e.g. \cite{Tu:Book}) is the $\cR^{loc}$-algebra generated by non-stated simple $(n,n)$-tangle diagrams in $\B$ modulo isotopy (rel to the boundary) and relation \eqref{eq.loop}. The product is obtained by concatenating horizontally. 
\begin{figure}[hbtp!]
 \includegraphics[width=2cm]{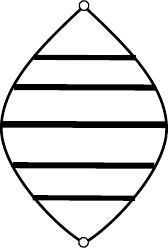} \includegraphics[width=2cm]{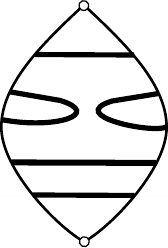}
\caption{On the left the unit of $T_5$. On the right an element 
of $T_5$.}\label{fig:temperely}
\end{figure}

A $\cR^{loc}$-basis of $T_n$ is given by simple $(n,n)$-tangle diagrams without closed components; 
define $\epsilon:T_n\to \cR^{loc}$ be the dual of the element $1$ with respect to this basis. 
The $n^{th}$ Jones-Wenzl idempotent is an element $JW_n\in T_n$ defined by recursion as explained in Figure \ref{fig:joneswenzl}.
\begin{figure}[hbtp!]
\raisebox{-1.3cm}{\includegraphics[width=2cm]{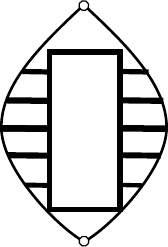}}\ =\ \raisebox{-1.3cm}{\includegraphics[width=2cm]{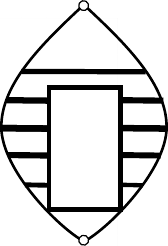}}$+\frac{[n-1]}{[n]}$ \raisebox{-1.3cm}{\includegraphics[width=2cm]{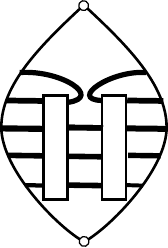}}
\caption{The recursion relation for $JW_n\in T_n$. By definition $JW_1=1\in T_1$.}\label{fig:joneswenzl}
\end{figure}

The following is the key property of the $JW_n$, see e.g. \cite{Tu:Book}.
\begin{proposition}\label{prop:joneswenzlproperties} One has $\epsilon(JW_n)= 1$. 
For all $x\in T_n$, 
it holds $JW_n  x=x JW_n=\epsilon(x)JW_n$. In particular $JW_n^2=JW_n$.
\end{proposition}
For a simple $(n,n)$-tangle diagram $x\in T_n$ and $\vec{\mu_l},\vec{\mu_r} \in \{\pm \}^n$ let $x(\vec{\mu_l},\vec{\mu_r})$ be the stated $\pcB$-tangle diagram which is $x$ with states $\vec{\mu_l}$ on $e_l$ and states $\vec{\mu_r}$ on $e_r$, and the height order on each of $e_l$ and $e_r$ is from bottom to top. By linearity, for  $y\in T_n$, we define  $y(\vec{\mu_l},\vec{\mu_r})$. This is well-defined since \eqref{eq.loop} is part of the defining relations of stated skein algebra. Thus, if $y$ is trivial $(n,n)$-tangle diagram, then $y(\boeta, \bomu)$ is the element $\al_{\boeta \bomu}$ described in Figure \ref{fig:bigon8}.
\begin{example}\label{ex:joneswenzl} If all the components of $\boeta$ are the same and all the components of $\bomu$ are the same, then 
$JW_n(\boeta, \bomu)= \al_{\boeta \bomu}$.
 Indeed $JW_n$ is equal to trivial $(n,n)$-tangle diagram plus the linear combination of diagrams each contains an arc whose endpoints are both in $e_l$ or in $e_r$; such an arc is 0 by \eqref{eq.arcs}.
\end{example}
\begin{proposition}\label{prop:joneswenzl} For $\bomu\in \{\pm\}^n$ let $o(\bomu)$  be obtained by reordering increasingly the states of $\bomu$ and $no(\vec{\mu})$ is the minimal number of exchanges needed to do so. Then
\begin{align} JW(\vec{\mu_l},\vec{\mu_r}) &=q^{2\ no(\vec{\mu_l})+2\ no(\vec{\mu_r})} JW(o(\vec{\mu_l}),o(\vec{\mu_r})) \\
\Delta(JW(\vec{\mu_l},\vec{\mu_r}))&=\sum_{j=0}^n \binom{n}{j}_{q^4} JW(\vec{\mu_l},\vec{\eta_j})\otimes JW(\vec{\eta_j},\vec{\mu_r}),
\end{align}
where $\vec{\eta_j}$ is the increasing state containing $j$ signs $+$ and $n-j$ signs $-$.
\end{proposition}

}

\begin{proof}
Observe that if one exchanges a sign $-$ and a $+$ which are not in the increasing order along $e_r$ then by \eqref{eq.order} one gets $q^2$ times the reordered term and $q^{\frac{1}{2}}$ times a term killed by $JW_n$. Since a similar argument (using Lemma \ref{r.refl}) shows that each reordering along $e_l$ multiplies $JW$ by $q^2$,  the first statement follows. 
The second statement is a consequence of the fact that $JW_n^2=JW$ in $T_n$  and Proposition \ref{r.grbinom}. 
\end{proof}


\def\Bcan{B_{\mathrm{can} }}
\subsection{Kashiwara's basis for $\OSL$}
\red{
 We will see that the celebrated Kashiwara canonical basis of $\OSL$, see \cite{Ka}, is the geometrically defined basis $B(\cB,\ori_+)$ of Theorem \ref{thm.basis1a}, up to powers of $q$. Here $\ori_+$ is the positive orientation of the boundary $\pcB$ of the bigon $\cB$.

\FIGc{bigon8}{Elements $\beta_{\boeta \bomu}$ (left) and $\beta'_{\boeta \bomu}$ (right) with $\boeta=(-,+,+)$ and $\bomu=(--+)$; the difference is the direction of the left edge.}{2cm}

First recall $B(\cB,\ori_+)$. For sequences $\boeta,\bomu \in \{\pm \}^n$
 let $\beta_{\boeta,\bomu}$ be the stated $\pcB$-tangle diagram consisting of $n$ horizontal arcs and stated by $\boeta$ on the left edge and $\bomu$ on the right edge, with height order given by $\ori_+$; see 
 Figure \ref{fig:bigon8}. 
Then 
 $$ B(\cB,\ori_+)= \{ \beta_{\boeta,\bomu} \mid \boeta, \bomu  \text{ are increasing}\}.$$
Let $\beta'_{\boeta,\bomu}$ be the same $\beta_{\boeta,\bomu}$ with reverse height order on the left edge, see Figure \ref{fig:bigon8}. Define
$$ \Bcan:= \{ \beta'_{\boeta,\bomu} \mid  \beta_{\boeta,\bomu} \in B(\cB;\ori_+) \}
=\{ \al_{\boeta,\bomu} \mid \boeta \text{ is decreasing}, \bonu \text{ is increasing}.\}
 .$$
 Here $\al_{\boeta,\bomu}$ is defined in Figure \ref{fig:bigon0}.
 
From Lemma \ref{r.refl} which deals with height exchange, one can easily show that
\be 
\beta'_{\boeta,\bomu} = q^{h(\boeta)} \beta_{\boeta,\bomu}, \ h(\boeta):= \frac{1}{2}(n_+ n_- + n_+ + n_- -n_+^2 - n_-^2  )
\ee
where $n_+$ (respectively $n_-$) is the number of $+$ (respectively $-$) in the sequence $\boeta$. It follows that $\Bcan$ is also an $\cR$-basis of $\Ss(\cB)$.
}

\red{
\begin{proposition}\label{prop:kashiwara}

Via the isomorphism of Theorem \ref{teo:SL2}, the basis $\Bcan$ coincides with the canonical basis defined by Kashiwara  \cite[Proposition 9.1.1]{Ka}. Both bases $\Bcan$ and $B(\cB;\ori_+)$ are positive with respect to the product and to the coproduct, i.e. for $B= \Bcan$ or $B(\cB;\ori_+)$ and $\alpha,\beta \in B$ one has
$$ \alpha \beta\in \mathbb{N}[q^{\pm1}]\cdot B $$
$$ \Delta(\alpha)\in \mathbb{N}[q^{\pm1}]\cdot B\otimes B.$$
\end{proposition} 
}
\begin{proof}
The first statement is an observation directly following Theorem \ref{teo:SL2}: in \cite[Proposition 9.1.1]{Ka} the basis is $\{c^la^mb^n, l,m,n\geq 0\}\sqcup \{c^ld^mb^n, l,n\geq 0, m>0\}$.
\red{
As $B(\cB;\ori_+)$ is equal to $\Bcan$ up to powers of $q$, one needs only to prove the second statement for $\Bcan$.
}
For positivity of multiplication, it is sufficient to check it on pairs of generators: there are then $16$ cases. 
All of them are straightforward; we provide some instances among the most complicated cases  where the right hand sides are in  $\mathbb{N}[q^{\pm1}]\cdot\Bcan$:
 $$\alpha_{+-}\cdot \alpha_{-+}=\alpha_{-+}\cdot \alpha_{+-}, \alpha_{+-}\cdot \alpha_{--}=q^{-2}\alpha_{--}\cdot \alpha_{+-}, \alpha_{--}\cdot \alpha_{-+}=q^{2}\alpha_{-+}\cdot \alpha_{--}$$ $$\alpha_{+-}\cdot \alpha_{++}=q^2\alpha_{++}\cdot \alpha_{+-}, \alpha_{++}\cdot \alpha_{-+}=q^{-2}\alpha_{-+}\cdot \alpha_{++}, \alpha_{++}\cdot \alpha_{-+}=q^{-2}\alpha_{-+}\cdot \alpha_{++}.$$

Once positivity is known for multiplication, the statement for comultiplication can be checked on generators where it is straightforward. 
\end{proof}

\begin{remark} A direct proof of positivity using pictures is also easy and left to the reader.
\end{remark}

\section{Comodule structures, co-tensor products and braided tensor products}\label{sec.comodule}

\def\De{\Delta_e}
\def\eD{\null _e\Delta}
\def\bal{{\text{-bal}}}

In this section we show that given any edge of $\fS$ the skein algebra $\SS$ has a natural structure of $\OSL$-comodule algebra. We show how to decompose this comodule into finite dimensional comodules. We then identify the image of the splitting homomorphism using the Hochshild cohomology, and give a dual result using Hochshild homology. When $\fS$ is the result of gluing two surfaces $\fS_1$ and $\fS_2$ to two edges of an ideal triangle, we show that the skein algebra $\SS$ is canonically isomorphic to the braided tensor product of $\cS(\fS_1)$ and $\cS(\fS_2)$. In this section $\cR=\BZ[q^{\pm 1/2}]$.

\subsection{Comodule}\label{sub:comodule} Suppose  $e$ is a boundary edge of a  punctured bordered surface $\fS$. Recall that by cutting out of $\fS$ a bigon $\fB$ whose right edge is $e$ and canonically identifying $\fS\setminus int(\fB)$ with $\fS$, we get an $\cR$-linear map
$$ \De: \cSs(\fS) \to \cSs(\fS) \otimes \cSs(\fB),$$
see Figure \ref{fig:split1}.
Similarly cutting out of $\fS$ a bigon $\fB$ whose left edge is $e$ and canonically identifying $\fS\setminus int(\fB)$ with $\fS$, we get an $\cR$-linear algebra homomorphism
$$ \eD: \cSs(\fS)\to \cSs(\fB \sqcup \fS) \equiv   \cSs(\fB) \ot \cSs(\fS).$$

 \begin{proposition}\label{prop:comodule} (a) The map $\D_e: \cSs(\fS) \to  \cSs(\fS) \otimes \cSs(\fB)$ gives $\cSs(\fS)$ a right comodule-algebra structure over the Hopf algebra $\cSs(\fB)$. Similarly $\eD$ gives gives $\cSs(\fS)$ a left comodule-algebra structure over the Hopf algebra $\cSs(\fB)$. 
 
 (b) It holds $\eD=\fl \circ (\Id_{\cSs(\fS)}\otimes \rot_*)\circ \De$ where $\fl (x\otimes y)=y\otimes x$ and $\rot_*:\cSs(\B)\to \cSs(\B)$ is the algebra involution defined by \eqref{def.rot}.

(c)  If $e_1,e_2$ are two distinct boundary edges, the coactions on the two  edges commute, i.e.  $$(\Delta_{e_2}\otimes \id)\circ \Delta_{e_1}=\red{(\fl\otimes \id)\circ}(\Delta_{e_1}\otimes \id)\circ \Delta_{e_2}.$$
\end{proposition}

\begin{proof} (a) The associativity of $\D_e$ follows from the
the commutativity  of the splitting maps of theorem \ref{thm.1a}(c). Applying $(\id \ot \epsilon)$ to Equation \eqref{eq.coact} and using the value of $\epsilon(\al_{\boeta\bmu})$ from~\eqref{eq.counit1}, we get that 
$$ (\id \ot \epsilon) \D_e(x) = x_\bmu=x, \quad \forall x \in B(\fS,\ori).$$
Hence $\D_e$ gives $\cSs(\fS)$ the structure of a right $\cSs(\fB)$-comodule.

Recall that $\cSs(\fS)$ is a comodule-algebra over the bialgebra $\cSs(\fB)$, see e.g. \cite[Proposition III.7.2]{Kassel}, if and only if the map $ \De: \cSs(\fS) \to \cSs(\fS) \otimes \cSs(\fB)$ is an algebra homomorphism. The last fact follows easily from the definition of $\D_e$.

(b) Observe that $(\rot_*\otimes \rot_*)\circ \Delta^{op}=\Delta\circ \rot_*$. 

(c) is clear from the definition.
\end{proof}
By identifying $\SB$ with $\OSL$ using Theorem \ref{teo:SL2}, the above proposition also provides $\SS$ with the structure of a $\OSL$-comodule. 
More in general, we will use the following terminology:
\begin{definition}[Surfaces with indexed boundary]\label{defi:indexed}
A punctured bordered surface $\fS$ has \emph{indexed boundary} if its boundary edges are partitioned into two ordered sets (the left and right ones, with indices $L$ and $R$ respectively): $e^L_1,\ldots e^L_n, e^R_1,\ldots e^R_m$.
\end{definition}

If $\fS$ has indexed boundary then $\cSs(\fS)$  is naturally endowed with a structure of $$(\OSL^{\otimes n},\OSL^{\otimes m})-{\rm bicomodule}$$ by the left coaction $\Delta^L:\cSs(\fS)\to \OSL^{\otimes n}\otimes \cSs(\fS)$ defined by 
$$\Delta^L:=
(\Id_{\OSL}^{\otimes n-1}\otimes \eDD{e^L_n})\circ (\Id_{\OSL}^{\otimes n-2}\otimes \eDD{e^L_{n-1}})\circ \cdots \circ (\Id_{\OSL}\otimes \eDD{e^L_2})\circ \eDD{e^L_1}$$ and the right coaction 
$$\Delta^R:=(\Delta_{e^R_1}\otimes\Id_{\OSL}^{\otimes m-1} )\circ (\Delta_{e^R_2}\otimes\Id_{\OSL}^{\otimes m-2} )\circ \cdots \circ (\Delta_{e^R_{m-1}}\otimes\Id_{\OSL} )\circ \Delta_{e^R_1}.$$ 

Furthermore notice that $\cSs(\fS)$ is not only a bicomodule but a bicomodule-algebra as each of the above maps $\Delta_{e^R_i}$ or $\eDD{e^L_j}$ are also morphisms of algebras.

\subsection{Quantum group $\USL$} \label{sub:module}
Recall that the quantized enveloping algebra $\USL$ is the Hopf algebra generated over the field $\mathbb{Q}(q^{1/2})$ by $K^{\pm1},E,F$ with relations \begin{equation}\label{eq:uqprod}
KE=q^4 EK,\quad KF=q^{-4}FK,\quad [E,F]=\frac{K-K^{-1}}{q^2-q^{-2}}.\end{equation}
The coproduct and the antipode are given by 
\begin{align}
\label{eq:uqcoprod}
\Delta(K)=K\otimes K,\quad \Delta(E)=1\otimes E+E\otimes K, \quad \Delta(F)=K^{-1}\otimes F+F\otimes 1\\
S(K)=K^{-1},\ S(E)=-EK^{-1}, S(F)=-KF.
\end{align}
  
We emphasize that $\USL$ is defined over the field $\BQ(q^{1/2})$. There is an integral version $U^L_{q^2}(\mathfrak{sl}_2)$, defined by Lusztig \cite{Lusztig}, which is the $\cR$-subalgebra of $\USL$ generated by $K^{\pm1}$ and the divided powers  $E^{(n)}:=\frac{E^n}{[n]!},F^{(n)}:=\frac{F^n}{[n]!}$. Here $[n]! = \prod_{i=1}^n (q^{2i} - q^{-2i}) /(q^2 - q^{-2})$.

One has a non degenerate Hopf pairing  
\be \la \cdot, \cdot \ra: \USL\ot_\cR \OSL \to \Qq.
\label{eq.form1}
\ee
This is a Hopf duality since it satisfies (with Sweedler's coproduct notation)
\be 
\la x, y_1 y_2 \ra = \sum \la x', y_1 \ra \la x'', y_2 \ra,  \quad \la x_1 x_2, y\ra = \sum \la x_1, y' \ra \la x_2, y'' \ra
\ee
The values of the form on generators are given by
\begin{align}
 \left \langle K, \begin{pmatrix}
 a & b \\ c & d
 \end{pmatrix}  \right \rangle= \begin{pmatrix}
 q^2 & 0 \\ 0 & q^{-2}
 \end{pmatrix}, \ \left \langle E, \begin{pmatrix}
 a & b \\ c & d
 \end{pmatrix}  \right \rangle= \begin{pmatrix}
 0 & 1 \\ 0 & 0
 \end{pmatrix}, \ \left \langle F, \begin{pmatrix}
 a & b \\ c & d
 \end{pmatrix}  \right \rangle= \begin{pmatrix}
 0 & 0 \\ 1 & 0
 \end{pmatrix}.
\end{align}

\begin{lemma}
\label{lem:integralpairing} The form \eqref{eq.form1} on $U^L_{q^2}(\mathfrak{sl}_2)$ is integral, i.e. it restricts to a map
$$  U^L_{q^2}(\mathfrak{sl}_2)\ot_\cR \OSL \to \cR= \BZ[q^{\pm1/2}]  .$$

\end{lemma}
\begin{proof} It is enough to check that $\langle E^{(n)},x\rangle, \langle F^{(n)},x\rangle \in \cR$ for $n\geq 1$ and $  x\in \OSL$. 
Since $\Delta(E^{(n)})=\sum_{i=0}^{n} q^{2i(n-i)}E^{(i)}\otimes E^{(n-i)}K^i  $ it is sufficient to check the statement for the evaluations of $ E^{(i)}$ and $K^j$ on $a,b,c,d$ where this is a straightforward computation. Similarly for $F^{(n)}$. 
\end{proof}

\red{Recall that the rotation by $180^\circ$ of the bigon induces the involution $r:\OSL\to \OSL$, see Theorem \ref{teo:SL2}. Let $r^*:\USL\to \USL$ be the adjoint of the map $r$. We will show that $r^*$ is equal to the map $\rho$ of Lusztig's book \cite[Chapter 19]{Lusztig}, which is used in the study of canonical bases of quantum groups.

Let $\USL-Mod$ (respectively $Mod-\USL$) be the monoidal category of left (respectively right) $\USL$-modules.
}
\begin{lemma}[Left and right modules]\label{lem:LRfunctors}

(a) The map  $r^*$ is an algebra antimorphism involution and a coalgebra morphism, i.e. for $x,y\in \USL$  one has $$(r^*)^2(x)=x\qquad r^*(xy)=r^*(y)r^*(x), \quad \Delta(r^*(x))=(r^*\otimes r^*)\circ \Delta(x).$$

Explicitly, the value of $r^*$ on the generators is 
\be r^*(E)=q^2KF,\qquad  r^*(K)=K,\qquad r^*(F)=q^{-2}EK^{-1}.
\label{eq.rr}
\ee

(b) The map $r^*$ induces monoidal functors $$LR:\USL-Mod\to Mod-\USL, {\rm and}\ RL:Mod-\USL\to \USL-Mod$$ which are inverse to each other as follows: for each left (resp. right) module $M$
 the module $LR(M)$ (resp. $RL(M)$) is the right (resp. left) module whose underlying vector space is $M$ and on which the action of $x\in  \USL$ 
 is given by, with on $\alpha \in M$,
  $$ \alpha\cdot x:=r^*(x)\cdot \alpha \quad (resp.\ \  x\cdot \alpha:= \alpha\cdot r^*(x) ).$$

\end{lemma}

\begin{remark} Formula \eqref{eq.rr} shows that $r^*$ is equal to $\rho$ of \cite[Chapter 19]{Lusztig}.
\end{remark}
\begin{proof}
(a) Since $r$ is a algebra involution and a co-algebra antimorphism by  Theorem \ref{teo:SL2} and the Hopf pairing is non-degenerate,  its dual $r^*$ is an algebra antimorphism involution and a coalgebra morphism.

It is sufficient to compute $r^*$ on the generators where one can verify the values provided in the statement.
For instance:$$\langle r^*(E),a\rangle=\langle E,a\rangle=0=\langle q^2 KF,a\rangle=\langle q^2K\otimes F,a\otimes a+c\otimes b\rangle$$
$$\langle r^*(E),b\rangle=\langle E,c\rangle=0=\langle q^2 KF,b\rangle=\langle q^2K\otimes F,a\otimes b+b\otimes d\rangle$$
$$\langle r^*(E),c\rangle=\langle E,b\rangle=1=\langle q^2 KF,c\rangle=\langle q^2K\otimes F,c\otimes a+d\otimes c\rangle$$
$$\langle r^*(E),d\rangle=\langle E,d\rangle=0=\langle q^2 KF,d\rangle=\langle q^2K\otimes F,c\otimes b+d\otimes d\rangle$$

The verification for the pairings with $a,b,c,d$ for $r^*(F)$ and $r^*(K)$ are similar.

(b) $LR$ and $RL$ are functors as $r^*:\USL\to \USL$ is an algebra antimorphism; they are inverse to each other as $r^*$ is an involution. Monoidality is a consequence of the fact that $r^*$ is a coalgebra morphism. 
\end{proof}
\subsection{Module structure of $\cSs(\fS)$.}\label{sub:modulestructure} As usual, the Hopf duality implies that 
every right (resp. left) $\OSL$-comodule $V$  has a natural structure of a left (resp. right) $\USL$-module, via the following construction. For $a \in \USL$ and $v \in V$, one has
\be  a \cdot v:= \sum \la a, b' \ra v', \quad \text{where}\quad \Delta_r(v) = \sum v' \ot b'.
\label{eq.dual}
\ee

To be precise, we have to replace $V$ by $V\ot_\cR \Qq$, since $\USL$ is defined over $\Qq$.

 \def\SSQ{\SS \ot_\cR \Qq}

In particular, for an edge $e$ of $\fS$ the right comodule structure $\eD: \SS \to  \SS \ot \SB$ gives $\SSQ$ a left module structure over $\USL$, and we want to understand this module structure.

 Fix an orientation $\ori$ of the boundary $\pfS$. Recall that $B(\fS;\ori)$ is an $\cR$-basis of $\cSs(\fS)$. For each edge $e$ let $B_{e,d}(\fS;\ori) \subset B(\fS;\ori)$ be the set of all $\al \in B(\fS;\ori)$ such that $|\al\cap e|=d$ and all the states on $\al \cap e$ are signs $+$. Let
$B_e(\fS;\ori) = \cup _{d=0}^\infty B_{e,d}(\fS;\ori)$. 
For $\alpha\in B_{e,d}(\fS;\ori)$ and $\vec{\eta}\in \{\pm\}^d$, let $\alpha(\vec{\eta})$ be the same $\alpha$ except for the states of $e\cap \alpha$ which are given by $s(x_i)= \eta_i$, where $x_1, \dots, x_d$ are the points of $\al\cap e$ listed in decreasing order. In particular let $\alpha_j:=\alpha(+,+,\cdots ,+,-,-, \dots,-)$ where the number of $-$ is $j$. For example, $\alpha=\alpha_0$.
 \begin{lemma}[Module structures of $\cSs(\fS)$ along an edge $e$]\label{lem:uqmodule}
The left action of $\USL$ on $\SSQ$, dual to $\eD$,  is: 
$$K_{left}(\alpha_j)=q^{2(d-2j)}\alpha_j$$
$$E_{left}(\alpha_0)=0,\ {\rm and}\ E_{left}(\alpha_j)=[j]_{q^2}\alpha_{j-1} \mod F^e_{d-1}$$
$$F_{left}(\alpha_d)=0,\ {\rm and} \ F_{left}(\alpha_j)=[d-j]_{q^2}\alpha_{j+1} \mod  F^e_{d-1}$$
where  $F^e_{d-1}= F^e_{d-1}(\SSQ)$ is the $\Qq$-span of elements $\beta\in B(\fS;\ori)$ with $|\beta \cap e| < d$. 
The right action, dual to $\De$, is given by the left action of $r^*(K),r^*(E),r^*(F)$ (see Lemma \ref{lem:integralpairing}).

\end{lemma}
\begin{proof} \red{
By \eqref{eq.coact},  $\Delta(\alpha(\vec{\eta}))=\sum_{\vec{\epsilon} \in \{\pm\}^d} \alpha(\vec{\epsilon}) \otimes \alpha_{\vec{\epsilon}\vec{\eta}}$ where  $\alpha_{\vec{\epsilon}\vec{\eta}}$ is defined in Figure \ref{fig:bigon0}. If we define inductively $\Delta^{[d]} = (\Delta\ot\id^{\ot (d-2)} )\circ\Delta^{[d-1]}$ for $d\ge 3$, with $\Delta^{[2]} = \Delta$, then
\begin{align}
  \Delta^{[d]} (K)&= K^{\ot d},\\
 \Delta^{[d]} (E)&= \sum_{j=1}^d 1^{\ot (j-1)} \ot E \ot K^{\ot  (d-j)},  \label{eq.DE} \\
 \Delta^{[d]} (F)&= \sum_{j=1}^d (K^{-1}) ^{\ot j} \ot F \ot 1^{\ot d-j-1}  \label{eq.DF}.
\end{align}
Applying these} to  compute the Hopf pairing of $K,E,F$ with $\alpha_{\vec{\epsilon}\vec{\eta}}$ we get
$$K_{left}(\alpha(\vec{\eta}))=q^{2\sum \eta_i}\alpha(\vec{\eta}),$$ 
$$E_{left}(\alpha(\vec{\eta}))=\sum_{j=1}^d (\delta_{\eta_j,-}) q^{2\sum_{k=j+1}^d \eta_k}\alpha(\eta_1,\cdots ,\eta_{j-1}, +, \eta_{j+1},\cdots, \eta_d),$$
$$F_{left}(\alpha(\vec{\eta}))=\sum_{j=1}^d (\delta_{\eta_j,+}) q^{-2\sum_{k=1}^{j-1} \eta_k}\alpha(\eta_1,\cdots ,\eta_{j-1},-,\eta_{j+1},\cdots ,\eta_d).$$
Now the main claim is a direct computation using relation \eqref{eq.order}.
\end{proof}

Let  $\fS$ have indexed boundary $\partial \fS=\{e^L_1,\ldots e^L_m, e^R_1,\ldots e^R_n\}$ as explained in Subsection \ref{sub:comodule}. The Hopf duality gives $\SSQ$ an  algebra bimodule structure over $(\USL^{\otimes n},\USL^{\otimes m})$. (Notice the inversion between left and right when passing to modules). 

For each $\vec m\in \mathbb{N}^m$ and $\vec n\in \mathbb{N}^n$, let $B_{\vec{m},\vec{n}}(\fS;\ori)$ be defined as:
$$B_{\vec{m},\vec{n}}(\fS;\ori)=\left(\bigcap_{i=1}^m B_{e_i^L,m_i}(\fS;\ori)\right)\cap \left(\bigcap_{j=1}^n B_{e_j^R,n_j}(\fS;\ori)\right).$$
Let also, for each $\vec{j}\leq \vec{m}$ and $\vec{h}\leq \vec{n}$ (component-wise) and each $\alpha\in B_{\vec{m},\vec{n}}(\fS;\ori)$ let $\alpha_{\vec{j},\vec{h}}\in B(\fS;\ori)$ be the skein identical to $\alpha$ but for its state which is increasing and contains $\vec{j}$ (resp. $\vec{h}$) signs $-$ on the left (resp. right) edges.

\begin{theorem}\label{teo:module} Suppose that $\fS$ has indexed boundary $\partial \fS=\{e^L_1,\ldots e^L_n, e^R_1,\ldots e^R_m\}$.

(a) For each $\vec m\in \mathbb{N}^m$ and $\vec n\in \mathbb{N}^n$ and each $\al \in B_{\vec{m},\vec{n}}(\fS;\ori)$, 
 the $ (\USL^{\otimes n},\USL^{\otimes m})$-bimodule generated by $\alpha$ (namely $\USL^{\otimes n}\cdot \alpha \cdot \USL^{\otimes m}$) is irreducible and isomorphic to $V^L_{n_1}\otimes\cdots \otimes V^L_{n_n}\otimes V^R_{m_1}\otimes\cdots \otimes V^R_{m_m}$, where $V^L_k$ (resp. $V^R_k$) is the irreducible left (resp. right) module on $\USL$ with highest weight $k$. 

(b) As $ (\USL^{\otimes n},\USL^{\otimes m})$-bimodules, we have
\be 
\SSQ = \bigoplus \USL^{\otimes n}\cdot \alpha\cdot \USL^{\otimes m}.
\ee
where the sum is taken over all $\vec{m}\in \mathbb{N}^m,\vec{n}\in \mathbb{N}^n$, and all $\alpha\in B_{\vec{m},\vec{n}}(\fS;\ori)$.
In particular, the bimodule $\cSs(\fS)$ is a direct sum of finite dimensional bimodules over $ (\USL^{\otimes n},\USL^{\otimes m})$.

(c) Furthermore the bimodule structure restricts to that of a $ (U^L_{q^2}(\mathfrak{sl}_2)^{\otimes n},U^L_{q^2}(\mathfrak{sl}_2)^{\otimes m})$-bimodule (where $U^L_{q^2}(\mathfrak{sl}_2)$ is the integral version of $\USL$) and a decomposition similar to the above one holds:
\be 
\cSs(\fS) = \bigoplus \foo_{\vec{j}\leq \vec{m},\vec{h}\leq \vec{n}} \left( U^L_{q^2}(\mathfrak{sl}_2)^{\otimes n}\cdot \alpha_{\vec j,\vec{h}}\cdot U^L_{q^2}(\mathfrak{sl}_2)^{\otimes m}\right)
\ee
where the direct sum is taken over all $\vec{m}\in \mathbb{N}^m,\vec{n}\in \mathbb{N}^n$ and all $\alpha\in B_{\vec{m},\vec{n}}(\fS;\ori)$, and the $\foo_{\vec{j},\vec{h}}$ symbol stands for the non direct sum.
 
\end{theorem}

\begin{proof} 

(a) Fix $m_i, n_j$ and $\alpha$ as in the statement and let $JW(\alpha)$ be the skein obtained by inserting a $JW_{m_i}$ near $e^L_i$ and a $JW_{n_j}$ near $e^R_j$ for all $i,j$. By Example \ref{ex:joneswenzl} it is clear that $\al=JW(\al)$ and by Lemma \ref{lem:uqmodule} that it is a highest weight vector of weight $q^{2m_i}$ for the action of the $i^{th}$-copy of $\USL$ for each $i\leq m$; similarly it is a highest weight vector of weight $q^{2n_j}$ for the right action of the $j^{th}$-copy of $\USL$.  Furthermore, by Lemma \ref{lem:uqmodule} and the fact that the $m^{th}$-Jones Wenzl projector kills the self-returns, the orbit of $\al$ is exactly the span of the vectors $JW(\alpha_{\vec j,\vec{h}})$ with $\vec j\leq \vec m$ and $\vec h\leq \vec n$. 

(b)
It is straightforward from $a)$ and from Theorem \ref{thm.basis1a}.

(c) If a left (resp. right) $\USL$-module weight $M$ (over $\Qq$) has a basis formed by weight vectors over which the action of $E^{(r)}, F^{(r)},r\geq 1$ has coefficients in $\cR$, then $M$ restricts to a $U^L_{q^2}(\mathfrak{sl}_2)$-module; 
 we claim that  the basis $B(\fS;\ori)$ of $\cSs(\fS)$ has this property. Indeed since the structure of module is induced by the right (resp. left) comodule structure on each edge and the Hopf pairing between $\USL$ and $\OSL$, and since the comodule structure is integral in the basis $B(\fS;\ori)$, it is sufficient to observe that the Hopf pairing between $\USL$ and $\OSL=\cSs(\B)$ extends to a $\cR$-bilinear Hopf pairing between $U^L_{q^{2}}(\mathfrak{sl}_2)$ in the basis $B(\B,\ori)$: this is the content of point $(1)$ of Lemma \ref{lem:integralpairing}. 

To prove that the direct sum decomposition still holds,
let 
 $$B^{<}_{\vec{m},\vec{n}}(\fS;\ori):=\{\alpha\in B(\fS;\ori)|\# (\alpha\cap e^L_i)\leq m_i ,\# (\alpha\cap e^R_j)\leq n_j, \forall i\leq m, \forall j\leq n \}\setminus B_{\vec{m},\vec{n}}(\fS;\ori).$$ 
 
To prove the claim, we will show that for each $\alpha \in B_{\vec{m},\vec{n}}(\fS;\ori)$, the following holds:
$$\left(\cR\cdot B^{<}_{\vec{m},\vec{n}}(\fS;\ori)\right)\bigcap \left(\foo_{\vec{j}\leq \vec{m},\vec{h}\leq \vec{n}}\left( U^L_{q^2}(\mathfrak{sl}_2)^{\otimes n}\cdot \alpha_{\vec j,\vec{h}}\cdot U^L_{q^2}(\mathfrak{sl}_2)^{\otimes m}\right)\right)=\{0\}.$$

We start by remarking that if $\alpha\in B_{\vec{m},\vec{n}}(\fS;\ori)$, then $\alpha=JW(\alpha)$ and so, over $\Qq$, its orbit is a direct summand of $\cSs(\fS)$ and thus it has trivial intersection with the $\cR$-span of $B^{<}_{\vec{m},\vec{n}}(\fS;\ori)$: 
$$\left(\cR\cdot B^{<}_{\vec{m},\vec{n}}(\fS;\ori)\right)\bigcap \left( \USL^{\otimes n}\cdot \alpha\cdot \USL^{\otimes m}\right)=\{0\}.$$

Furthermore, by the point (b), given $\vec{j}\leq \vec{m},\vec{h}\leq \vec{n}$,  there exist $c(\vec{j},\vec{h})\in \cR\setminus 0$ and $\ell(\vec{j},\vec{h})\in \cR\cdot B^{<}_{\vec{m},\vec{n}}(\fS;\ori)$ such that $c(\vec{j},\vec{h}) \alpha_{\vec{j},\vec{h}}+\ell(\vec{j},\vec{h})$ is in the orbit of $\alpha$: 
$$c(\vec{j},\vec{h}) \alpha_{\vec{j},\vec{h}}+\ell(\vec{j},\vec{h})\in U^L_{q^2}(\mathfrak{sl}_2)^{\otimes n}\cdot \alpha\cdot U^L_{q^2}(\mathfrak{sl}_2)^{\otimes m}.$$

Now suppose that for some $l_i\in U^L_{q^2}(\mathfrak{sl}_2)^{\otimes m}$ and $r_i\in  U^L_{q^2}(\mathfrak{sl}_2)^{\otimes n}$ and some $\vec{j}_i,\vec{h}_i$ it holds $$\sum_{i} l_i\cdot \alpha_{\vec{j}_i,\vec{h}_i}\cdot r_i\in \cR\cdot B^{<}_{\vec{m},\vec{n}}(\fS;\ori)\setminus \{0\}.$$
Then, multiplying by $\prod_{i} c(\vec{j}_i,\vec{h}_i)$ (which gives a non-zero vector as $\cSs(\fS)$ is free as a $\cR$-module) we also get that 
$\sum_{i} l_i \cdot \alpha\cdot r_i \in \cR\cdot B^{<}_{\vec{m},\vec{n}}(\fS;\ori)\setminus \{0\}$,
which as already argued is impossible. 

\end{proof}

 \begin{example}\label{ex:bigonright}
Let $\B$ be the bigon whose edges $e_l$ and $e_r$ are declared to be respectively of type $L$ and $R$. Then by Theorem \ref{teo:SL2}, $\cSs(\B)$ is the right and left module $\USL$-module $\OSL$: the left action is induced by the right comodule structure coming from $e_r^R$ and the right action from $e_l^L$. 
If we let $\B^R$ be the bigon where both $e_l$ and $e_r$ are declared to be of type R (right), then $\cSs(\B^R)$ is a left $(\USL)^{\otimes 2}$-module; the action of $x\otimes y$ on a skein $b\in \cSs(\B^R)$ is given by 
$$(x\otimes y) \cdot b=x\cdot b\cdot r^*(y)$$ where $r^*(y)$ is the algebra antimorphism provided in Lemma \ref{lem:LRfunctors} and the left and right actions are those on $\cSs(\B)$ described above. 
\end{example}

\subsection{Co-tensor product} Suppose  $U$ is a coalgebra over a ground ring $\cR$. Assume 
$M$ is a left $U$-comodule with coaction  $\Delta_M:M\to U\otimes_{\cR} M$, and $N$ a right $U$-comodule with coaction $\Delta_N:N\to N\otimes_R U$. 
Then the cotensor product $N \Box _U M $  is
$$N\Box_U M:=\{v\in N\otimes M \mid (\Delta_N\otimes \id_M)(v)=(\id_N\otimes \Delta_M)(v)\}.$$

Cotensor product is a special case of the following notion of Hochshild cohomology. 
Assume $V$ is a $\cR$-module with a left $U$-coaction and a right $U$-coaction:
$$ \D_r: V \to V \ot U, \quad _l\Delta: V \to U \ot V .$$
The 0-th Hochshild cohomology of $V$ is defined by
$$ HH^0(V)= \{ x\in V \mid \D_r(x) =  \fl(\null_l\D(x)),$$
where $\fl: V \ot U\to U \ot V$ is the flip $\fl(x\ot y) = y \ot x$.

With $M$ and $N$ as above, define a left $U$-coaction and a right $U$ coaction on $N \ot_R M$ by
\begin{align*}
\D_r &: N \ot_R M \to N \ot_R M \ot_R U, \quad \D_r(n \otimes m) = \sum n' \ot m \ot u' \ \text{if}\  \D_N(n) = \sum n' \ot u' \\
\null_l\D &: N \ot_R M \to  U \ot_R N \ot_R M, \quad \null_l\D (n \otimes m) = \sum u'' \ot n \ot m''  \ \text{if}\  \D_M(m) = \sum u'' \ot m''.
\end{align*}

Then the cotensor product $N\Box_U M= HH^0(N\otimes_{\cR} M)$.

\subsection{Splitting as co-tensor product and Hochshild cohomology}

\def\cD{\null _c\D}
\def\Dc{\D_c}
\def\lD{\null _l\D}
\def\Dr{\D_r}

Suppose $c_1, c_2$ are distinct boundary edges of a punctured bordered surface $\fS'$ and $\fS= \fS'/(c_1=c_2)$, with $c\subset \fS$ being the common image of $c_1$ and  $c_2$.  The splitting homomorphism gives an embedding 
$\theta_c: \cSs(\fS) \embed \cSs(\fS') $ and  we will make precise the image of $\theta_c$.

 \FIGc{split2}{(a) The middle shaded part is the bigon, while the left and the right shaded parts are part of $\fS'$. Gluing $c_1=e_l$ gives the right coaction $\Dr$ and gluing $\er=c_2$ gives the left coaction $\lD$.
 (b) Element $\null  x _{\bnu\bmu}\in \cSs(\fS')$. The horizontal lines are part of $x$. Note the order of indices.}{2.6cm}

\begin{theorem}\label{teo:cotensor} 
Suppose $c_1, c_2$ are distinct boundary edges of a punctured bordered surface $\fS'$ and $\fS= \fS'/(c_1=c_2)$. The splitting homomorphism  
\be  \theta_c: \cSs(\fS) \embed \cSs(\fS') \label{eq.split6}.
\ee
maps $\SS$ isomorphically onto the Hochshild cohomology $HH^0(\cS(\fS'))$, which is a $\cS(\cB)$-bimodule via the left coaction $\lD:= \null_{c_2}\D$ and the right coaction $\Dr:= \D_{c_1}$ (see Figure \ref{fig:split2}(a)).

In particular, if $c_1$ is a boundary edge of $\fS'_1$ and $c_2$ is a boundary edge of $\fS'_2$ which is disjoint from $\fS'_1$, and $\fS=(\fS'_1 \sqcup \fS'_2)/(c_1=c_2)$,  then $\theta_c$ maps $\cSs(\fS)$ isomorphically onto
the cotensor product of $\cSs(\fS'_1)$ and $\cSs(\fS'_2)$ over $\cSs(\fB)$.
\end{theorem}

\def\coef{\mathrm{coef}}
\def\btau{{\vec \tau}}
\begin{proof} Let us identify $\SS$ with its image under $\theta_c$.
 From the splitting formula \eqref{eq.slit} it is easy to see that $\cSs(\fS)\subset HH^0(\cSs(\fS'))$. Let us prove the converse inclusion.
Assume  $0\neq v \in HH^0(\cSs(\fS'))$.  By definition, this means
\be 
\Dr(v) - \fl(\lD(v))=0.
\ee

Choose an orientation $\ori$ of $\pfS$ and an orientation of $c$. Then $\ori$ and the orientation of $c_1$ and $c_2$ induced from $c$ give an orientation $\ori'$ of $\pfS'$. 
Recall that  $B(\fS';\ori')$ is a free $\cR$-basis of $\cSs(\fS')$. Let $\tilde B(\fS';\ori')$ be the set of all
isotopy classes $x$ of $\ori'$-ordered $\pM'$-tangle diagrams which are increasingly stated on every boundary edge except for $c_1$ and $c_2$. 
If $\bmu$ is a state of $x \cap c_1$ and $\bnu$ is a state of $x\cap c_2$, let $x _{\bnu\bmu}$ be the  stated $\ori'$-ordered $\pM'$-tangle
 diagram whose states on $x \cap c_1$ and $x \cap c_2$ are respectively $\bmu$ and $\bnu$. See Figure \ref{fig:split2}(b). 
If $\bmu$ and $\bnu$ are increasing, then $x _{\bnu\bmu} \in B(\fS', \ori')$ is a basis element. For each $i=1,2$ let $S_{x\cap c_i}$ and $S_{x \cap c_i}^\uparrow$ be respectively the set of all states and the set of all increasing states of $x \cap c_i$. Then
$$ B(\fS',\ori') = \{ x_{\bnu\bmu} \mid x \in \tilde B(\fS', \ori'), \bmu \in S_{x \cap c_1}^\uparrow, \bnu \in S_{x \cap c_2}^\uparrow \}.$$
Using the above  $\cR$-basis $B(\fS';\ori')$ of $\cSs(\fS')$, we can present $v\in  \cSs(\fS')$ in the form
\be v=\sum_{x\in X(v)}\,  \sum_{  \bmu \in S_{x \cap c_1}^\uparrow }\,  \sum_{  \bnu \in S_{x \cap c_2}^\uparrow } \, \coef(v, x_{\bnu\bmu})\, \, x_{\bnu\bmu},
\label{eq.v}
\ee
where $X\subset \tilde B(\fS';\ori')$ is a minimal finite set, so that for each $x\in X$, there are $\bmu, \bnu$ such that  the coefficient  $\coef(v, x_{\bnu\bmu})$ is non-zero.

Let $m(v)=\max\{ |x \cap c_1|, |x\cap c_2|,  x\in X(v)\}$. We show by induction on $m(v)$ that $v \in \cSs(\fS)$.

For $i=1,2$ let $X_i(v)= \{ x\in X(v) , |x \cap  c_i| = m\}$.   If $x\in X_1(v)\cap X_2(v)$, then $|x\cap c_1|=|x \cap c_2|= m(v)$, and there is an element $\bar x\in \B(\fS,\ori)$ such that $x$ has coefficient non-zero in the result of splitting $\bar x$ along $c$. From the definition of the splitting map we have
\be 
\coef(\theta(\bar x), x_{\btau\btau})= 1, \label{eq.max}
\ee
where $\btau= (+)^m$ is the state consisting of $m$ plus signs. Let 
$$v' = v - \sum _{x\in X_1(v) \cap X_2(v)} \coef(v, x_{\btau \btau} ) \, \theta(\bar x).$$ 
If $m(v') < m(v)$ then  we are done by induction. Assume that $m(v')=m(v)$. One of $X_1(v'), X_2(v')$ is not empty, and
without loss of generality assume $X_2(v') \neq \emptyset$. Formula \eqref{eq.v} for $v'$ has the form
\be 
v'=\sum_{x\in X(v')}\,  \sum_{  \bmu \in S_{x \cap c_1}^\uparrow }\,  \sum_{  \bnu \in S_{x \cap c_2}^\uparrow } \, \coef(v', x_{\bnu\bmu})\, \, x_{\bnu\bmu},
\label{eq.vp}
\ee
and because of \eqref{eq.max} we can assume that there is no $x_{\tau\tau}$ on the right hand side of \eqref{eq.vp}.

Let $p^{c_2}_m:  \cSs(\fS') \to \cSs(\fS')$ be
 the projection onto the homogeneous part $G^{c_2}_m(\cSs(\fS'))$,  and $p^{\er}_m: \cSs(\fB) \to \cSs(\fB)$ be the projection onto the homogeneous part $G^{\er}_m(\cSs(\fB))$, see the end of Section \ref{sec.filtration}. Explicitly, for $x\in X(v')$ we have
\be
p^{c_2}_m (x_{\bnu\bmu})= \begin{cases} 0 \quad &\text{if} \ \bnu \neq \btau \\
x_{\btau \bmu} & \text{if} \ \bnu = \tau,
\end{cases} \quad p^{\er}_m (\al_{\bnu\bmu})= \begin{cases} 0 \quad &\text{if} \ \bmu \neq \tau \\
\al_{\bnu\btau}  & \text{if} \ \bmu = \btau.
\end{cases}
\ee

From Formula \eqref{eq.coact} for the coaction, we have, for $x\in X(v')$ and $(\bnu,\bmu) \neq (\btau,\btau)$, 
\begin{align}
\label{eq.sp1}x _{\bnu\bmu} & \overset{\Dr}{\longrightarrow}  \sum _{\boeta\in  S_{x \cap c_1}}  x _{\bnu\boeta}  \ot  \al _{\boeta\bmu}  \overset{p^{c_2}_m \ot p^{\er}_m}{\longrightarrow} 0 \\ 
\label{eq.sp2}x _{\bnu\bmu} & \overset{\lD}{\longrightarrow} \sum _{\boeta\in  S_{x \cap c_2}}  \al _{\bnu\boeta} \ot   x _{\boeta\bmu} \overset{\fl}{\longrightarrow} \sum _\boeta   x _{\boeta\bmu}  \ot  \al _{\bnu\boeta} \overset{p^{c_2}_m \ot p^{\er}_m}{\longrightarrow} \begin{cases}   0 \ &\text{ if } \ x \not \in X_2(v') \\
x_{\btau \bmu} \ot \al _{\bnu\btau}  \ &\text{ if } \ x \in X_2(v'). \end{cases}
\end{align}
It follows that
$$ 0= (p^{c_2}_m \ot p^{\er}_m )(\fl (\lD (v')) - \Dr(v'))=  \sum_{x\in X_2(v')} \sum_{  \bmu \in S_{x \cap c_1}^\uparrow }\,  \sum_{  \bnu \in S_{x \cap c_2}^\uparrow } \, \coef(v', x_{\bnu\bmu})\,  \, x_{\btau \bmu} \ot \al _{\bnu\btau}.$$
As the right hand side is a linear combination of elements of a basis, all the coefficients 
$v$ there are 0. This means $X_2(v')=\emptyset$, a contradiction. Thus $m(v') < m(v)$ and we are done.
\end{proof}

\begin{remark}\label{rem:cotensor}
Theorem \ref{teo:cotensor} holds also if we change the base ring to $\mathbb{C}$ by evaluating $q$ to a non-zero complex number. 
\end{remark}
\def\SSQp{\cS(\fS') \ot_\cR \Qq}
Using the above result together with Theorem \ref{teo:module} we can deduce a similar result for $\USL$-modules.
Recall that for a bi-module $V$ over a $\Qq$-algebra $U$ the 0-homology group is
$$ HH_0(V) = V/\Qq{\text - span}\la  a \cdot v - v \cdot a \mid a \in U, v \in V\ra.$$

\begin{theorem}\label{r.Hoch2}
Suppose $c_1, c_2$ are distinct boundary edges of a punctured bordered surface $\fS'$ and $\fS= \fS'/(c_1=c_2)$, with $c$ being the common image of $c_1$ and $c_2$. Then the composition
\be  \SSQ \overset{\theta_c} \longrightarrow  \cS(\fS') \ot_R \Qq \to HH_0(\cS(\fS') \ot_R \Qq)
\label{eq.iso6}
\ee
is an isomorphism of $\Qq$-vector spaces. Here $\SSQp$ is a $\USL$-bimodule via the dual actions of $\Delta_{c_1}$ and $\eDD{c_2}$.

In particular, if $\fS'= \fS_1 \sqcup \fS_2$ with $c_1 \subset \fS_1$ and $c_2 \subset \fS_2$, then the map in \eqref{eq.iso6} is an isomorphism between  $\SSQ$ and  $(\cS(\fS_1) \ot_R \Qq) \ot_{\USL} (\cS(\fS_2) \ot_R \Qq)$. 
\end{theorem}

\begin{proof}{

The decomposition of $\cS(\fS') \ot_R \Qq$ given by part (b) of Theorem \ref{teo:module} shows that it is sufficient to prove that if $V^R_{m_2}$ (resp. $V^L_{m_1}$) is the irreducible $m_2+1$-dimensional (resp. $m_1+1$-dimensional) right (resp. left) $\USL$-module, then the composition of natural map
\be  HH^0(V^L_{m_1}\otimes V^R_{m_2}) \hookrightarrow V^L_{m_1}\otimes V^R_{m_2}\onto 
 HH_0( V^L_{m_1}\otimes V^R_{m_2})
\ee
is an isomorphism of vector spaces. We will see that this follows from the fact that every finite-dimensional $\USL$-module $V$ is equivalent to its dual $V^*$.

First observe that since the pairing between $\OSL$ and $\USL$ is non degenerate, $HH^0(V^L_{m_1}\otimes V^R_{m_2})$ can be equivalently defined as 
$$HH^0(V^L_{m_1}\otimes V^R_{m_2})=\{v\otimes w\in V^L_{m_1}\otimes V^R_{m_2}| x\cdot v\otimes w=v\otimes w\cdot x,\ \forall x\in \USL\}.$$
Then using the isomorphism between $V^R_{m_2}$ and $(V^L_{m_2})^*$ we have :
$$HH^0(V^L_{m_1}\otimes V^R_{m_2})=Hom_{\USL}(V^L_{m_2},V^L_{m_1})=\delta_{m_1,m_2} \Qq$$
by Schur's lemma.

Now let's take the dual of the above equation and get: 
\be  (HH_0( V^L_{m_1}\otimes V^R_{m_2}))^*\hookrightarrow 
  (V^R_{m_2})^*\otimes (V^L_{m_1})^*\onto (HH^0(V^L_{m_1}\otimes V^R_{m_2}))^*
 \ee
 where the first arrow maps an element of $(HH_0( V^L_{m_1}\otimes V^R_{m_2}))^*$ to some $f\in (V^L_{m_1}\otimes V^R_{m_2})^*$ such that $f(x \cdot v\otimes w)=f(v\otimes w\cdot x)$ for all $x\in \USL$ and $v\otimes w\in V^L_{m_1}\otimes V^R_{m_2}$.  Using again the isomorphism between $(V^R_{m_2})^*$ and $V^L_{m_2}$ we have that the image of $(HH_0( V^L_{m_1}\otimes V^R_{m_2}))^*$ in $V^L_{m_2}\otimes (V^L_{m_1})^*$ is $Hom_{\USL}(V^L_{m_1},V^L_{m_2})=\delta_{m_1,m_2}\Qq$ by Schur's lemma. 
 
 To conclude, observe that if $m_1=m_2$ then the image of the inclusion $HH^0(V^L_{m_1}\otimes V^R_{m_1}) \hookrightarrow V^L_{m_1}\otimes V^R_{m_1}\simeq V^L_{m_1}\otimes (V^L_{m_1})^*=Hom(V^L_{m_1},V^L_{m_1})$ is given by the multiples of the identity map. But the kernel of the projection $V^L_{m_1}\otimes V^R_{m_1}\onto HH_0(V^L_{m_1}\otimes V^R_{m_1})$ is the sub vector space of $Hom(V^L_{m_1},V^L_{m_1})$ spanned by the matrices of the form $xM-Mx$ where $x$ represents the action of an element of $\USL$ and $M\in Hom(V^L_{m_1},V^L_{m_1})$; thus it is contained the set of matrices with zero trace and so the projection of $HH^0(V^L_{m_1}\otimes V^R_{m_1})$ in $HH_0(V^L_{m_1}\otimes V^R_{m_1})$ is nonzero.

}   
\end{proof}
\begin{remark}
By the splitting theorem and Proposition \ref{r.Hoch2}, $\cSs(\fS)$ is both a submodule and a quotient module of $\cSs(\fS')$.
\end{remark}

\begin{example}\label{ex:bigonR}
Clearly, if in Theorem \ref{r.Hoch2} $c_1=e_i^R$ and $c_2=e_j^L$ belong to two distinct connected components of $\fS'$, then one can restate the $HH_0$ simply as a tensor product over a copy of $\USL$ acting on the left on the skein algebra of one component and on the right on the other.

In particular, if $\fS$ is obtained by glueing a bigon $\B$ along its right edge to a left edge of $\fS'$ then $\cSs(\fS)=\cSs(\fS')\otimes_{\USL} \cSs(\B)$ is isomorphic to $\cSs(\fS')$ as it can be seen  directly by Theorem \ref{thm.1a}.
 
If $\B^R$ is the bigon whose edges are both declared to be of type $R$ (right), then $\cSs(\B^R)$ is a left module over $\USL^{\otimes 2}$ (see Example \ref{ex:bigonright}). Then glueing $\B^R$ to $\fS'$ along one edge of type $L$, shows that 
$$\cSs(\fS)=\cSs(\fS')\otimes _{\USL} \cSs(\B^R).$$
The resulting surface $\fS$ is still homeomorphic to $\fS'$ but the edge on which the glueing has been performed has been transformed from an edge of type $L$ to one of type $R$. This corresponds to applying Lemma \ref{lem:LRfunctors} to the module structure coming from that edge. 
\end{example}
\begin{remark}\label{rem:glueingcommutes}
If $c'_1,c'_2$ are two other edges of $\partial \fS'$ (and then of $\partial \fS$), \eqref{eq.iso6} is an isomorphism of $\USL$-modules for the structure associated to $c'_1$ and $c'_2$. Furthermore the theorem can be applied independently to glue also $c'_1$ and $c'_2$ and the final isomorphism between $\cSs(\fS'/(c_1=c_2,c'_1=c'_2))\otimes \Qq$ and  $HH^0(\cSs(\fS')\otimes \Qq)$ (with respect to the $\USL^{\otimes 2}$-bimodule structure) does not depend on the order in which the glueing was performed.
\end{remark}

\def\uot{\underline{\ot}}
\def\oneD{{ _1\Delta }}
\def\CSS{ \cSs(\fS)}
\def\ufS{\underline{\fS}} 
\def\uast{\underline \ast}
\subsection{Braided tensor product}

Let $U$ be a dual quasitriangular Hopf algebra. 
Assume $A$ is an algebra admitting two right $U$-comodule-algebra structures $\Delta_1:A\to A\otimes U$ and $\Delta_2:A\to A\otimes U$ which commute, i.e.
\be (\Delta_1\otimes \Id_U)\circ \Delta_2 =(\Id_A \otimes \fl )\circ(\Delta_2\otimes \Id_U)\circ \Delta_1 
\label{eq.D12}
\ee
where $\fl :U\otimes U\to U\otimes U$ is the flip operator. Denote the common operator of \eqref{eq.D12} by $\Delta_{12}$.

Observe that since $\Delta_1$ and $\Delta_2$ commute,  $A$ can be endowed with a right $U$-comodule structure $\bD:=\Delta_1 \uast \Delta_2: A \to A \ot U$ defined by
\be 
\bD(x) = \sum x' \ot u_1 u_2 \ \text{if} \ \Delta_{12}(x) = \sum x'\ot u_1 \ot u_2.
\ee

However $\bD: A \to A \ot U$ is not an algebra homomorphism, i.e. $A$ is not a right $U$-comodule algebra with respect to $\bD$, even though it is a right $U$-comodule algebra with respect to each of  $\Delta_1$ and $\Delta_2$. So we define a new product. For $x,y\in A$ let
\be \label{eq:selfproduct}
 x\uast y = \sum x' y' \rho(u \ot v) \ \text{if} \ 
\Delta_2 (x) =\sum x' \ot u,  
\Delta_1(y) =\sum y' \ot v.
\ee
 It is easy to check that $\uast$ gives $A$ a new associative product, and we call $A$ with this new product  the \emph{self braided product} of $\Delta_1$ and $\Delta_2$ and denote it $\uot A$.

\begin{lemma} 
With respect to the product $\uast$ and the right $U$-comodule given by $\bD=\Delta_1 \uast \Delta_2$, the algebra $A$ is a right $U$-comodule-algebra. 
\end{lemma}
\begin{proof} We have to show that for $x,y\in A$ one has 
\be \bD (x\uast y)= \bD(x)\uast\,  \bD (y).  \label{eq.uw1}
\ee
Let us now  write $\Delta_{12}(x)$ and $\Delta_{12}(y)$ as
\begin{align*}
\Delta_{12}(x)= \sum x' \ot u_1 \ot u \in A \ot U \ot U \\
\Delta_{12}(y)= \sum y' \ot v  \ot v_2 \in A \ot U \ot U.
\end{align*}
 Using  the commutativity of $\Delta_1$ and $\Delta_2$ and a simple calculation, we obtain
\begin{align}
 \bD(x\uast y) &= \sum \rho(u''\ot v'') \, x' y'  \ot u_1   v'u'    v_2      \label{eq.eq1}\\
 \bD(x)\uast\,  \bD (y)  &= \sum  \rho(u'\ot v') \, x' y'  \ot u_1   u'' v''     v_2   \label{eq.eq2} ,
\end{align}
where $\Delta(u)= \sum u' \ot u'', \Delta(v) = v' \ot v''$ are coproducts in $U$.
The right hand sides of \eqref{eq.eq1} and \eqref{eq.eq2} are equal thanks to \eqref{eq.rflip}. This proves \eqref{eq.uw1}. For those who are familiar with graphical calculations in Hopf algebras,
we provide a graphical proof in Figure \ref{fig:selftensorproductcompatible}.
\end{proof}
\begin{figure}[htpb]
\raisebox{1cm}{$\bD(\cdot \uast \cdot)=$}
\includegraphics[width=1.9cm]{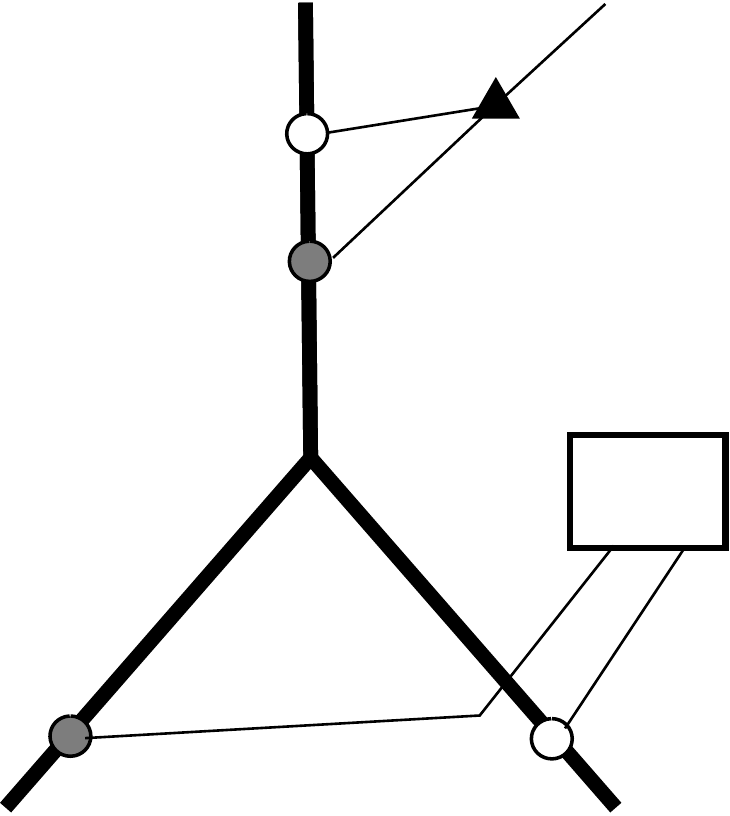}\raisebox{1cm}{\ =}\includegraphics[width=1.9cm]{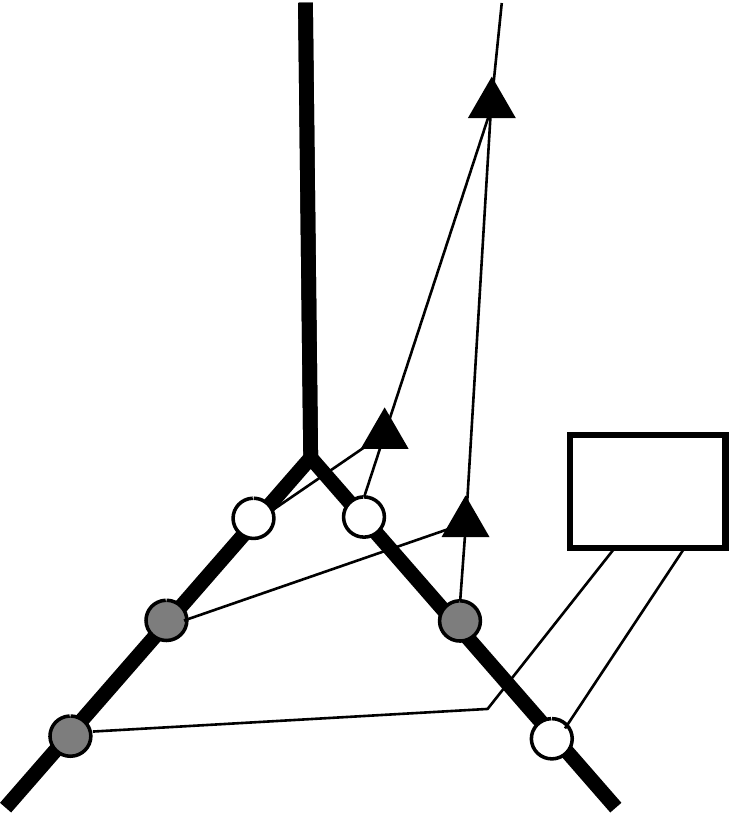}\raisebox{1cm}{\ =}\includegraphics[width=1.9cm]{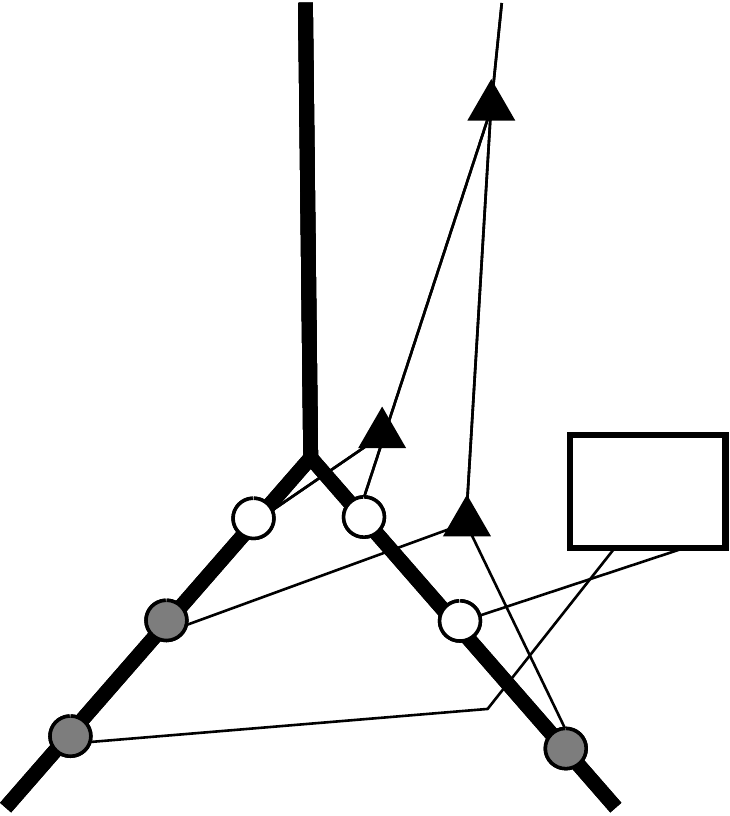}\raisebox{1cm}{\ =}\includegraphics[width=2.08cm]{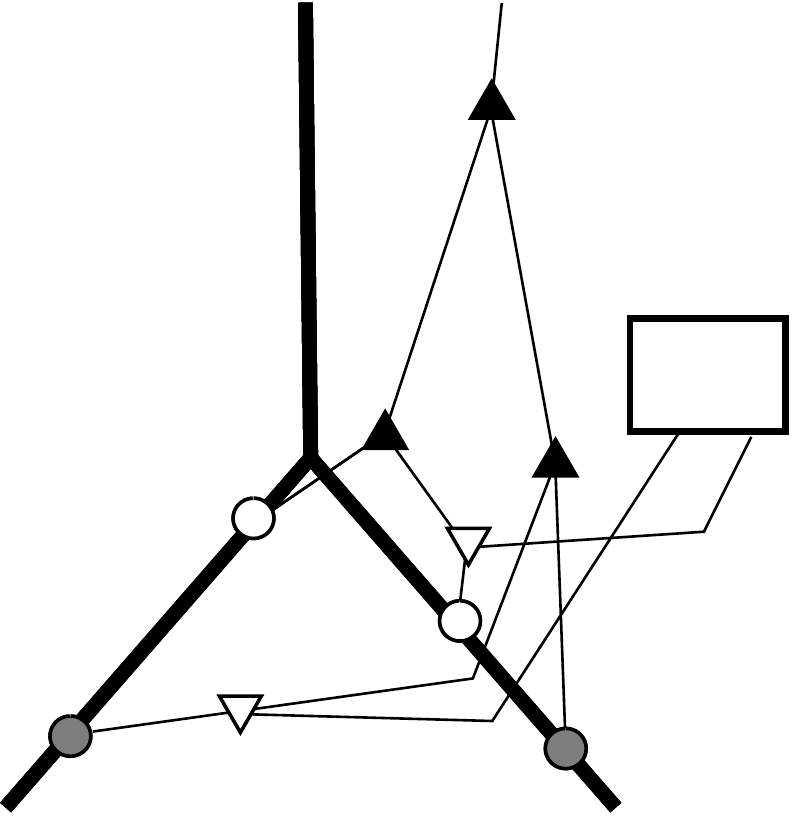}\raisebox{1cm}{\ =}\includegraphics[width=2cm]{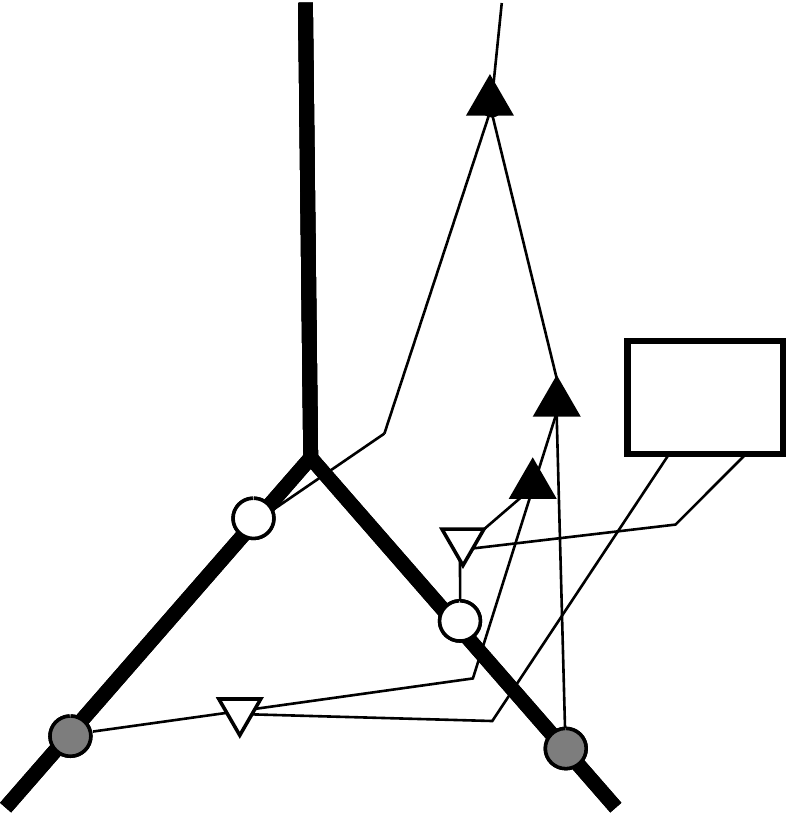}\raisebox{1cm}{=}\\\

\raisebox{1cm}{=}\includegraphics[width=1.7cm]{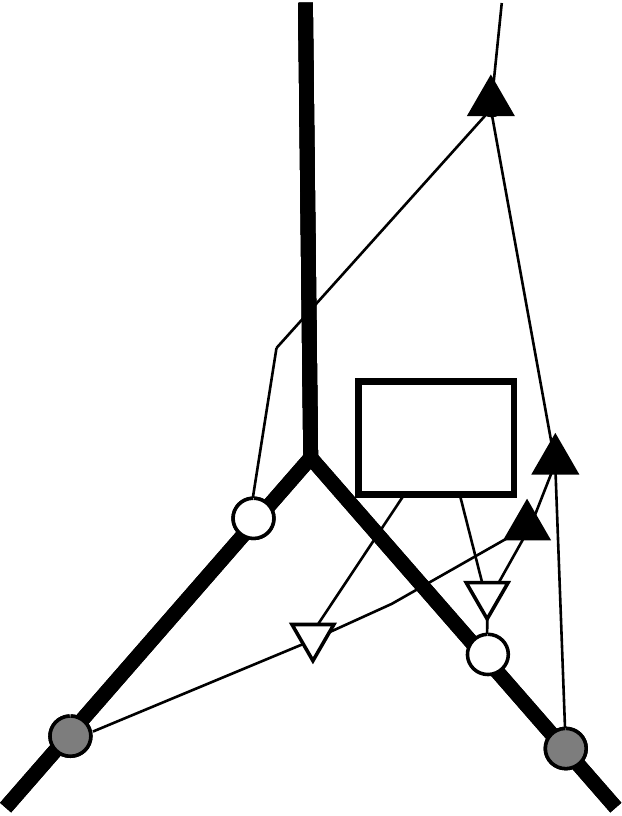}\raisebox{1cm}{\ =}\includegraphics[width=1.7cm]{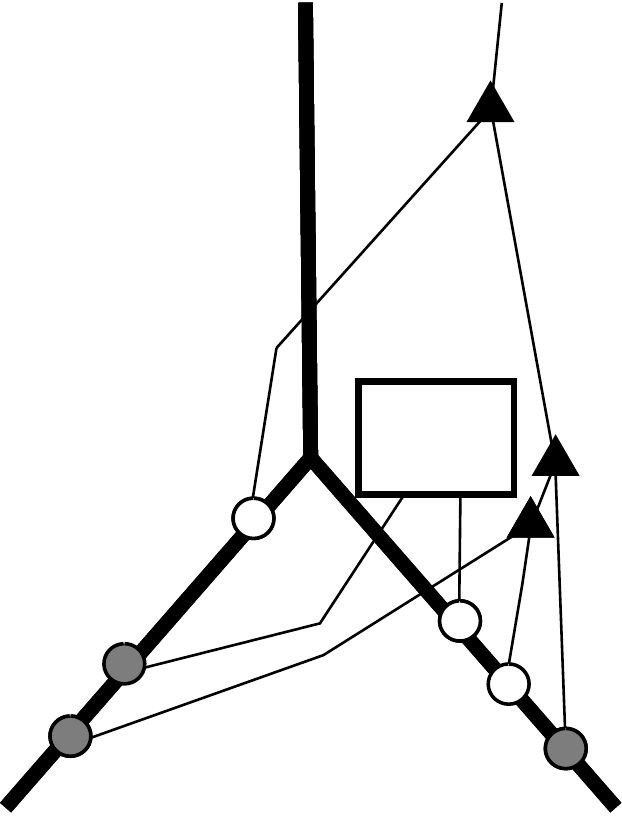}\raisebox{1cm}{\ =}\includegraphics[width=1.7cm]{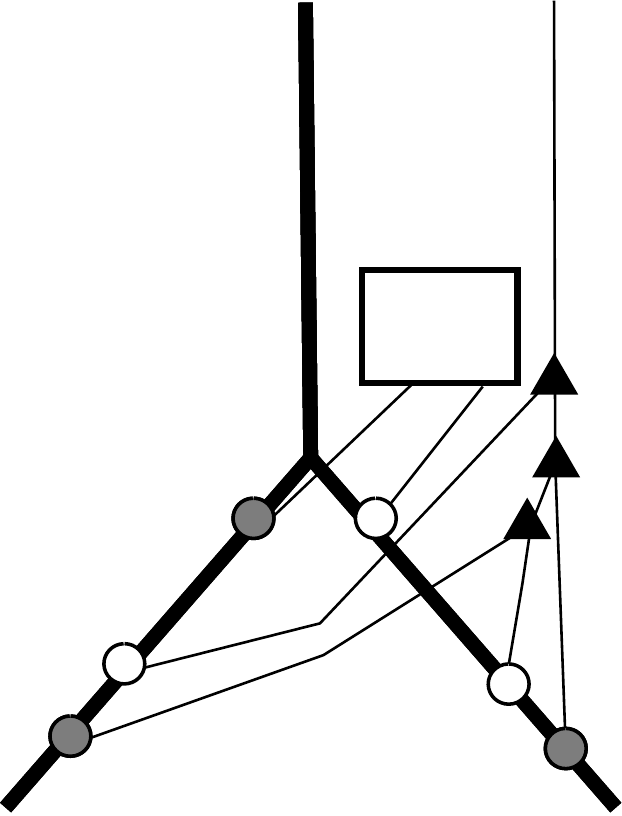}\raisebox{1cm}{\ =}\includegraphics[width=2cm]{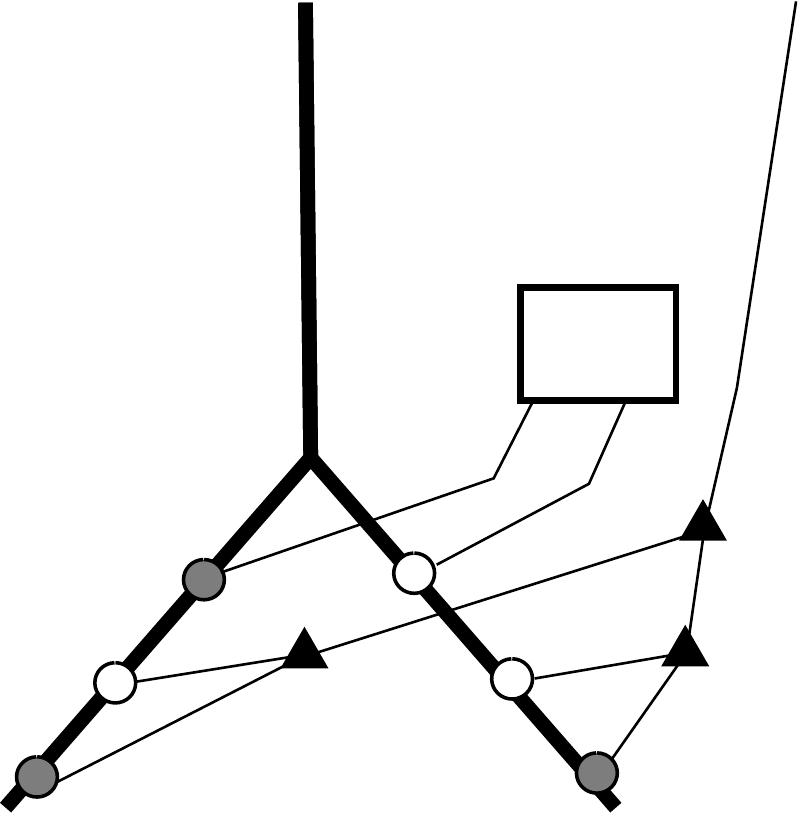}\raisebox{1cm}{$=\bD(\cdot) *\bD(\cdot)$}
\caption{The proof of the compatibility of $\bD$ and $\uast$. The diagrams are to be read from bottom to top, the thick (resp. thin) strands are $A$-colored (resp. $U$-colored), the crossings are flips, the white (resp. gray) solid dots represent $\Delta_1$ (resp. $\Delta_2$), the rectangle represents the co-$R$-matrix $\rho$, the black (resp. white) triangle is the product (resp. coproduct) of $U$, the thick trivalent vertex is the initial product in $A$. The equalities follow in order from: compatibility of $\Delta_i$ with the product of $A$, $\Delta_1$ commutes with of $\Delta_2$, coassociativity of $\Delta_i$, associativity of the product in $U$, equation \eqref{eq.rflip},  coassociativity of $\Delta_i$, $\Delta_1$ commutes with $\Delta_2$, associativity of the product in $U$.
}\label{fig:selftensorproductcompatible}
\end{figure}

\begin{example}[Braided tensor product]
The first  example of the above structure is the well-known {\em braided tensor product of two right comodule-algebras}, which we describe in details for the reader's convenience.
Suppose $A_1, A_2 $ are right comodule-algebras over a dual quasitriangular  Hopf algebra $U$. The tensor product 
$A=A_1 \ot_{\cR} A_2$ has two commuting right $U$-comodule structures. Namely  $\Delta_1=(\Id_{A_1}\otimes \fl )\circ (\Delta_{A_1}\otimes \Id_{A_2})$ and $\Delta_2=\Id_{A_1 }\otimes \Delta_{A_2}$.

By the above construction, $A$ has  the structure of a right $U$-comodule algebra, with the coaction $\bD= \Delta_1 \uast \Delta_2$ and the product $\uast$. Explicitly, for $x\in A_1$ and $y\in A_2$, the coaction is
$$\bD (x\ot y)= \sum (x' \ot y') \ot uv \ \text{if} \ \Delta_{A_1}(x)= \sum x' \ot u, \Delta_{A_2}(y)= \sum y' \ot v. $$

Let us now describe the product $\uast$. 
Identify $A_1$ with $A_1 \ot \{1\}$ and $A_2$ with $\{1\} \ot A_2$ (as subsets of $A_1\ot_{\cR}A_2)$.  Then the new product \eqref{eq:selfproduct} is given by
\be  \label{eq.brprod}
x\uast y = \begin{cases} xy \qquad &\text{if } \ x, y \in A_1 \ \text{or } \ x, y \in A_2 \\
x \ot y &\text{if } \ x \in A_1 , y \in A_2 \\
\sum \rho(u\ot v)(y' \ot  x')
&
\text{if } \ x \in A_2, y \in A_1, 
\end{cases} 
\ee
where $\Delta_{A_2}(x)= \sum x'\ot u,\Delta_{A_1}(y)= \sum y'\ot$,   and $\rho$ is the co-$R$-matrix.

The algebra $A_1 \ot_{\cR}A_2$ with this new product is called the {\em braided tensor product} of $A_1$ and $A_2$, and is denoted by $A_1 \uot_U A_2$.
For details see \cite{MajidFoundation}.

\end{example}

\begin{example}[Transmutation]  \label{ex.trans}
Assume that $U$ is a dual quasitriangular Hopf algebra and let $A=U$.  Then $\Delta_2:= \Delta: A \to A\ot U$ gives $A$ a right $U$-comodule algebra structure. To get another right $U$-comodule structure one converts the standard left comodule structure to a right one by
\be \Delta_1(x) := \sum x' \ot S(u) \ \text{if} \ \Delta(x) = \sum u \ot x'.
\ee

However $\Delta_1$ is not compatible with the algebra structure of $A$. One can twist the product of $A$ using the co-$R$-matrix $\rho$ of $U$ to make both $\Delta_1$ and $\Delta_2$ right comodule algebras as follows.
 Define a new product on $A$ using the common value of \eqref{eq.rflip}, i.e.
 \be  x \odot y = \sum \rho (x'\ot y') x'' y'',\ \text {if} \ \Delta(x) =\sum x' \ot x'', \Delta(y) = \sum y' \ot y''.
 \label{eq.newprod}
 \ee
It is easy to check that this gives $A$ a new product, with which both $\Delta_1$ and $\Delta_2$ give $A$  right $U$-comodule algebra structures. Besides, $\Delta_1$ and $\Delta_2$ commute.

Our construction now gives $A$ a right $U$-comodule algebra structure whose coaction $\bD=\Delta_1\uast \Delta_2$ and whose product $\uast$ are given by
\begin{align}
\bD(x) &= \sum x'' \ot S(x') x''' \ \text{if}\  (\Delta\ot \Id_U)\circ \Delta(x) = \sum x'\ot x'' \ot x'''
\label{eq.coact2}\\
x \uast y&= \sum   x'' y'' \rho( S(x' x''') \ot S( y')  )\  \text{if}\  \Delta(y) = \sum y'\ot y''.\label{eq.transmprod}
\end{align} 
It turns out that the coaction \eqref{eq.coact2} is exactly the {\em right coadjoint action}, see \cite[Example 1.6.14]{MajidFoundation} and the product $\uast$ is 
 exactly the {\em covariantized product} of \cite[Example 1.6.14]{MajidFoundation}. The algebra $A$, with this 
new product $\uast$ and the original coproduct $\Delta$,  is known as the {\em transmutation} of $A$, and is a {\em braided group} in the braided category of $U$-comodules, see~\cite{MajidFoundation}.

\end{example}

\subsection{Attaching an ideal triangle is a braided tensor product} \label{sec:braidedtensor}
Suppose $e, e_1, e_2$ are oriented edges of an ideal triangle $\P_3$  as depicted in Figure \ref{fig:brtensor0}.
\FIGc{brtensor0}{Left: Ideal triangle $\P_3$. Middle: Glueing $\fS$ and $\cP_3$ by $a_1=e_1$ and $a_2=e_2$ to get $\ufS$. Right: tangle diagram $x\in \cSs(\fS)$ and its image $f(x)\in \cSs(\ufS)$}{2.8cm}

Let $\fS$ be a (possibly disconnected) punctured bordered surface, with two boundary edges $a_1,a_2 \subset \partial\fS$. Define $\underline{\fS} = (\fS \sqcup \P_3)/(e_1=a_1, e_2= a_2)$, see Figure \ref{fig:brtensor0}. (A special case is when $\fS=\fS_1\sqcup \fS_2$ and $a_1\subset \fS',a_2\subset \fS''$.)
For $i=1,2$ the algebra $\CSS$ has a right comodule algebra structure $\Delta_i:= \Delta_{a_i}: \cSs(\fS) \to \CSS\ot \OSL$. 
The two coactions $\Delta_1$ and $\Delta_2$ commute, see Lemma \ref{prop:comodule}. Hence we can define the self braided tensor product $\uot \cSs(\fS)$ of $\Delta_1$ and $\Delta_2$, which gives $\CSS$ a new right comodule algebra structure over $\OSL$. On the other hand, $\Delta_e: \cSs(\underline{\fS}) \to \cSs(\underline{\fS}) \ot \OSL$ gives $\cSs(\underline{\fS})$ a right comodule algebra structure over $\OSL$.

\begin{theorem}\label{teo:braidedtensor}
The right $\OSL$-comodule algebra $\cSs(\underline{\fS})$ is naturally isomorphic to the self braided tensor product $\uot \cSs(\fS)$, defined with the co-$R$-matrix $\rho'$ of \eqref{eq.coR2}.

In particular when $\fS=\fS_1\sqcup \fS_2$ and $a_1\subset  \fS_1,a_2\subset \fS_2$, the right comodule algebra $\cSs(\underline{\fS})$ over $\OSL$ is naturally isomorphic to the braided tensor product $\cSs(\fS_1) \uot \cSs(\fS_2)$.
\end{theorem}

\def\brtone{  \raisebox{-13pt}{\incl{1.3 cm}{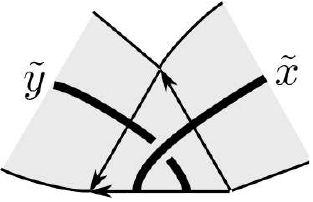}} }
\def\brttwo{  \raisebox{-15pt}{\incl{1.5 cm}{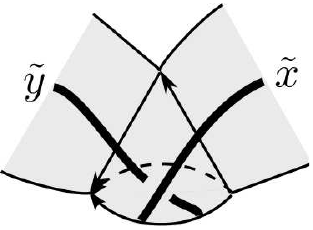}} }
\def\brtthree{  \raisebox{-13pt}{\incl{1.3 cm}{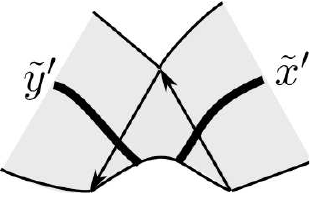}} }
\def\brtfour{  \raisebox{-8pt}{\incl{.8 cm}{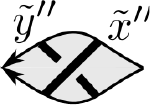}} }
\def\brtonenew{  \raisebox{-13pt}{\incl{1.3 cm}{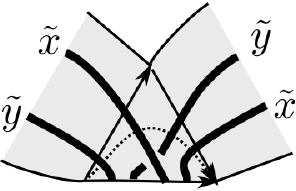}} }
\def\brttwonew{  \raisebox{-15pt}{\incl{1.5 cm}{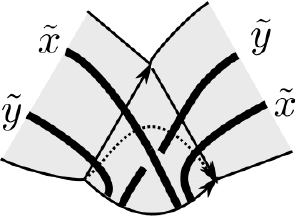}} }
\def\brtthreenew{  \raisebox{-13pt}{\incl{1.3 cm}{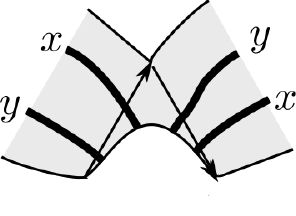}} }
\def\brtfournew{  \raisebox{-8pt}{\incl{.8 cm}{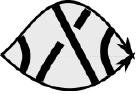}} }

\begin{proof}  Let $\ori_+$ be the positive orientation of the boundaries of $\fS$ and $\ufS$.  In the proof all tangle diagrams will have positive height order.

For a stated $\pfS$-tangle diagram $x$ let $f(x)$ be the stated $\partial \ufS$-tangle diagram obtained from $x$ by extending the strands ending on {$a_1\sqcup a_2$} until they end on $e$, with order on $e$ given by its positive direction,
see  Figure \ref{fig:brtensor0}. We require that $f(x)$ has no crossing inside $\cP_3$ and this makes $f(x)$ unique up to isotopy. Since $f$ clearly preserves the defining relations of a stated skein algebra, we can extend it to an $\cR$-linear map $f:\Ss(\fS)\to \Ss(\underline{\fS})$.

Recall that $\uot \cSs(\fS)$ is the same $\CSS$ with a new product $\uast$ given by \eqref{eq:selfproduct}.

\begin{lemma}The map $f: \uot \cSs(\fS) \to \cSs(\ufS)$ is an algebra homomorphism.
\end{lemma}
\begin{proof} 
Let $x,y$ be stated $\pfS$-tangle diagrams.

\FIGc{brtensorxy-new}{$xy,f(xy)$, and $f(x) f(y)$}{2.5cm}
We present $xy, f(xy), f(x)f(y)$ schematically as in Figure \ref{fig:brtensorxy-new}. By splitting along the dashed line in the picture of $f(x)f(y)$ and using the counit property which says $u = \sum u' \ve(u'')$, we get, with $\Delta_2(x) = \sum x'\ot u$ and $\Delta_1(y) = \sum y'\ot v$,
\begin{align*}
f(x) f(y)
&= \sum \brtthreenew \ \ot \  \ve\left( \brtfournew \right) 
\\
&= \sum \,  f {( {x'  y'})\rho' (u \ot  v)},
 \qquad \text{using co-$R$-matrix $\rho'$ of  \eqref{eq.coR2}} \\
&= f(x\uast y)  \qquad \text{by   \eqref{eq:selfproduct}}.
\end{align*}
Thus $f$ is an algebra homomorphism.
\end{proof}

It remains to show that $f$ is an $\cR$-linear isomorphism.
For $l\in \BN$ let $F_l(\cSs(\fS))= F_l^{a_1, a_2}(\cSs(\fS)) $ and $F_l(\cSs(\ufS))= F_l^{a_1, a_2}(\cSs(\ufS)) $ be the filtrations defined in Subsection \ref{sec.filtration}. In other words, $F_l(\cSs(\fS))$ is
the ${\cR}$-submodule spanned by stated tangle diagrams $\al$ such that $I(\al, a_1)+ I(\al,a_2) \le l$, and similarly  for $F_l(\cSs(\underline{\fS}))$.
Denote by $\Gr_*$ the corresponding graded $\cR$-modules.
It is clear that $f$ preserves the filtrations $F_l$. It is enough to show that $\Gr(f)$ is a bijection.

Let $B_{m,n}$ be the set of isotopy classes of simple $\pfS$-tangle diagrams $\al$ such that $I(\al, a_1)=m, I(\al, a_2)=n$ and $\al$ is increasingly stated on each boundary edge, except for $a_1$ and $a_2$ where it is not stated.
 Then 
\begin{align}
\Gr_l (\cSs(\fS)) &= \bigoplus _{ m+n=l, \ x\in B_{m,n}} V(x)\\
\Gr_l (\cSs(\fS)) &= \bigoplus _{ m+n=l, \ x\in B_{m,n}} W(x).
\end{align}
Here $V(x)$ is the $\cR$-submodule of $\Gr_l (\cSs(\fS))$ spanned by $\al\in B(\fS;\ori_+)$ such that $\al=x$ if we forget the states on $a_1\cup a_2$, and $W(x)$ is the $\cR$-submodule of $\Gr_l (\cSs(\underline{\fS}))$ spanned by
$z\in B(\underline{\fS};\ori_+)$ such that $z\cap \fS= x$. It is enough to show that $\Gr(f)$ is an isomorphism from $V(x)$ to $W(x)$ for $x\in B_{m,n}$. 

Note that both $V(x)$ and $W(x)$ are free $\cR$-modules, and both have rank
$ (m+1)(n+1)$.
Indeed there are $m+1$ increasing states on $x\cap a_1$ and $n+1$ increasing states on $x\cap a_2$ and these can be chosen independently, thus $\rk_{\cR}(V(x))=(m+1)(n+1)$. For what concerns $W(x)$, observe that if  $z\in B(\underline{\fS};\ori_+)$ such that $z\cap \fS= x$ then $z\cap \cP_3$ consists of  $k$ arcs (for some $k\in [0,\min(m,n)]$) connecting $ e_1$ and $ e_2$, $m-k$ arcs connecting $ e_1$ and $ e$, and finally $n-k$ arcs connecting $ e_2$ and $ e$; furthermore $z\cap e$ is increasingly stated so that there are exactly $(m+n-2k+1)$ such $z$.
 Thus we have:
$$\rk_{\cR}(W(x))=
\sum_{k=0}^{min(m,n)}(m+n-2k+1)=(m+1)(n+1).$$
The reordering relation \eqref{eq.order} implies the relation in Figure \ref{fig:brten6-new}, which converts arcs connecting $e_1$ and $e_2$ to arcs with one end in $e$.  This shows that $\Gr(f): V(x) \to W(x)$ is surjective.
\FIGc{brten6-new}{}{1.7cm}

 Since both $V(x)$ and $ W(x)$ are free $\cR$-modules having the same rank, we conclude that $\Gr(f): V(x) \to W(x)$ is an isomorphism. This completes the proof of the theorem.
\end{proof}

\subsection{Examples: Polygons, punctured bigons, and punctured monogons} \label{sub:monogons}

\begin{example}[Polygon] \label{polygons}
The {\em polygon $\P_n$ }is the standard  disc with $n$ punctures on its boundary removed.
Note that the  triangle $\P_3$ is the result of attaching an ideal triangle to two bigons. By Theorem \ref{teo:braidedtensor} we have 
\be  \cS(\P_3) \cong \OSL\underline \ot_{\OSL} \OSL,
\ee
where each copy of $\OSL$ is a right $\OSL$-comodule algebra via the coproduct. Consequently, $\cS(\P_3) \cong \OSL \ot_\cR \OSL$ as $\cR$-modules, and its algebra structure is described by \eqref{eq.brprod}. From here one can get a presentation of $\cS(\P_3)$. In \cite{Le:TDEC}, a presentation of $\cS(\P_3)$ was obtained by brute force calculation.

The $n$-gon $\P_n$ is the result of attaching an ideal triangle to the disjoint union of $\P_{n-1}$ and the bigon. By induction, we obtain
\begin{corollary}\label{cor:polygons} One has
$$\Ss(\P_n)\cong \OSL\uot\cdots \uot \OSL,$$
 where there are $(n-1)$ copies of $\OSL$.
 \end{corollary}
 
 \end{example}

\def\sw{\mathrm{sw}}
\def\ro{\mathrm{ro}}
\begin{example}[Punctured bigons] Let $\B_n$ be the bigon $\B$ with $n$ interior punctures removed. 
For example 
$\B_0=\B$. Like in the case of $\B$, we will show that $\Ss(\B_n)$ has a natural structure of a Hopf alagebra where all the operations can be defined geometrically. Recall that we  denote $e_l$ and $e_r$ the left and right boundary edges of $\B$.
\FIGc{bigonn}{The embedding of the 4-punctured bigon $\B_4$ into $\B_8$ and the splitting homomorphism. The composition gives the coproduct.}{2cm}

Let $\iota:\B_n\hookrightarrow \B_{2n}$ be the inclusion identifying $\B_n$ with the complement of $n$ closed disjoint arcs, each  connecting two punctures of  the $2n$-punctures of $\B_{2n}$; See Figure \ref{fig:bigonn}. Let  $\Delta:\cSs(\B_n)\to \cSs(\B_n)\otimes\cSs(\B_n)$ be the map induced by $\iota$ and then splitting along the vertical arc connecting the two boundary ideal vertices of $\B_{2n}$ and identifying the two halves with $\B_n$. 

Let also $\epsilon:\cSs(\B_n)\to \cR$ be the map obtained by including $\B_n$ in $\B_0=\B$ and then applying $\epsilon :\cSs(\B)\to \cR$. 

Finally  define $S:\cSs(\B_n)\to \cSs(\B_n)$ as the $\cR$-linear map whose value on a stated tangle  $\alpha$ in $\B_n \times (-1,1)$ is obtained by first switching all the states $\eta$ to $-\eta$ and all the framing vectors $v$ to $-v$, then rotating $\alpha$ by $180^\circ$ around the axis passing through the two boundary ideal vertices and finally multiplying the result by $(\sqrt{- 1}\,  q)^{(\delta_{e_l}(\alpha)-\delta_{e_r}(\alpha))}$ (where $\delta_{e}(\alpha)$ was defined in Subsection \ref{sec.filtration} as the sum of the states on $\alpha\cap e$). It is easily checked that all the defining relations \eqref{eq.skein}-\eqref{eq.order} are preserved so that $S$ is well-defined. 

\begin{proposition}\label{lem:Bnstandardhopf}

 For each $n\geq0$ the skein algebra $\cSs(\B_n)$, endowed with $\Delta,\epsilon,S$, is a Hopf algebra.

\end{proposition}
\begin{proof} From the definition of the splitting homomorphism it is clear that $\Delta$ is an algebra homomorphism. The
coassociativity of $\Delta$ is a direct consequence of the fact that applying twice $\iota$ induces the same morphism as the identification of $\B_n$ with the complement of $n$ disjoint arcs in $\B_{3n}$ each containing 3 punctures as depicted here:
$$\raisebox{-1cm}{\incl{2cm}{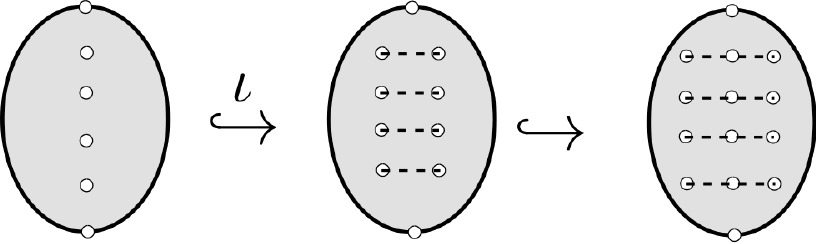}}
$$
The map $\epsilon$ is a morphism of algebras by its definition; to verify that $(\epsilon\otimes \Id)\circ \Delta=\Id$ observe that if in Figure \ref{fig:bigonn} we fill the left punctures then we obtain the initial $\B_n$.

That $S$ is an antimorphism is a consequence of the fact the revolution by $180^\circ$ about the 
the axis connecting the two boundary vertices reverses the height order in $\B_n \times (-1,1)$.

We are left to prove  that  
\be (S\otimes \Id) \circ \Delta=(\Id\otimes S) \circ \Delta=\epsilon, 
\label{eq.SS}
\ee and since $S$ is an antimorphism, it is enough to check this identity on a set of generators. From relation \eqref{eq.order} we get that the set of  horizontal arcs, with all possible states, generates $\cSs(\B_n)$.
If $\al_{\eta\mu}$ is a horizontal arc with state $\eta$ on the left and state $\mu$ on the right, then the definition gives
$$S(\alpha_{\eta{\eta}})=\alpha_{\overline{\eta}\, \overline{\eta}}\  {\ \rm and\ }\  S(\alpha_{{\eta}\overline{\eta}})=-q^{2\eta}\alpha_{{\eta}\overline{\eta}},$$ so that on $\B_0=\B$, $S$ coincides with the antipode defined in \eqref{eq.ap1}. Now Identity \eqref{eq.SS} for $\al_{\eta\mu}$ follows from the same identity for the antipode in $\cSs(\B)=\OSL$.
\end{proof}
\begin{remark} (a) Since $\B_n$ is the result of attaching $2n$ ideal triangles to $n+1$ bigons, Theorem \ref{teo:braidedtensor} can be used to show that as $\cR$-algebras $\cSs(\B_n) \cong \OSL^{\ot (n+1)} $ where the tensor product is over $\cR$ and the algebra structure  of $\OSL^{\ot (n+1)}$ is the unique one determined by:

(i) the subset $A_i= 1^{\ot (i-1)}\ot_\cR \OSL \ot_\cR 1^{\ot (n+1-i)}\subset \OSL^{\ot (n+1)}$ is isomorphic to  $\OSL$ as  $\cR$-algebras for each $i=1,\dots, n+1$, and 

(ii) for $a\in A_i$ and $b\in A_j$ with $i< j$, one has $ab= a\ot b$ and
\begin{align}
ba  = \sum  \bar \rho'(b'\ot a') \rho'(b'''\ot a''')  a'' b'' 
\end{align}
where $\rho'$ is the co-R-matrix defined by \eqref{eq.coR2} and  $\bar \rho'$ is its inverse.

(b) For the  case  $n=1$, Proposition  \ref{lem:Bnstandardhopf} and a presentation of $\cSs(\B_1)$ were also independently obtained in \cite{Ko} via a direct calculation. 

\end{remark}

\end{example} 

\begin{example}[Punctured monogons] Let $\M_n$  be the monogon $\M$ with  $n$ punctures in its interior removed, see  Figure \ref{fig:monogonsbraided}. Let  $\bD:\cSs(\M_n)\to \cSs(\M_n)\otimes \cSs(\B)$ be the right $\cSs(\B)$-comodule algebra structure induced by the only boundary edge of $\M_n$.

Similarly $\cSs(\B_n)$ has two commuting right comodule-algebra structures over $\cSs(\B)$ induced by $e_l$ and $e_r$, let $\Delta_1$ be the one induced $e_l$ and $\Delta_2$ be the one induced by $e_r$.

Like in the bigon case, let $\iota:\M_n\hookrightarrow \M_{2n}$ be the inclusion identifying $\M_n$ with the complement of $n$  disjoint arcs, each  connecting two punctures of $\M_{2n}$. Observe that $\M_{2n}$ is the result of attaching an ideal triangle to two copies of $\M_n$, see Figure \ref{fig:monogonsbraided}.
By Theorem \ref{teo:braidedtensor} we have isomorphism of algebras $\cSs(\M_{2n})=\cSs(\M_n)\uot\cSs(\M_n)$.
 Let  $\Delta:\cSs(\M_n)\hookrightarrow \cSs(\M_{2n})=\cSs(\M_n)\uot\cSs(\M_n)$ be the map induced by $\iota_*$ and this isomorphism. 
\begin{figure}
$\iota:\raisebox{-1cm}{\includegraphics[width=2cm]{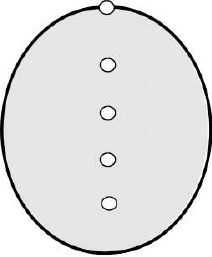}}\hookrightarrow\raisebox{-1cm}{\includegraphics[width=2cm]{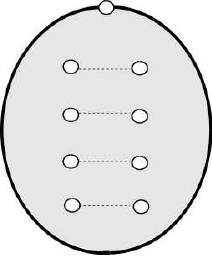}}\simeq  \raisebox{-1cm}{\includegraphics[width=2cm]{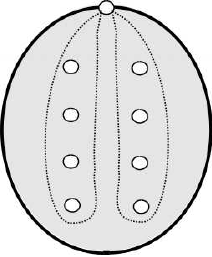}}$
\caption{
The inclusion $\iota:\M_4\hookrightarrow \M_8$ and the decomposition of $\M_{8}$ in a triangle and two copies of $\M_4$ used to fix the isomorphism $\cSs(\M_8)=\cSs(\M_4)\uot\cSs(\M_4)$.}\label{fig:monogonsbraided}
\end{figure}

\begin{proposition}
(a) For each $n\geq 0$ the algebra $\cSs(\M_n)$ endowed with the map $\Delta$ and the map $\epsilon:\cSs(\M_n)\to \cSs(\M_0)=\cR$ induced by inclusion is a bialgebra object in the category of $\OSL$-comodules (i.e. its product, coproduct, unit and counit are morphisms of $\OSL$-comodules). 

(b) 
The $\OSL$ comodule-algebra $\cSs(\M_n)$ is isomorphic to the self braided tensor product $\uot \cSs(\B_{n-1})$.
In particular $\cSs(\M_1)$ is isomorphic as a Hopf algebra to $BSL_q(2)$,  the ``transmutation'' of $\OSL$, or ``braided version'' or ``covariant version'' of $\OSL$ (see \cite{MajidFoundation} Examples 4.3.4 and 10.3.3). 
\end{proposition}
\begin{proof}
(a) The inclusion $\iota$ induces an injective  algebra  homomorphism  $\iota_*:\cSs(\M_n)\hookrightarrow \cSs(\M_{2n})=\cSs(\M_n)\uot\cSs(\M_n)$. As in the case of bigons, the coassociativity follows from the fact that applying twice $\iota$ induces the same morphism as the identification of $\M_n$ with the complement of $n$ disjoint arcs in $\M_{3n}$ each containing 3 punctures as depicted here:
$$\raisebox{-1cm}{\includegraphics{braidedmonogon1-eps-converted-to.pdf}}\hookrightarrow\raisebox{-1cm}{\includegraphics{braidedmonogon2-eps-converted-to.pdf}}  \hookrightarrow\raisebox{-1cm}{\includegraphics{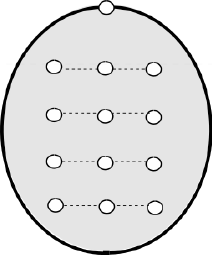}}.$$
The map $\epsilon$ is a morphism of comodule-algebras by its definition; to verify that $(\epsilon\otimes \Id)\circ \Delta=\Id$ observe that if in Figure \ref{fig:monogonsbraided} we fill the left punctures then we obtain the initial $\M_n$. 
The last fact to verify is that $\Delta$ is a morphism of comodules, i.e. denoting $\bD:\cSs(\M_n)\to \cSs(\M_n)\otimes \cSs(\B)$ the right comodule structure, that $$(\Delta\otimes \Id)\circ \bD=(\Id\otimes \Id\otimes m)(\Id\otimes \fl\otimes\Id)\circ(\bD\otimes \bD)\circ\Delta,$$ where $m:\cSs(\B)\otimes \cSs(\B)\to \cSs(\B)$ is the multiplication and $\fl$ is the flip. 
Since $\Delta$ is a morphism of algebras, it is sufficient to check this on generators of $\M_n$. For this we refer to Figure \ref{fig:deltamonogon}, where we depict the case $n=1$ but the proof is similar for other $n$. Letting $\alpha(x,y)\in B(\M_n,\ori)$ be the generator represented by a horizontal arc whose states are $y,x\in \{\pm\}$ depicted in the figure, the equality in the figure shows that we have: $$\Delta(\alpha(x,y))=q^{\frac{1}{2}}\alpha(+,y)\uot \alpha(x,-)-q^{\frac{5}{2}}\alpha(-,y)\uot\alpha(x,+).$$ Also, from the definition of $\bD$ one computes $\bD(\alpha(x,y))=\sum_{\epsilon,\eta} \alpha(\eta,\epsilon)\otimes (\alpha_{\eta x}\alpha_{\eta y})$ where $\alpha_{xy}$ denote the standard generators of $\cSs(\B)$.
Now the verification is a straightforward computation.
\begin{figure}
$\raisebox{-1.0cm}{\put(57,55){$x$}\put(-7,55){$y$}\includegraphics[width=2cm]{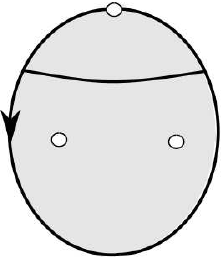}}=\ \ \ \ -q^{\frac{5}{2}}\raisebox{-1.0cm}{\put(57,55){$x$}\put(-7,55){$y$}\put(15,-6){$-$}\put(35,-6){$+$}\includegraphics[width=2cm]{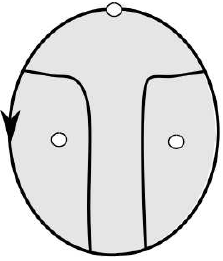}}+ q^{\frac{1}{2}}\raisebox{-1.0cm}{\put(57,55){$x$}\put(-7,55){$y$}\put(15,-6){$+$}\put(35,-6){$-$}\includegraphics[width=2cm]{productbraidedcoproduct2-eps-converted-to.pdf}}.$
\caption{Re-expressing $\Delta(\alpha(x,y))$ as a braided tensor product. The equality is a consequence of \eqref{eq.order}.}\label{fig:deltamonogon}
\end{figure}

\def\uA{{\underline A}}
\def\udot{{\overset \cdot -}}
\def\uS{\underline S}
(b) As $\M_n$ is obtained by attaching an ideal triangle to $\B_{n-1}$, the claim follows from Theorem \ref{teo:braidedtensor}.
The isomorphism of $\cSs(\M_1)$ with $BSL_q(2)$ then follows from Example \ref{ex.trans}. Let us make explicit the isomorphism.

Let $T:\cSs(\B)\to \cSs(\M_1)$ be the $\cR$-linear isomorphism obtained on a skein $a$ by embedding $\rot_*(\overline{\inv}_{e_l}^{-1}(a))$ (see Subsection \ref{sec.invers}) 
in the monogon through the embedding depicted in Figure \ref{fig:monogons} (a) and finally extending the strands of $\rot_*(\overline{\inv}_{e_l}^{-1}(a))$ until they hit the boundary of $\M_1$ (i.e. by applying the map $f$ defined in the proof of Theorem \ref{teo:braidedtensor}). We claim that pulling back through $T$ the coaction of $\cSs(\M_1)$ we get the right adjoint coaction on $\OSL=\cSs(\B)$ and pulling back the product  $*$ on $\cSs(\M_1)$ we get the product $\uast$ on $BSL_q(2)$ defined in Example \ref{ex.trans}:
$$\Delta^{coad}=(T^{-1}\otimes Id)\circ \bD\circ T \qquad {\rm and} \qquad
\uast=T^{-1}\circ(\cdot *\cdot)\circ (T\otimes T).$$  

\begin{figure}
\raisebox{1cm}{(a)}\includegraphics[width=2cm]{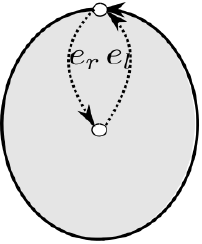}\raisebox{1cm}{\ \ (b)}\raisebox{1cm}{\ $\bD\Bigg($}\includegraphics[width=2cm]{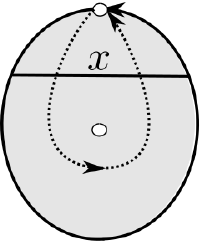}\raisebox{1cm}{$\Bigg)=\ $}\includegraphics[width=2cm]{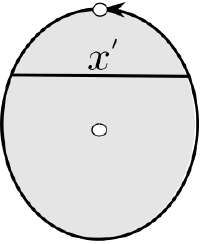}\raisebox{1cm}{$\ \otimes\ $}\includegraphics[width=1cm]{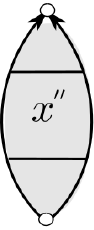} 
\caption{(a) The $\cR$-linear map $T$ embeds $\B$ into $\M_1$ as shown, after applying the map $\overline{\inv}_{e_l}^{-1}$ and rotating by $\pi$. (b) The right coaction of $\cSs(\B)$ on $\cSs(\M_1)$ is obtained by cutting on the dotted arc; in the figure, a skein $x=T(a)$ is cut in three parts by the dotted arc, the mid one is $T(a_{(2)})$ while the rightmost and leftmost are respectively $S(a_{(1)})$ and $a_{(3)}$ so that their product in the bigon cut out by the dotted curve is $S(a_{(1)})a_{(3)}$. This shows that the pullback of $\bD$ is the adjoint coaction.}\label{fig:monogons}
\end{figure}
Using the definitions of $T$, of the antipode on $\cSs(\B)$ \eqref{eq.antipode} and of $\bD$, one verifies directly that $(T^{-1}\otimes Id)\circ \bD\circ T(x)=x''\otimes S(x')x'''$ for any $x\in \OSL=\cSs(\B)$ (see Figure \ref{fig:monogons} (b)). 
Then since $*$ and $\bD$ are compatible it is sufficient to check that the pullback of $*$ equals $\uast$ on the generators of $\OSL$; this is a straightforward computation. 
A graphical explanation is as follows. Observe that if $x=T(a),y=T(b)$, then $x*y=(T(a*b))_{(1)}\epsilon((T(a*b))_{(2)})$ (this holds in general comodule algebras), so that we have:
\begin{center}
\includegraphics[width=2cm]{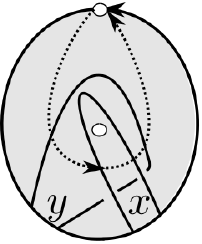}\raisebox{1cm}{$\ =\ $}\includegraphics[width=2cm]{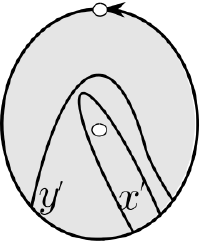}\raisebox{1cm}{$\cdot \epsilon\Bigg($}\includegraphics[width=1cm]{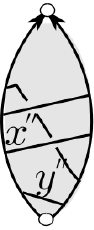} \raisebox{1cm}{$\Bigg)$}
\end{center}
and that the right hand side equals $$T(a_{(2)}\cdot b_{(2)}) \rho' (S(a_{(1)}) a_{(3)} \ot S(b_{(1)}))\epsilon(b_{(3)})=T(a_{(2)}\cdot b_{(2)}) \rho' (S(a_{(1)}) a_{(3)} \ot S(b_{(1)}))$$ where $\rho'$ was defined in \eqref{eq.coR2}. This proves that $\uast=T^{-1}\circ(\cdot *\cdot)\circ (T\otimes T)$.
\end{proof}

\end{example}
\section{A lift of the Reshetikhin-Turaev operator invariant} \label{sec.WT}
In this section we show that a Reshetikhin-Turaev operator invariant of tangles can be lifted to an invariant with values in $\OSL$. In this section $\cR=\BZ[q^{\pm 1/2}]$.

\def\tV{\tilde V}
\def\dr{\partial_r}
\def\dl{\partial _l}

\subsection{Category of non-directed ribbon graphs} We will present the category of {\em non-directed ribbon graphs} \cite{Tu:Book}, also known as {\em framed tangles} \cite{Oh}, in the form convenient for us.

 The bigon is canonically isomorphic (in the category of punctured bordered surface) to the square $S= [0,1] \times (0,1)$.
Under the isomorphism $\el$ and  $\er$ are mapped respectively to $\{0\}\times (0,1)$ and $\{1\} \times (0,1)$, and abusing notation we also denote $\{0\}\times (0,1)$ and $\{1\} \times (0,1)$ respectively $\el$ and $\er$. We identify $S$ with $S \times \{0\}$ in $M:=S \times (-1,1)$. We have $\partial M = \partial S \times (-1,1) = (\el \cup \er) \times (-1,1)$.

\FIGc{tangle}{Left: Square $S=[0,1] \times (0,1)$, with edges $\el$ and $e_r$. Middle: tensor product $\beta\otimes \beta'$. Right: composition $\beta \circ \beta'$, which can be defined only when $|\dr\beta|=|\dl\beta'|$ .}{2.5cm}

Recall that in the definition of a $\partial M$-tangle we require the boundary points over \red{any} boundary edge have distinct heights (see Subsection \ref{sub:tangle}).
If we change this requirement to: all boundary points are in $\partial S$ (in particular they all have the same height) we get the notion of a {\em $\partial S$-tangle}. Formally, a {\em $\partial S$-tangle} is a framed compact 1-dimensional unoriented manifold $\beta$ properly embedded in $M=S \times (-1,1)$ such that $\partial \beta$ has height 0, i.e.  $\partial \beta \subset \partial S =\el \cup \er$, and the framing at every boundary point of $\beta$ is vertical. Let $\dr \beta= \beta\cap \er$ and $\dl\beta = \beta \cap \el$.
Two $\partial S$-tangles are {\em $\partial S$-isotopic} if they are isotopic in the class of $\partial S$-tangles. If $|\dr\beta|=k$ and $|\dl\beta|=l$, then our notion of a  $\partial S$-tangle is the notion of  a non-directed ribbon $(k,l)$-graph without coupons in \cite{Tu:Book}.

After an isotopy we can bring $\beta$ to a generic position (with respect to the projection from $S \times (-1,1)$ onto $S$) and make the framing vertical everywhere\red{. The} projection of $\beta$ together with the over/under information at every crossing, is called a $\partial S$-tangle diagram of $\beta$. The isotopy class of $\beta$ is totally determined by any of its diagrams. 

The {\em non-directed ribbon graph category} is the category whose set of objects is $\BN$ and a morphism from $k$ to $l$ is an isotopy class of {\em $\partial S$-tangle} $\beta$ such that $|\dr\beta|=k$ and $|\dl\beta|= l$, with the usual composition (see Figure \ref{fig:tangle}). If the tangles are oriented, then one would get the usual ribbon tangle category.

If $\beta, \beta'$ are two $\partial S$-tangles, define their tensor product $ \beta \otimes \beta'$ as the result of putting $\beta$ above $\beta'$ as in Figure \ref{fig:tangle}.
Under the tensor product and the composition, morphisms of the non-directed ribbon tangle category are generated by the five {\em elementary $\partial S$-tangles} depicted in Figure \ref{fig:eletangles}. 

\FIGc{eletangles}{Five elementary tangles}{1.5cm}

\def\capp{  \raisebox{-7pt}{\incl{.7 cm}{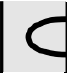}} }
\def\cupp{  \raisebox{-7pt}{\incl{.7 cm}{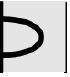}} }
\def\Xp{  \raisebox{-7pt}{\incl{.7 cm}{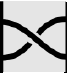}} }
\def\Xn{  \raisebox{-7pt}{\incl{.7 cm}{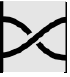}} }

 From the ribbon category of finite-dimensional modules over the quantum group $\USL$ we get the Reshetikhin-Turaev operator invariant of $\partial S$-tangles, see \cite{Tu:Book}.
  Let us describe this operator invariant in a special case. Let $V$ be the free $\cR$-module with basis $g_+, g_-$. The above mentioned operator invariant is the unique 
 functor $Z$ from the non-directed ribbon tangle category to the category of $\cR$-modules preserving the tensor product such that $Z(n) = V^{\ot n}$ and the values of the elementary tangles are given by
\begin{align}
Z \left(\capp  \right) &: V^{\otimes 2} \to R, \   g_+ \ot g_-  \to  q^{-\frac 12}, \ g_- \ot g_+ \to  -q^{-\frac 52}, \ g_+ \ot g_+ \to 0, \  g_- \ot g_- \to 0  \label{eq.cap}\\
Z \left(\cupp  \right) &: R\to V^{\otimes 2},    1 \to  -q^{\frac 52} \, (g_+ \ot g_- )+ q^{\frac 12}\,  (g_- \ot g_+)  \label{eq.cup}\\
Z \left(\Xp  \right)&: V^{\otimes 2} \to V^{\otimes 2},\   Z \left(\Xp  \right) = q\, \id + q^{-1} \left(  Z \left(\cupp  \right) \circ Z \left(\capp  \right)\right),  \label{eq.Xp}\\
Z \left(\Xn  \right)&: V^{\otimes 2} \to V^{\otimes 2},\   Z \left(\Xp  \right) = q^{-1}\, \id + q \left(  Z \left(\cupp  \right) \circ Z \left(\capp  \right)\right), \label{eq.Xn}
\end{align}
see \cite{Co}.  Here our $g_\pm$ are related to the basis vectors  $g_{\frac 12}$ in \cite{Co} by
\begin{align*}
 g_+= -\sqrt{-1} \,  q^{-3/2} g_{\frac 12}, \quad 
g_-= g_{-\frac 12}.
\end{align*}

\def\bon{{\vec \nu}}
\def\bom{{\vec \mu}}

Thus if $\beta$ is a $\partial S$-tangle with $|\dl\beta|=l$ and $|\dr\beta|=k$ then $Z(\beta)$ is an $\cR$-linear map  $V^{\ot k} \to V^{\ot l}$ which depends only on the isotopy class of $\beta$. 
 
For $\bon=(\nu_1,\dots,\nu_l)\in \{\pm \}^l$ and $\bom =(\mu_1,\dots,\mu_k) \in \{\pm \}^k$ we can define the matrix entry $ _\bon Z(\beta) _\bom\in R$ such that
\be 
Z(\beta) (g_{\mu_1} \ot \dots \ot g_{\mu_k} )= \sum_{\bon \in \{\pm \}^l}(\null  _\bon Z(\beta) _\bom)\,  g_{\nu_1} \ot \dots \ot g_{\nu_l}.
\ee

\begin{remark} In fact $V$ and all its tensor powers are  modules over the quantum group $\USL$, and all the operators $Z(\beta)$ are $\USL$-morphisms. But we don't need the structure of $\USL$-modules here. 
When $k=l=0$, we have $Z(\beta)\in R$, which is equal to the Kauffman bracket polynomial of $\beta$.

\end{remark}

\subsection{From $\partial M$-tangles to $\partial S$-tangles} Suppose $\gamma$ is a $\partial M$-tangle. We can  $\partial M$-isotope $\gamma$ so that its diagram $D$   has the height order on $\el$ and $\er$ determined by the arrows in Figure \ref{fig:tangle2}. This diagram determines a unique class of $\partial S$-tangle, denoted by $\bar \gamma$. Note that the arrows of $\er$, $\el$ are irrelevant for $\bar \gamma$. It is easy to see that the map $\gamma\to \bar \gamma$ is a bijection from the set of $\partial M$-isotopy classes of $\partial M$-tangles to the set of $\partial S$-isotopy classes of $\partial S$-tangles.

\FIGc{tangle2}{Direction of boundary edges, used to determine the height order}{2cm}

Suppose $|\gamma\cap \el|=l$ and $|\gamma\cap \er|=k$, and  $\bon=(\nu_1,\dots,\nu_l)\in \{\pm \}^l$ and $\bom =(\mu_1,\dots,\mu_k) \in \{\pm \}^k$. Let $_\bon \gamma _\bom$ be the stated $\partial M$-tangle \red{ whose underlying tangle} is $\gamma$ and whose states on $\gamma \cap \el$ (respectively on $\gamma \cap \er$) from top to bottom by the height order are $\nu_1,\dots, \nu_l$ (respectively $\mu_1,\dots, \mu_k$).
\begin{theorem} \label{teo:cat}
Assume the above notation. Consider $_\bon \gamma _\bom$ as an element of $\Ss(\cB)$. Then
\be  \epsilon \left ( _\bon \gamma _\bom   \right)\  =\ _\bon Z(\bar \gamma) _\bom. 
\label{eq.cat}\ee
\end{theorem}
Thus we see that the tangle invariant of  $_\bon \gamma _\bom$ with values in $\Ss(\cB)$ is stronger than the Reshetikhin-Turaev operator invariant. 

\begin{proof} Suppose $\gamma_1, \gamma_2$ are $\partial M$-tangles. 
Since $\overline{\gamma_1 \gamma_2}= \bar \gamma_1 \otimes \bar \gamma_2$, if \eqref{eq.cat} is true for $\gamma = \gamma_1$ and $\gamma=\gamma_2$, it is true for $\gamma= \gamma_1 \gamma_2$.

\def\boe{{\vec \eta}}

Now suppose $\gamma_1, \gamma_2$ are obtained by splitting a $\partial M$-tangle $\beta$  along an ideal edge. By the splitting formula \eqref{eq.coproduct1} and the definition of $\D$, 
$$ \Delta(\null  _\bon \beta _\bom) = \sum_{\boe}   \null _\bon (\gamma_1)_\boe \, \null _\boe (\gamma_2) _\bom. $$
Applying $\epsilon \ot \id$ to the above, we get
$$ \null  _\bon \beta _\bom = \sum_{\boe}  \epsilon( \null _\bon (\gamma_1)_\boe) \, \null _\boe (\gamma_2) _\bom.$$
Applying $\epsilon$ to the above, we get
$$ \epsilon(\null  _\bon \beta _\bom ) = \sum_{\boe}  \epsilon( \null _\bon (\gamma_1)_\boe) \, \epsilon(\null _\boe (\gamma_2) _\bom),$$ 
which shows that if \eqref{eq.cat} holds for $\gamma=\gamma_1$ and $\gamma=\gamma_2$ then it holds for $\beta=\gamma_1 \circ \gamma_2$.

Thus it is enough to check \eqref{eq.cat} for the elementary tangles, for which  \eqref{eq.cat} follows from the explicit formulas \eqref{eq.cap}--\eqref{eq.Xn}.
\end{proof}

\subsection{A $1+1$-TQFT}

\def\CatBim{\mathsf{SetABim}}
\def\NSCatBim{\mathsf{U_{q^2}(\mathfrak{sl}_2)\!-\!finBim}}

Let $\Cob_{1,1}$ be the symmetric monoidal category whose:

\noindent{\bf Objects} are numbered disjoint unions of open unoriented segments.\\
{\bf Morphisms} are diffeomorphism classes of punctured bordered surfaces $\fS$ with indexed boundary. Explicitly if $\partial \fS=e_1^L,\cdots ,e^L_m, e^R_1,\cdots, e^R_n$ then $\fS\in Mor(e_1^L\sqcup \cdots \sqcup e^L_m,e^R_1\sqcup \cdots \sqcup e^R_n)$ and the composition of morphisms is given by the glueing of marked surfaces explained above (associativity of compositions is ensured by the fact that we consider diffeomorphism classes of surfaces). In particular the identity morphism of $e_1\sqcup\cdots \sqcup e_n$ is a disjoint union of $n$ copies of $\B$.\\
{\bf Tensor product} is the disjoint union, where the components of $(e_1\sqcup \cdots \sqcup e_n)\sqcup (e'_1\sqcup \cdots \sqcup e'_m)$ are ordered as $e_1\sqcup \cdots e_n\sqcup e'_1=e_{n+1}\sqcup e'_m=e_{m+n}$.

In order to define the target category of our TQFT functor, let us fix some notation. 
Given a finite set $C$, we will then denote by $\USL^{\otimes C}$ the algebra obtained as the tensor product $\bigotimes_{c\in C} \USL$ where each copy of $\USL$ in the tensor product is indexed by a distinct element of $C$.

\begin{definition}[$\NSCatBim$]\label{def:NSCatBim}
Let $\NSCatBim$ be the category whose objects are pairs $(C,[M])$ where $C$ is a finite set, $M$ is a right module over $\USL^{\otimes C}$ which is a direct sum of finite dimensional modules and $[M]$ is its isomorphism class. A morphism from $(C,[M])$ to $(C',[M'])$ in $\NSCatBim$ is the \emph{isomorphism class} of a bimodule $B$ over $(\USL^{\otimes C},\USL^{\otimes C'})$ which is a direct sum of finite dimensional bimodules and such that $[M\otimes_{\USL^{\otimes C}} B]=[M']$. The composition of $[B]:(C,[M])\to (C',[M'])$ and $[B']:(C',[M'])\to (C'',[M''])$ is $[B\otimes_{\USL^{\otimes C'}}B']$ (the composition is associative as we consider bimodules up to isomorphisms).  The monoidal structure on $\NSCatBim$ is given by $(C,[M])\otimes (C',[M']):=(C\cup C',[M\otimes_{\cR} M'])$ and its symmetry is given by exchanging $(C,[M])$ and $(C',[M'])$. 
\end{definition}

Then let $\cSs:\Cob_{1,1}\to \NSCatBim$ be defined as $$\cSs(e_1\sqcup \cdots \sqcup e_n)=(C=\{e_1,\ldots e_n\},[\Qq\otimes_\cR\cSs(\B)^{\otimes C}])$$ and for a punctured bordered surface $\fS$ whose boundary is \red{a union of} $C=\{e_1^L,\ldots ,e^L_n\}$ and $C'=\{e_1^R,\ldots ,e^R_m\}$ let $\cSs(\fS)$ be the isomorphism class of the $(\USL^{\otimes C},\USL^{\otimes C'})$-bimodule $\Qq\otimes_{\cR}\cSs(\fS)$.
\begin{theorem}[Skein algebra as a TQFT]\label{teo:skeintqft}
The functor $\cSs$ is a symmetric monoidal functor into $\NSCatBim$.
\end{theorem}
\begin{proof}
By point b) of Theorem \ref{teo:module} it holds $\Qq\otimes_{\cR}\cSs( \B)=\bigoplus_{i\geq 0} V^L_i\otimes V^R_i$ where $V^L_i$ (resp. $V^R_i$) is the irreducible $i+1$-dimensional left (resp. right) module over $\USL$. 
Then, arguing exactly as in the proof of Theorem \ref{r.Hoch2} one sees that for each $j\geq 0$ it holds $$[V^R_j\otimes_{\USL}(\Qq\otimes_{\cR}\cSs(\B))]=[V^R_j].$$  Then $\Qq\otimes_{\cR}\cSs(\B)$ represents the identity morphism $(\{e\},[M])\to (\{e\},[M])$ (for any edge $e$) if restricted to finite dimensional right $\USL$-modules (which are all direct sums of $V^R_j$'s). 
Let $\fS'$ and $\fS''$  be two bordered punctured surfaces with boundaries indexed so that $\partial^L\fS'
=\{e_1,\ldots ,e_n\}=\partial^R\fS''$ and let $\fS$ be the surface obtained by glueing $\fS'$ and $\fS''$ by identifying the corresponding edges of $\partial^L \fS'$ and $\partial^R\fS''$ via an orientation reversing diffeomorphism. Then $\cSs(\fS')$ (resp. $\cSs(\fS'')$) is a right (resp. left) module over $\USL^{\partial^L\fS'}$ (resp. over $\USL^{\partial^R\fS''=\partial^L\fS'}$). To conclude, a repeated application of Theorem \ref{r.Hoch2} shows that the following holds up to isomorphism:
$$\cSs(\fS)=\cSs(\fS')\otimes_{\USL^{\otimes \partial^L\fS'}} \cSs(\fS).$$  \end{proof}
  
 \begin{remark}
 The previous construction can be improved by passing to the setting of $2$-categories in order to consider objects no longer up to isomorphisms. This requires to consider marked surfaces and bimodules and will be dealt with in another work.  
 \end{remark} 

\def\Bred{\overline{B}}
\def\thetar{\bar\theta}
\def\poS{\partial \mathring {\Sigma}}
\section{A non-symmetric modular operad} \label{sub:cooperad}

In this section we show that stated skein algebras provide an example of ``non symmetric geometric modular operad''. Such objects were defined by Markl (\cite{Mar}) as a generalisation of ``modular operads'' initially defined by Geztler and Kapranov (\cite{GK}). Given a monoidal category $C$, Markl defined a NS modular operad in $C$ as a monoidal functor $NSO:\mathsf{MultiCyc}\to C$ where $\mathsf{MultiCyc}$ is a suitable category of $\mathsf{MultiCyc}$ \red{``multicyclic sets''}. 
In this section we rephrase Markl's definition in the case of a suitable category of punctured bordered surfaces $\NSGraph$; then we define a NS geometric modular operad as a monoidal functor $NSO:\NSGraph\to C$. Finally we re-interpret skein algebras as an example of an NS geometric modular operad with values in $\NSCatBim$ (see Definition \ref{def:NSCatBim}). 

\subsection{The category of topological multicyclic sets $\NSGraph$.} 
In this section all surfaces will be oriented and all homeomorphisms will preserve the orientation. 

A \emph{cutting system} in a bordered punctured surface $\fS$ is a finite linearly ordered set ${\bf \alpha}$ of pairwise disjoint ideal oriented arcs $\alpha_1,\cdots ,\alpha_k\subset \fS$ (see Subsection \ref{sub:puncturedsurfaces}); a homeomorphism of cutting systems $\alpha$ and $\beta$ in $\fS$ is a homeomorphism $\phi:\fS\to \fS$ such that $\phi({\bf \alpha})={\bf \beta}$ so that it preserves the ordering and the orientations of the arcs. 
Cutting along all the arcs of a cutting system ${\bf \alpha}$ produces a bordered punctured surface $\cut_{{\bf \alpha}}(\fS)$ whose homeomorphism class depends only on the homeomorphism class of ${\bf \alpha}$. We will say that a cutting system ${\bf \alpha}$ is \emph{disconnecting}  if each arc in ${\bf \alpha}$ disconnects $\fS$. 

If the connected components of $\fS$ are linearly ordered then one can order the connected components of $\cut_{\bf \alpha}(\fS)$ as follows. Since $\cut_{\bf \alpha}(\fS)=\cut_{\alpha_k}\left( \cut_{\alpha_{k-1}}\left( \cdots\cut_{\alpha_1} (\fS)\right)\right),$ it is sufficient to define how to do it for  of the cut along a single ideal arc $\alpha$. 
If $\alpha$ does not disconnect, then there is a natural bijection between the components of $\fS$ and $\cut_{\alpha}(\fS)$ which induces the ordering on those of the latter surface. If $\alpha$ disconnects $\fS$, since both $\alpha$ and $\fS$ are oriented there is a well defined notion of the connected component of $\cut_{\alpha}(\fS)$ ``lying at the left'' and ``at the right of $\alpha$'' we then order them so that left precedes right and they are in the same position in the global ordering of the components of $\fS$ as the component they come from.

\begin{definition}[$\NSGraph,\NSForest$]
Let $\NSGraph$ be the category whose objects are homeomorphism classes of punctured bordered surface whose connected components are linearly ordered, and where a morphism $\fS' \to \fS$ is a homeomorphism class of a cutting system $\alpha$ in $\fS$ such that $\cut_{\alpha}(\fS)$ is homeomorphic to $\fS'$.
The category $\NSForest$ is the subcategory whose objects are disjoint unions of polygons (see Example \ref{polygons}) and whose morphisms are those represented by disconnecting cutting systems. 
\end{definition}
If $\phi:\fS'\to \cut_{\alpha}(\fS)$ and $\psi:\fS\to \cut_{\beta}(\fS'')$ are homeomorphisms, then the composition of the morphisms associated to $\alpha$  and $\beta$ is the homeomorphism class of $\psi(\alpha)\sqcup \beta\subset \fS''$ where the numbering of the arcs of $\psi(\alpha)$ is lower than those of $\beta$. 
The identity morphism is represented by the class of the empty cutting system and it is straightforward to check that the composition is associative, so that the above are indeed categories. 

Both $\NSGraph$ and $\NSForest$ are symmetric monoidal categories. Indeed the tensor product of $\fS'=\fS'_1\sqcup \cdots \sqcup\fS'_k$ and $\fS=\fS_1\sqcup \cdots \sqcup \fS_h$ (where $\fS'_i,\fS_j$ are connected for all $i,j$ and the linear order of the components is increasing from left to right) is defined as $$\fS'\otimes \fS:=\fS'_1 \sqcup \cdots \fS'_k\sqcup \fS_1\sqcup \cdots \sqcup \fS_h.$$
On the level of morphisms, if $\alpha\subset \fS_1$ and $\beta\subset \fS_2$ are two cutting systems then $\alpha\otimes \beta=\alpha\sqcup \beta$ where the linear order of the arcs of $\alpha$ is lower than that of the arcs in $\beta$. 
The symmetry is given by exchanging the components, so with the above notations $s(\fS'\otimes \fS)=\fS\otimes \fS'$ and $s(\alpha\otimes \beta)=\beta\otimes \alpha$.

 The following definition is a reformulation of Markl's \cite{Mar} (Definition 4.1) in the context of punctured bordered surfaces:
\begin{definition}[NS Modular Operads]
Let $C$ be a symmetric monoidal category.
A $NS$ (non symmetric) geometric modular operad in $C$ is a symmetric monoidal functor $$O:\NSGraph\to C.$$
A $NS$ cyclic operad in $C$ is a symmetric monoidal functor $O:\NSForest\to C$.
\end{definition} 

\subsection{NS geometric modular operads from skein algebras}
Recall that if $\B$ is the bigon with one edge of type ``left'' and one of type ``right'', then $\cSs(\B)=\OSL$ as a $(\USL,\USL)$-bimodule.  Let also $\B^R$ be the bigon whose edges are declared to be both of type $R$ (right edges) then $\cSs(\B^R)$ is the left module over $\USL^{\otimes 2}$ whose underlying space is $\OSL$ and on which the action of $x\otimes y\in \USL^{\otimes 2}$ is given by $x\otimes y\cdot b=x\cdot b\cdot r^*(y)$ (see Example \ref{ex:bigonright}).

\begin{theorem}[Skein algebras as non symmetric operads]\label{teo:NSmodoperad}
There is a geometric NS-modular operad $NSO$ in $\NSCatBim$ defined on an object $\fS$ of $\NSGraph$ as $$NSO(\fS)=(\Edge(\fS),[\Qq\otimes_{\cR}\cSs(\fS)])$$ where $\fS$ is the surface whose edges are all indexed to be of type L (left) and where we see $\Qq\otimes_{\cR}\cSs(\fS)$ as a right module over $\USL^{\otimes  \Edge({\bf \fS})}$ as explained in Subsection \ref{sub:modulestructure}.

If $\phi:{\fS'}\to { \fS}$ is a morphism associated to a cutting system $\alpha$, then let $$NSO(\phi)=[\Qq\otimes \cSs(\B)^{\otimes \Edge(\fS)}\otimes \cSs(\B^R)^{\otimes \alpha}]$$
where $\cSs(\B^R)^{\otimes \alpha}$ is the skein algebra of a disjoint union of one copy of $\B^r$ per arc $\alpha_i\in \alpha$ whose boundary edges correspond to the edges of $\partial \fS'$ lying respectively at the left and at the right of $\alpha_i$. 
\end{theorem}
\begin{proof}
First of all observe that the functor is well defined as all surfaces are seen up to orientation preserving diffeomorphism and all modules and bimodules in $\NSCatBim$ are seen up to isomorphism. 
Then we observe that the skein algebra of a disjoint union of $n$ bigons $$\Qq\otimes_\cR \cSs(\sqcup_{j=1}^n \B_j)=\Qq\otimes_\cR\cSs(\B)^{\otimes n}=\Qq\otimes_\cR\OSL^{\otimes n}$$ is the identity of $(\{1,2,\cdots ,n\},[M])$ (where $i$ is the left edge of the $i^{th}$-bigon) for any right module $M$ which is a direct sum of finite dimensional modules over $\USL^{\otimes n}$. Indeed $M$ is a direct sum of modules of the form $W\otimes V^R_{j}$ where $V^R_j$ is the $j+1$-dimensional irrep of $\USL$ and $W$ is a right module over $\USL^{\otimes n-1}$ which is itself a tensor product of finite dimensional modules. As proved in Theorem \ref{teo:module} $\cSs(\B)=\OSL=\bigoplus_{i} V^L_i\otimes V^R_i$ so that by the same arguments as in the proof of Theorem \ref{r.Hoch2} it holds:
$$[W\otimes V^R_{j}  \otimes_{U_{q^2}(\mathfrak{sl}_2)_{n}} \left( \bigoplus_i V^L_i\otimes V^R_i\right)]=\bigoplus_i[W\otimes \left(V^R_{j}  \otimes_{\USL} V^L_{i}\right)\otimes V^R_{i} ]=[W\otimes V^R_{j} ].$$
This shows that tensoring over $\USL$ with a single copy of $\Qq\otimes_\cR \cSs(\B)$ provides the identity morphism; by Remark \ref{rem:glueingcommutes} repeating this along all the boundary edges one gets that tensoring with $\Qq\otimes_\cR\OSL^{\otimes n}$ is the identity of $(\{1,2,\cdots ,n\},[M])$ for any $M$ decomposing into a direct sum of finite dimensional modules.

Now we prove that if $i,j$ are two distinct boundary edges of a (possibly disconnected) surface $\fS'$, then
\begin{equation}\label{eq:glueingbigon}
[\cSs(\fS')\otimes_{\USL^{\{i,j\}}}\cSs(\cB^R)]=[\cSs(\fS)]
\end{equation}
where $\cSs(\cB^R)$ is seen as left module over $\USL^{\{i,j\}}$ and $\fS$ is the surface obtained by glueing the edges $i,j$ by an orientation reversing homeomorphism. 
Indeed by Remark \ref{rem:glueingcommutes} and Example \ref{ex:bigonR} to glue $\B^R$ along $i$ and $j$, one can first glue $\fS'$ and $\B^R$ along $i$ thus obtaining the surface $\fS'$ whose edge $i$ has been changed to type $R$ (see Example \ref{ex:bigonR}) and then operating a self-glueing along $i$ and $j$ on this surface. 
By Theorem \ref{r.Hoch2} the overall result is $\Qq\otimes_{\cR}\cSs(\fS)$.
Then if $\alpha$ is a cutting system given by $c$ arcs, by  Remark \ref{rem:glueingcommutes} applying $c$ times \eqref{eq:glueingbigon} we get that tensoring with $\Qq\otimes_{\cR}\cSs(\B^R)^{\otimes \alpha}$ is performing the glueing inverting the cut associated to the cutting system $\alpha$.   

\end{proof}

\section{Reduced skein algebra}\label{sec:reduced} We show that the stated skein algebra $\SS$ has a nice quotient $\bSS$, called the {\em reduced stated skein algebra}, which can be embedded in a quantum torus. This quotient is still big enough to contain the ordinary skein algebra and the Muller skein algebra. Unlike the case of the full fledged version $\SS$, when $\fS$ is an ideal triangle, the reduced version $\bSS$ is a quantum torus. The construction of the quantum trace map follows immediately from the splitting theorem for the reduced stated skein algebra. 

Throughout we fix a  punctured bordered  surface $\fS = \bfS \setminus \cP$ and we will denote $\Ss=\Ss(\fS)$.

\subsection{Definition}

A non-trivial arc $\al\subset \fS$, which is the closed interval $[0,1]$ properly embedded in $\fS$ \red{ not homotopic relative its endpoints to a subset of the boundary $\pfS$}, is called a {\em corner arc} 
 if  it is as that depicted in Figure \ref{fig:badarc3}(a), i.e. it cuts off from $\fS$ a triangle with one ideal vertex. Such an ideal vertex is said to be surrounded by the corner arc $\al$. 
 
 A {\em bad arc} is a stated corner arc whose states 
  are as in the figure Figure \ref{fig:badarc3}, i.e. they are $-$ followed by $+$ if we go along the arc counterclockwise around a surrounded vertex.
\FIGc{badarc3}{(a) a bad arc (b) the splitting of a bad arc}{2.5cm} 
The reduced stated skein algebra $\Ssr$ is defined to be  the quotient of \red{$\Ss(\fS)$} by the 2-sided  ideal $\Ibad$ generated by bad arcs.

 \def\bth{{\bar \theta}}
 \subsection{Basis}
Let  $\ori_+$ be the orientation of $\partial \fS$ induced by that of $\fS$, i.e. every boundary edge has positive orientation.
Then $B:=B(\fS;\ori_+)$ is an $\cR$-basis of $\SS$.  Let $\Bred=\Bred(\fS)\subset B$  be the subset consisting of all elements in $B$ which contain no bad arc.
 \begin{theorem} \label{teo:redbasis}
  The set $\Bred$ is a free $\cR$-basis of the $\cR$-module $\Ssr$.
 \end{theorem} 

\begin{proof}  Let $A\subset \Ss$ be the $\cR$-span of $\Bred$ and $A'\subset \Ss$ be the $\cR$-span of $B \setminus \Bred$. One has $\Ss = A \oplus A'$. Let us prove that the ideal $\Ibad$ is equal to $A'$. 

 Proof that $A'\subset \Ibad$. Let $\gamma\in (B \setminus \Bred)$, i.e. $\gamma$ contains a bad arc. We have to show that $\gamma\in \Ibad$. If an arc in $\gamma$ (at some corner) is bad, then the positive orientation and  increasing  states imply that all the arcs closer to the vertex of that corner are bad, see Figure~\ref{fig:badarc2}. 
\FIGc{badarc2}{If the outer arc is bad, then all inner arcs are bad, too.}{1.5cm}

Thus we assume that $\gamma$ has a bad arc which is an inner most arc, see Figure~\ref{fig:move1}(b).

\FIGc{move1}{(a) The product $\al\beta$, here $\al$ (in red) is a bad arc, (b) an element $\gamma\in B$ which has a bad  innermost arc (in red). }{3cm}

We have the relations in Figure \ref{fig:rel5}, which are part of Lemma~\ref{r.refl}.
\FIGc{rel5}{Moving endpoint with negative state (left) and positive state (right)}{2.3cm}
The first relation allows us to move the end of the red arc with state $-$ (in $\gamma$) up until we get the diagram in Figure~\ref{fig:move1}(a), which is of the form $\al\beta$, where $\al$ is a bad arc. The result is that $\gamma \bue\al \beta$. Thus, $\gamma \in \Ibad$.

 Proof that $\Ibad \subset A'$. We have to show that  $\al \beta, \beta \al \in A'$ for any bad arc $\al$ and any $\beta \in B$. 

The product $\al\beta$: In this case, $\al \beta$ is presented as in Figure \ref{fig:move1}(a). We already saw that $\al \beta \bue \gamma$, where $\gamma$ is as in Figure \ref{fig:move1}(b). Since $\gamma\in A'$, we see that $\al\beta \in A'$.

\FIGc{move2}{(a) Product $\beta\al$, where $\al$ is a bad arc (in red), (b) the diagram $\gamma$}{2.5cm}

The product $\beta \al$: In this case, $\beta \al$ is presented as in Figure \ref{fig:move2}(a). Using the 2nd relation in Figure \ref{fig:rel5}, we get that $\beta\al \bue  \gamma$, where $\gamma$ is as in Figure \ref{fig:move2}(b). Since $\gamma\in A'$, we see that $\al\beta \in A'$.

Thus, $\Ibad =A'$. Hence as $\cR$-modules, $\Ssr=\Ss/J \cong A$, which has $\Bred$ as an $\cR$-basis. 
\end{proof}            
 \def\bB{\bar B}
\begin{remark} Positive order is used substantially in the proof. For other orientation of $\pfS$, the set similar to $\bB$ might not be the basis of $\bSS$.
\end{remark}

\begin{corollary} The ordinary skein algebra $\ooS(\fS)$ and the Muller skein algebra $\SMuller$ embed naturally into the reduced skein algebra $\bSS$.
\end{corollary}
\begin{proof} Clearly the standard basis of the ordinary skein algebra and the standard basis of the Muller skein algebra (where all the states are $+$) are subsets of the basis $\bB$ of $\bSS$.
\end{proof}

\subsection{Corner elements} 
\begin{proposition}\label{r.corner} Let $u$ be a stated corner arc with both states positive  and $v$ be the same arc  with both  states negative. Then $uv=vu=1$ in $\bSS$.
\end{proposition}
\begin{proof} In $\bSS$ we have 
$$ vu=  \corone \ = \ q^2 \, \cortwo  + q^{-1/2}\,  \corthree \ = \ q^{-1/2}\,  \corthree  \, = 1,$$
where the second identity follows from \eqref{eq.order} and  the last follows from \eqref{eq.arcs2}. Similarly,
$$ uv =  \coronep \ = \ q^{-2} \, \cortwop  - q^{-5/2}\,  \corthreep \ = \ q^{-5/2}\,  \corthreep  \, = 1,$$
where the second identity follows from \eqref{eq.order} and  the last follows  from \eqref{eq.arcs2}.
\end{proof}
\subsection{Filtration}

For a finite collection $\fA$ of ideal arcs or simple closed loops let $F^\fA_n(\bSS)$ be the $\cR$-submodule of $\bSS$ spanned by stated tangle diagrams $\al$ such that $\sum_{a\in \fA} I(a, \al) \le n$. Then $(F^\fA_n(\bSS))_{n=0}^\infty$ is a filtration of $\bSS$ compatible with the algebra 
structure. Denote by $\Gr^\fA(\bSS)$ the associated graded algebra:
$$ \Gr^\fA(\bSS) = \bigoplus_{n=0}^\infty \Gr^\fA_n(\bSS), \ \red{ \text{where}\ } \Gr^\fA_n(\bSS) = F^\fA_n(\bSS)/F^\fA_{n-1}(\bSS).$$

From Theorem \ref{teo:redbasis} we have the following analog of Proposition \ref{r.basis2}.
\begin{proposition}\label{r.basis2a} Suppose $\fA$ is a collection of boundary edges of $\fS$.

(a) The set $\{\al \in \bB \mid \sum_{a\in \fA} I(\al, a) \le n\}$ is an $\cR$-basis of $F^\fA_n(\bSS)$.

(b) The set $\{\al \in \bB \mid \sum_{a\in \fA} I(\al, a) = n\}$ is an $\cR$-basis of $\Gr^\fA_n(\bSS)$.
\end{proposition}

\subsection{Splitting theorem}
 \begin{theorem} \label{teo:rsplit}
Suppose  $\fS'$ is the result of splitting $\fS$ along an interior ideal arc $a$. The splitting algebra embedding
$\theta_a: \SS \embed \cS(\fS')$
descends to an algebra embedding
\be  \bar\theta_a: \Ssr \to \bcS(\fS').
\label{eq.bth}
\ee

Besides, if $a$ and $b$ are two disjoint ideal arcs in the interior of $\fS$, then 
\be \bth_a \circ \bth_{b} = \bth_{b} \circ \bth_a 
\label{eq.com7}
\ee
\end{theorem}     
    
 \def\aL{{a'}}
 \def\aR{{a''}} 
 \def\TL{\mathrm{TL}}
\def\BL{\mathrm{BL}}
\def\TR{\mathrm{TR}}
\def\BR{\mathrm{BR}}
\def\SL{{S'_L}}
\def\SR{{S'_R}}
\def\ttal{\tilde{\tal}}
\def\ttbeta{\tilde{\tbeta}}
\def\DL{D_L}
\def\PR{P_R}
\def\PL{P_L}  
   
    \begin{proof}
    Suppose $\al\subset \fS$ is a bad arc. The geometric intersection $I(\al,a)$ is 0 or 1. In the first case $\theta_a(\al)=\al$ is also a bad arc in $\fS'$.  In the second case the splitting of $\al$, given in Figure~\ref{fig:badarc3}(b), has a bad arc for both values of  $\nu\in \{\pm\}$. It follows that $\theta_a(\Ibad) \subset \Ibad$. Hence $\theta_a$ descends to an algebra homomorphism $\bth_a:  \bSS \to \bcS(\fS')$ and we also have~\eqref{eq.com7}.

It remains to show that $\bth_a$ is injective. Let $0\neq x\in \Ssr $. We have to show that $\bth_a(x) \neq0$. 
Since $\bB(\fS)$ is an $\cR$-basis, there is a non-empty finite set $S \subset \bB(\fS)$ such that
\be 
x = \sum_{\al\in S} c_\al \al, \quad 0 \neq c_\al \in \cR.
\ee 
Let $k=\max_{\al\in S} I(\al, a)$. Then $S':=\{ \al\in S \mid I(\al,a)=k\}$ is non-empty.

Let $\pr:\fS'\to \fS$ be the projection and  $\aL,\aR \subset \fS'$ be the boundary edges which are  $\pr^{-1}(a)$. To simplify the notations we write $F^\fA_n$ and $\Gr^\fA_n$ for respectively $F^\fA_n(\bcS(\fS')) $ and $\Gr^\fA_n(\bcS(\fS'))$.
From the formula of the splitting homomorphism, for every $\al\in S$,
$$ \bth_a(\al) \in F^\aL_k \cap F^\aR_k \subset  F_{2k}^{\{\aL,\aR\}}.$$
Let $ P:  F_{2k}^{\{\aL,\aR\}} \onto\Gr_{2k}^{\{\aL,\aR\}}$ be the canonical projection.
Clearly if $ \al\in S\setminus S'$ then $P(\al)=0$. We consider  Case 1 and Case 2 below.
\\
{\bf Case 1:} {\em   There exists $\beta\in S$ such that $P(\bth_a(\beta))\neq0$.}

\FIGc{split7}{The split surface $\fS'$, with orientations $\ori'$ on $\pfS'$. The top left, top right, bottom left, and bottom right corners are marked respectively $TL, TR, BL, BR$.} {2.5cm}

Choose an orientation of $a$ such that the induced orientation on $\aR$ is positive. Then the induced orientation on $\aL$ is negative, see Figure \ref{fig:split7}. Let $\ori'$ be the orientation of $\pfS'$ which is positive everywhere except for the edge $a'$ where it is negative. For $\al\in S'$ its lift $\tal=\pr^{-1}(\al)$ is a {\em partially stated tangle diagram}: it is stated everywhere except for endpoints on $\aL\cap \aR$, and the endpoints on each of $\aL$ and $\aR$ are ordered by $\ori'$. Let $\tal^+$ be the same $\tal$ except that the order on $\aL$ (and hence on all edges) is given by the positive orientation.

For $ 0\le j\le k$ let
$s_j(\tal)$ (respectively $s_j(\tal^+))$ be the stated tangle diagram  which is  $\tal$ (respectively $\tal^+$) where  the states on each of $\aL$ and $\aR$ are increasing and having exactly $j$ minus signs. Then $s_j(\tal^+)$ is either equal to $0$ in $\bcS(\fS')$ or belongs to the basis set $\bB(\fS')$.
\def\tbeta{\tilde{\beta}}

By Proposition \ref{r.grbinom} and then Proposition \ref{r.orichange} we have, for some $f(\al,j) \in \BZ$, 
\be 
P(\bth_a(\al)) = \sum_{j=0}^k  \binom kj_{q^4} \, s_j (\tal) =  \sum_{j=0}^k  \binom kj_{q^4} q^{f(\al,j)}s_j(\tal^+). \label{eq.lu1}
\ee
Since  $P(\bth_a(\beta))\neq0$,  there is $l$ such that $s_l(\tbeta^+)\neq 0$ in $\bcS(\fS')$ and hence $s_l(\tbeta^+)\in  \bB(\fS')$.

Using \eqref{eq.lu1} we have
\be 
P (\bth_a((x))= \sum_{\al\in S'} \sum_{j=0}^k \binom kj_{q^4} q^{f(\al,j)} \, c_\al\, s_j(\tal^+).
\label{eq.lu2}
\ee

As  $\al \in S'$ can be recovered from $\tal$, if $\al\neq \beta$  then the two partially stated diagrams $\tal$ and $\tbeta$  are not isotopic.
 It follows that  $s_l(\tbeta^+)\neq s_{j}(\tal^+)$ for all $j$ and all $\al \neq \beta$. It is also clear that $s_j(\btheta^+) \neq s_l(\btheta^+)$ for $j \neq l$.
 Hence the right hand side of \eqref{eq.lu2} is not 0, since the basis element $s_l(\tbeta^+)$ has non-zero coefficient, and all other elements $s_j(\tal^+)$ is either 0 or a basis element different from $s_l(\tbeta^+)$. Thus $P (\bth_a((x))\neq 0$ and consequently $\bth_a(x)\neq 0$. This completes the proof in Case 1.

\noindent {\bf Case 2:} {\em  For all $\al\in S$ we have $P(\bth_a(\al))=0$.} Identity \eqref{eq.lu1} shows that $s_j(\tal^+)=0$ for all $0\le j\le k$ and all $\al\in S'$.

Incident with $\aL$ there are two corners, the top left corner and the bottom left corner. 
 Similarly, incident with $\aR$ there are the top right corner and the bottom right corner, see Figure \ref{fig:split7}. A corner arc of $\tal$ at one of these four corners has  one end stated and one end not stated, and it is called a {\em negative (respectively positive)} corner arc if this only state is negative (respectively positive).

For $\al\in S'$ and $\nu\in \{\pm\}$ let  $\TL_\nu(\al)$  be the number of top left corner arcs  whose only state  is $\nu$. Define $\TR_\pm(\al), \BL_\pm(\al), \BR\pm(\al)$ similarly.

\begin{lemma} Suppose $\al\in S'$. One of the following two mutually exclusive cases happens:\\
 (i)  $\TL_-(\al) >0$ and $\BL_+(\al) >0$, or \\
 (ii) $\BR_-(\al) >0$ and $\BR_+(\al) >0$. \label{r.lu77}
\end{lemma}
\begin{proof} Since $s_0(\tal^+)$ is 0 in $\bcS(\fS')$, it has a bad arc. This implies either $\TL_-(\al) >0$ or $\BR_-(\al) >0$.

Assume $\TL_-(\al) >0$. Since $\al$ does not have a bad arc, we conclude that $\TR_+(\al)=0$. Then from $s_k(\tal^+)=0$ we see that $\BR_+(\al) >0$. Again since $\al$ does not have a bad arc, we conclude that $\BR_-(\al) =0$. Thus we have case (i) but not case (ii).

Assume $\BR_-(\al) >0$. Since $\al$ does not have a bad arc, we conclude that $\BL_+(\al)=0$. Then from $s_0(\tal^+)=0$ we see that $\TR_+(\al) >0$. Again since $\al$ does not have a bad arc, we conclude that $\BL_-(\al) =0$. Thus we have case (ii) but not case (i).
\end{proof}
The cases (i) and (ii) of Lemma \ref{r.lu77} partition  $S'= \SL \sqcup \SR$, where
$$ \SL=\{\al\in S' \mid \TL_-(\al)\, \BL_+(\al)  >0 \}, \quad  \SR=\{\al\in S' \mid \BR_-(\al)\, \BR_+(\al) >0 \}.$$

\begin{lemma} If $\al\in \SL$ then $\bth_a(\al)=0$ in $\Gr^\aL_k$. Similarly, if $\al\in \SR$ then $\bth_a(\al)=0$ in  $\Gr^\aR_k$.
\label{r.lu9}
\end{lemma}
\begin{proof} Suppose $\al\in \SL$. For  $\bnu=(\nu_1,\dots, \nu_k)$ the stated tangle diagram $(\tal,\bnu)$ is defined to be $\tal$ with states on both $\aL$ and $\aR$ are sequence $\bnu$ listed from top to bottom.
By definition,
\be 
\bth_a (\al) = \sum_{\bnu\in \{\pm \}^k} (\tal, \bnu)\label{eq.lu3a}.
\ee
Let $(\tal^+, \bnu)_L$ be 
$\tal^+$ whose states on $\aR$ is given by $\bnu$ but whose state on $\aL$ is given by a permutation of $\bnu$ such that the states are increasing on $\aL$. 
By Proposition \ref{r.orichange} 
\be (\tal, \bnu)\bue  (\tal^+, \bnu)_L  \quad \text{in } \Gr^\aL_k \label{eq.lu6a}.
\ee
If $\bnu$ has at least one negative sign then $(\tal^+, \bnu)_L$ has a bad arc in the bottom left corner (because $BL_+(\al) >0$) and hence is equal to 0 in $\bcS(\fS')$. If $\bnu$ has at least one positive sign  then $(\tal^+, \bnu)_L$ has a bad arc in the top left corner (because $TL_-(\al) >0$) and hence is equal to 0 in $\bcS(\fS')$. Thus we always have $(\tal^+, \bnu)_L=0$ in $\bcS(\fS')$. From \eqref{eq.lu6a} and \eqref{eq.lu3a} we conclude that $\bth_a (\al)=0$ in $\Gr^\aL_k$.

The other case follows from the above case by noticing that if one rotates the Figure~\ref{fig:split7} by $180^\circ$, then the top left corner becomes the bottom right corner.
\end{proof}

As $\emptyset \neq S'= \SL \sqcup \SR$, one of $\SL$ and $\SR$ is non-empty. Without loss of generality we can assume that $\SL$ is not empty.

Let $d = \min \{ \TL_-(\al) \mid {\al \in \SL}\}$ and $S'' =\{ \al \in \SL \mid \TL_-(\al) =d \}$. Then $S''\neq \emptyset$.
 
 Let $P_+: \bcS(\fS') \to \bcS(\fS')$ be the $\cR$-linear map defined on basis elements $\gamma\in \bB(\fS')$ by 
 $$ P_+(\gamma) =   \begin{cases}  \gamma \quad & \text{if} \ I(\gamma,\aL)=d, I(\gamma, \aR)=k,\  \text{all states on $\aL$ are $+$}  \\
0 & \text{otherwise.}\end{cases} 
 $$
 
For $\al\in S''$ let $\ttal$ be the stated tangle diagram obtained from $\tal^+$ by first removing the $d$ negative top left corner arcs  then providing states on $\aL$ and $\aR$ so that all states on $\aL$ are $+$ and the states on $\aR$ are increasing and having exactly $d$ negative signs. Since $\BR_-(\al)=0$, we see that $\ttal$ is an element of the basis $\bB(\fS')$. As $\al$ can be recovered from $\ttal$, the map $\al \to \ttal$ from $S''$ to $\bB(\fS')$ is injective.

Let $u$ be a top left corner arc whose both states are $+$.
\begin{lemma} For $\al\in S$ one has
\be P_+(\bth_a(\al)\, u^d) = \begin{cases} q^{-(k-d)(k-d-1)/2 }\,   \ttal \quad &\text{if} \ \al \in S''\\
0 & \text {if} \ \al \not\in S''.\end{cases}  \label{eq.lu5}
\ee 
\label{r.lu1}
\end{lemma}
\begin{proof} One has $S = (\SR \sqcup \SL) \sqcup (S \setminus S')$. If $\al\in (S\setminus S')$ then $I(\al, a) <k$ and hence $P^+(\bth_a(\al)) =0$.

If $\al\in \SR$ then by Lemma \ref{r.lu9} one has $\bth_a(\al)=0$ in $\Gr^\aR_k$ which means $\bth_a(\al)$ is a linear combination of elements $\gamma\in \bB$ with $I(\gamma,\aR) <k$. It follows that $P_+(\bth_a(\al))=0$.

It remains to consider the case $\al\in \SL= S'' \sqcup (\SL \setminus S'')$.
From \eqref{eq.lu3a},
\be
P_+(\theta_a (\al) u^d ) = \sum_{\bnu\in \{\pm \}^k} P_+((\tal, \bnu)\, u^d).  \label{eq.lu3b}
\ee
Recall that for $\beta\in \bB(\fS')$ one defines $\delta_\aL(\beta)$ as the sum of all the states of $\beta \cap \aL$. From the definition, if $\delta_\aL(\beta) \neq k-d$ then $P^+(\beta)=0$.
If $\bnu$ has $m$ negative signs where $m>d$ then 
$$ \delta_\aL((\tal, \bnu) \, u^d)=  k -2m +d < k-d,$$ 
and hence $P_+((\tal, \bnu) \, u^d)=0$. Thus we can assume that in the sum in \eqref{eq.lu3b}, the number of negative signs in $\bnu$ is $\le d$.

 Assume that $\TL_-(\al)=m$. Note that the $m$ negative top left corner arcs of $\tal$ are below any other components of $\tal$.  Hence the number of the first $m$ components of $\bnu$ must be negative since otherwise one of the $m$ top left corner arcs is bad and $(\tal,\bnu)= 0$ in $\bcS(\fS')$.
We conclude that if  $TL_-(\al) >d$ (that is, if $\al \in \SL\setminus S''$), then
\be P_+(\theta_a (\al) u^d ) =0 . \label{eq.lu7}
\ee
Moreover, if $TL_-(\al)=d$ (that is, $\al \in S''$), then 
\be
P_+(\theta_a (\al) u^d ) = P_+((\tal, \bnu_d)\, u^d), \label{eq.lu4}
\ee
where $\bnu_d\in \{\pm \}^k$ is the sequence whose first $d$ components are $-$ and all other components are $+$. The $d$ negative top left  corner arcs of $(\tal,\bnu_d)$ are all $v$, the corner edge with negative states on both ends.  Hence we have
$ (\tal, \bnu)= \ttal' v^d$, where $\ttal'$ is the same as $\ttal$ except that the order on $\aL$ is negative. Using the height exchange move between positive states, see Equation \eqref{eq.reor1},  and relation $vu=1$ we get that
$$P_+((\tal, \bnu_d)\, u^d) =  q^{-(k-d)(k-d-1)/2 }\ttal, $$
which proves \eqref{eq.lu5} and completes the proof of the lemma.
\end{proof}
Let us continue the proof of the theorem for Case 2. From Lemma \ref{r.lu1} we have
$$ P_+(\bth_a(x) u^d ) = \sum_{\al\in S''} c_\al q^{-(k-d)(k-d-1)/2 } \ttal,$$
which is non-zero since $\{\ttal\}$ are distinct elements of the basis $\bB(\fS')$. The theorem is proved.
    \end{proof}

\subsection{The bigon} The elements $\al_{\mu\nu}\in\cS(\cB)$ are defined in Section \ref{sec:bigon}.
 \begin{proposition} \label{r.rbigon} Let $\cB$ be the bigon.  There is an algebra isomorphism 
 $\bcS(\cB)\cong \cR[x^{\pm 1}]$ given by
 $\al_{++}\to  x , \ \al_{--} \to x ^{-1}, \ \al_{+-} \to 0, \  \al_{-+}\to 0$.

\end{proposition}     
\begin{proof} A presentation of the algebra $\Ss(\B)\cong \OSL$ is given by Theorem \ref{teo:SL2},  with generators $a=\alpha_{++}, b=\al_{--}, c= \al_{-+}, d= \al_{--}$ and relations \eqref{eq.rel1} and \eqref{eq.rel2}.  
The only bad arcs in $\cB$ are $\al_{-+}=c$ and $\al_{+-}=b$. Thus $\bcS(\cB)= \cS(\cB) /\Ibad$ has a presentation like that of $\OSL$, with additional relations $b=c=0$. From the quantum determinant relation in \eqref{eq.rel2} we get $ad=1$ in $\bcS(\cB)$. 

On the other hand it is easy to check that the relations $b=c=0$ and $ad=1$ imply all other relations in \eqref{eq.rel1} and \eqref{eq.rel2}. Hence
$$ \bcS(\cB) \cong \cR\la a,b,c,d\ra /(ad=1, b=c=0) \cong \cR[a^{\pm1}].$$
\end{proof}

\subsection{The triangle} 
Let $\P_3$ be the ideal triangle, with boundary edges $a,b,c$ as in Figure~\ref{fig:triangle}. 
Let $\al, \beta,\gamma$ be the corner arcs which are opposite respectively to  $a, b$ and $c$. For $\mu,\nu\in \{\pm \}$ and $\xi\in \{\al, \beta,\gamma\}$ let $\xi(\mu\nu)$ be the arc $\xi$ with states $\mu$ and $\nu$ on the end points such that $\nu$ follows $\mu$ along $\xi$ counter-clockwise (with respect to the vertex surrounded by $\xi$). 

\FIGc{triangle}{Edges $a,b,c$ opposite to corner arcs $\al, \beta,\gamma$}{2.5cm}

For an anti-symmetric $n\times n$ matrix $A = (a_{ij})_{i,j=1}^n$ the {\em quantum torus} associated to $A$ is the algebra with presentation
$$  \cR \la x_i^{\pm 1}, i= 1, \dots n\ra  /( x_i x_j = q^{a_{ij}} x_j x_i).$$
For basic properties of quantum tori see for example \cite[Section 2]{Le:QT}.

To the triangle $\P_3$ we associate the quantum torus $\bT$ with presentation
$$\bT := \cR \la \al^{\pm 1}, \beta^{\pm 1}, \gamma^{\pm 1} \ra /( q\al \beta  =  \beta \al, q\beta \gamma =  \gamma\beta, q\gamma \al =  \gamma \al).
$$

\def\SPt{\cS(\P_3)}
\def\bSPt{\bcS(\P_3)}
The cyclic group $\BZ/3= \la \tau \mid \tau^3=1\ra$ acts by algebra automorphisms on each of the algebras $\cS(\P_3)$, $\bcS(\P_3)$, and $\bT$ as follows. In short  $\tau$ is rotation $\tau$  by $2\pi/3$ counterclockwise about the center of the triangle. This rotation induces the algebra automorphism $\tau$ of $\SPt$; it also induces the algebra automorphism $\tau$ of $\bSPt$. On $\bT$ and  $\tau$ is given by
\begin{align*}
\tau(\al) =\beta,\ \tau(\beta) =\gamma, \  \tau(\gamma) = \al.
\end{align*}

\begin{theorem} \label{teo:redtri}
The reduced skein algebra $\bcS(\P_3)$ of the ideal  triangle is isomorphic to the quantum torus $\bT$. The isomorphism is $\BZ/3$-equivariant and given by 
\be  \al(++) \to \al, \quad \al(+-) \to q^{-\frac{1}{2}}\gamma \beta, \quad \al(-+) \to 0, \quad \al(--) \to \al^{-1}.
\label{eq.iso66}
\ee
\end{theorem}   
\begin{proof}
By \cite[Theorem 4.6]{Le:TDEC} the algebra  $\Ss(\P_3)$ is generated by $$X=\{ \al({\nu,\nu')},\beta({\nu,\nu'}), \gamma({\nu,\nu'})  \mid \nu, \nu'\in \{\pm \}\}$$ subject to the following relations and their images under $\tau$ and $\tau^2$:
\begin{align}
\label{rel1} \beta(\mu,\nu)\, \al(\mu',\nu')& = q \al(\nu, \nu') \,  \beta(\mu,\mu') -q^2 C^{\nu}_{\mu'} \, \gamma(\nu', \mu)
\\
\label{rel2} \al(-,\nu)\, \al(+,\nu')& = q^2 \al(+,\nu) \,\al(-,\nu') - q^{5/2} C^\nu_{\nu'}
\\
\label{rel2b} \al(\nu,-)\, \al(\nu',+)& = q^2 \al(\nu,+) \,\al(\nu',-) - q^{5/2} C^\nu_{\nu'}
\\
\label{rel3} \al(-,\nu)\, \beta(\nu',+)& = q^2 \al(+,\nu)\, \beta(\nu',-) - q^{5/2} \gamma(\nu,\nu')
\\
\label{rel4} \al(\nu,-)\, \gamma(+, \nu')& = q^2 \al(\nu,+)\, \gamma(-,\nu') + q^{-1/2} \beta(\nu',\nu).
\end{align}
As the only bad arcs are $\alpha(-,+), \beta(-,+),\gamma(-,+)$, the quotient $\bSPt$ is obtained by adding the relations $\alpha(-,+)= \beta(-,+)=\gamma(-,+)=0$, and from this presentation one can check that the map given by \eqref{eq.iso66} and its images under the action of $\BZ/3$ is an isomorphism.

Here is an alternative, more geometric proof. 
First in $\bSPt$  we have 
\be  \al(++) \, \al(--) =1 , \quad  \beta(++)\, \al(++) = q\,  \al(++)\,  \beta(++), \label{eq.rel55} \ee
and all its images under $\BZ/3$. In fact the first identity follows from Proposition \ref{r.corner} and the second follows from the height exchange identity \eqref{eq.reor1}.
It follows that the $\BZ/3$-equivariant map $f: \bT \to \bSPt$ given by
$$ f(\al)= \al(++),\ f(\al^{-1})= \al(--), \ \text{and images under $\BZ/3$},$$
gives a well-defined algebra homomorphism, as all the defining relations of $\bT$ are preserved under $f$. In $\bSPt$ we have
\be 
\gamma(+-) = q^{-1/2}\, \beta(++)\,  \gamma(--),
\ee
which follows from the identity in Figure \ref{fig:brten6-new} (where the left hand arc is stated to become $\al(+-)$). Thus all elements in the generator set  $X$ are in the image of $f$. This shows that $f$ is surjective. 

Let us show $f$ is injective. The set $\{ \al^k \beta^m \gamma^n \mid k,m,n,\in \BZ\}$ is an $\cR$-basis of $\bT$. Assume that there is a finite set $S\subset \BZ^3$ such that
\be  f\left (\sum_{(k,m,n) \in S} c_{k,m,n} \al^k \beta^m \gamma^n\right) =0, \ c_{k,m,n} \in \cR.
\label{eq.lu20}
\ee
 Multiplying the identity \eqref{eq.lu20} on the left by $f(\al^{k'} \beta^{m'} \gamma^{n'})$ with large $k', m', n'$ and using the $q$-commutations between $\al, \beta,\gamma$ we can assume that $k,m,n > 0$ in \eqref{eq.lu20}. For each $(k,m,n)\in \BN^3$ let $z(k,m,n)$ be the stated  simple tangle diagram consisting of $k$ arcs parallel to $\al$, $m$ arcs parallel to $\beta$, and $n$ arcs parallel to $\gamma$, with all state positive. Note that $z(k,m,n) \in \bB(\P_3)$.
Clearly the map $z:\BN^3 \to \bB(\P_3)$ is injective. 
As the diagram of $f(\al^k \beta^m \gamma^n)$ can be obtained from $z(k,m,n)$ by a sequence of height change moves of positively stated endpoints, the first identity of \eqref{eq.reor1} shows that 
$$ f(\al^k \beta^m \gamma^n)= q^{g(k,m,n)}\,  z(k,m,n)$$
for some $g(k,m,n) \in \BZ$. From \eqref{eq.lu2} we get
$$ \sum_{(k,m,n) \in S} c_{k,m,n} q ^{g(k,m,n)}\,  z(k,m,n) =0.$$
As $z(k,m,n)$ are distinct elements of the basis $\bB(\P_3)$, this forces all $c_{k,m,n}=0$. Hence $f$ is injective.

\end{proof}                                                                                                                      
\def\faces{\mathrm{\mathcal F(\cE)}} 
                                      
\subsection{The quantum trace map}  \label{sec.rTrace} Assume that $\fS$ is {\em triangulable}, i.e. $\fS$   is  not one of the following: a monogon, a bigon, a sphere with one or two punctures. 
{\em A triangulation} $\cE$ of $\fS$ is a collection consisting of all boundary edges and several ideal arcs in the interior of $\fS$ such that\\
(i) no two arcs in $\cE$ intersect and no two are isotopic, and\\
(ii) if $a$ is an ideal arc not intersecting any ideal arc in $\cE$ then $a$ is isotopic to one in $\cE$.

It is known that if $\fS$ is triangulable, then by splitting $\fS$ along all interior ideal arcs in $\cE$ we get a collection $\faces$ of ideal triangles.
By the splitting theorem, we get an algebra embedding of $\Ssr$ into a quantum torus
$$ \Theta: \Ssr \to \bigotimes_{\faces} \bT.$$

In addition to the quantum torus $\bT$ we associate the quantum torus $\bT'$ to the standard ideal triangle $\P_3$:
\begin{align}
 \bT' : = \cR \la a^{\pm 1}, b^{\pm 1}, c^{\pm 1} \ra /( qab  =  ba, qbc =  cb, qca =   ac).
\end{align}
One should think of $a,b,c$ as the edges opposite to $\al,\beta,\gamma$, see Figure \ref{fig:triangle}. 

\def\SPt{\cS(\P_3)}
\def\bSPt{\bcS(\P_3)}
The cyclic group $\BZ/3= \la \tau \mid \tau^3=1\ra$ acts by algebra automorphisms on  $\bT'$  by
\begin{align*}
\tau(a) =b, \ \tau(b) =c, \ \tau(c) = a.
\end{align*}

There is a $\BZ/3$-equivariant algebra embedding 
$\bT \embed \bT'$, defined by
$$ \al \to q^{1/2} bc , \quad \beta \embed q^{1/2} ca , \quad \gamma \to q^{1/2} ab.$$

Consider  the composition $\tr_q$ given by
$$ \tr_q: \Ssr \overset \Theta \embed  \bigotimes_{{\faces}} \bT \embed \bigotimes_{{\faces}} \bT'.$$

On the collection of all edges of all the triangles in $\faces$ define the equivalence relation \red{such that} $a' \cong a''$ if \red{$a'$ and $a''$} are glued together in the triangulation. Then the set of equivalence classes is canonically isomorphic to $\cE$.
Let $\cY(\cE)$ be the subalgebra of $\bigotimes_{\faces} \bT'$ generated by all $a'\otimes a''$ with $a\cong a'$ and all boundary edges (each boundary edge is equivalent only to itself). It is easy to see that the image of $\tr_q$ is in $\cY(\cE)$. Thus, $\tr_q$ restricts to
$$ \tr_q: \Ssr \to \cY(\cE).$$
 The algebra  $\cY(\cE)$ is a quantum torus, known as the Chekhov-Fock algebra associated to a triangulation $\cE$ of $\fS$, see \cite{BW,CF,Le:TDEC}.
 \red{ 
 The quantum trace map of Bonahon and Wong is an algebra homomorphism $\widehat \tr_q: \widehat \cS(\fS)\to \cY(\cE)$, where $\widehat \cS(\fS)$ is the coarser version of $\cS(\fS)$ defined using only \eqref{eq.skein} and \eqref{eq.loop}, see Subsection \ref{sub:statedskeinalgebra}.

\begin{theorem}\label{thm.qtrace} 
If $\cE$ is a triangulation of $\fS$ then the algebra embedding 
$ \tr_q: \Ssr \embed \cY(\cE)$ is a refinement of the 
the quantum trace map of Bonahon and Wong  in the sense that $\widehat \tr_q$ is the composition $ \widehat \cS(\fS) \onto \Ssr \overset {\tr_q} \embed \cY.$
\end{theorem} 
\begin{proof}  In \cite{Le:TDEC} an algebra homomorphism $\varkappa_\cE: \cS(\fS) \to \cY(\cE)$ is defined as the composition
$$ \SS \embed  \bigotimes_{{\faces}} \cS( \P_3 )  \embed \bigotimes_{{\faces}} \bT',$$
where the map from $\cS(\P_3)$ to $\bT'$  is exactly the composition $\cS(\P_3) \to \overline{\cS}(\P_3)\to \bT'$. It follows that $\varkappa_\cE$ is the composition $\cS(\fS) \onto \Ssr \overset {\tr_q} \embed \cY$. In \cite{Le:TDEC} it is proved that $\widehat \tr_q$ is the composition $ \widehat \cS(\fS) \onto \cS(\fS)  \overset {\varkappa_\cE} \to \cY$. Hence  $\widehat \tr_q$ is also the composition $ \widehat \cS(\fS) \onto \Ssr  \overset {\tr_q} \embed \cY$.
\end{proof}

Besides giving another proof of the existence of the quantum trace map, Theorem \ref{thm.qtrace} shows that the kernel of $\widehat\tr_q$ is the ideal generated by relations \eqref{eq.arcs}, \eqref{eq.order}, and the ideal $\CI^{\mathrm{bad}}$.
}

\subsection{Co/module structure for $\Ssr$}     The Hopf algebra structure of $\cS(\cB)$ descends to a Hopf algebra of $\bcS(\cB)$. We identify $\bcS(\cB)\equiv \cR[x^{\pm 1}]$ using the isomorphism of  Proposition \ref{r.rbigon}.  Then  $\Delta(x) = x \otimes x$ and $\epsilon (x)=1$.
                    
Arguing exactly as in Subsection \ref{sub:comodule}, one sees that for each surface $\fS$ and each edge $e$ of $\fS$, the algebra $\Ssr$ has both a left and a right $\cR[x^{\pm1}]$-comodule algebra structure (which is equivalent to a $\mathbb{Z}$-valued grading counting the number of $+$ and $-$ states of each skein along $e$): 
\begin{proposition}\label{prop:reducedcomodule} (a) The map $\D_e: \Ssr \to  \Ssr \otimes \bcS(\cB)$ gives $\Ssr$ a right comodule-algebra structure over the Hopf algebra $\cR[x^{\pm1}]$. Similarly $\eD$ gives gives $\Ssr$ a left comodule-algebra structure over the Hopf algebra $\cR[x^{\pm1}]$. 
 
(b)  If $e_1,e_2$ are two distinct boundary edges, the coactions on the two edges commute, i.e.  for instance $$(\Delta_{e_2}\otimes \id)\circ \Delta_{e_1}=(\Delta_{e_1}\otimes \id)\circ \Delta_{e_2}.$$
\end{proposition}

In the reduced setting, though, Theorem \ref{teo:cotensor} does no longer hold:
indeed, with the notation in there, if $\beta\in \Bred(\fS)$ is a basis element intersecting a cutting edge exactly once, then its image $\theta(\beta)$ under the cutting morphism is $\theta(\beta)=\beta'_{++}+\beta'_{--}$ where $\beta'_{++},\beta'_{--}\in \Bred(\fS')$ are identical except for their states on $\beta'\cap (c_1\cup c_2)$. But it is not difficult to check that $\beta'_{++}$ is balanced: $$\Delta_{c_1}(\beta'_{++})=\beta'_{++}\otimes \alpha_{++}+\beta'_{+-}\otimes \alpha_{-+}=\beta'_{++}\otimes \alpha_{++}={}_{c_2}\Delta(\beta'_{++})$$ because the class of $\alpha_{-+}=\alpha_{+-}=0\in \red{\Ss(\B)}$, still $\beta'_{++}$ is not in the image of $\theta$.

  \def\OO{\mathbb O}
\def\OOe{\sqrt{\OO_e}}                                                      
\section{The classical case: twisted bundles}
In this section we will suppose that $\fS$ is a connected, oriented surface with a non-empty set of boundary edges and let $\ori$ be the positive orientation of $\partial \fS$ i.e. that induced by the orientation of $\fS$. We will prove that if $q^{\frac{1}{2}}=1$ then $\cSs(\fS)$ is isomorphic to the algebra of regular functions on the affine variety of ``twisted bundles'' on $\fS$. \red{A} similar result for the case when $\partial \fS=\emptyset$ is well known (see for instance \cite{Thurston}). 

Fix an arbitrary  Riemannian metric and let $U\fS$ be the unit tangent bundle over $\fS$, with the canonical projection
 $\pi:U\fS\to \fS$.
A point in $U\fS$ is  a pair $(p,v)$,  where $p\in \fS$, $v\in T_p\fS, \|v\|=1$. For each immersion $\alpha:[0,1]\to \fS$ its {\em canonical lift} is \red{the} path $(\alpha(t),\frac{\dot{\alpha}(t)}{\| \dot{\alpha}(t))\|})$ in $U\fS$. In particular, since each edge $e$ of $\partial \fS$ is oriented by $\ori$, it has a canonical lift $\widetilde{e}\subset \partial U\fS$; we will denote $\widetilde{\partial \fS}:=\cup_{e\subset \partial \fS} \tilde{e}$.  If we let $-e$ be the edge oriented in the opposite way, then we get a different lift which we will denote $(-e)^\sim $. Let  $-\widetilde{\partial \fS}=\cup_{e\subset \partial \fS} (-e)^\sim$ and $\pm \widetilde{\partial \fS}=\widetilde{\partial \fS}\cup -\widetilde{\partial \fS}$.

 For a point $x\in \fS$ the fiber $\OO=\pi^{-1} (x)$ is a circle, and we will orient it  according to the orientation of $\fS$. It is clear that the \red{free} homotopy class of $\OO$ does not depend on $x$. 
 
 For each boundary edge $e$ choose a point $x\in e$.  
 Let $v\in T_x(e)$ be the unit tangent vector with orientation $\ori$. Then both $(x,v)$ and $(x,-v)$ are in $\pi^{-1}(x)$, and the half circle of $\pi^{-1}(x)$ going from $(x,v)$ to $(x,-v)$ in the positive direction is denoted by $\OOe$. The exact position of $x$ on $e$ will not be important in what follows. 

\begin{definition}[Fundamental Groupoids]
Let $X$ be a path connected topological space and $\{E_i\}_{i\in I}$ disjoint contractible subspaces of $X$.  The fundamental groupoid $\pi_1(X,\{E_i\}_{i\in I})$ is the groupoid (i.e. a category with invertible morphisms) whose objects are $\{E_i, i\in I\}$ and whose morphisms are the homotopy classes of oriented paths in $X$ with endpoints in $\cup_{i\in I} E_i$.
A morphism of groupoids if a functor of the corresponding categories. 
\end{definition}
Recall that a group is a groupoid with only one object.
\begin{lemma}[Extension of morphisms]\label{lem:extend}
With the above notation, let $E\subset X$ be a contractible subspace disjoint from $\cup_{i\in I} E_i$. Then given a morphism $\rho:\pi_1(X,\{E_i\}_{i\in I})\to G$ for some group $G$, an oriented path $\gamma$ connecting some $E_i$ to $E$, and an arbitrary $g\in G$ there is a unique extension $\rho':\pi_1(X,\{E_i\}_{i\in I}\cup \{E\})\to G$ of $\rho$ such that  $\rho'(\gamma)=g$ and  $\rho'(\alpha)=\rho(\alpha)$ for all $\alpha\in Mor(E_i,E_j)$ for some $i,j$. 
\end{lemma}
\red{
\begin{proof}
We sketch a proof. Let $\beta$ be a homotopy class of an oriented path connecting $E$ to some $E_j$; write $\beta=(\beta\circ \gamma)\circ \gamma^{-1}$. Observe that $\beta\circ \gamma\in Mor(E_i,E_j)$ thus $\rho$ is defined on it; hence define $\rho'(\beta)=\rho(\beta\circ \gamma)\cdot g^{-1}$. Similarly if $\beta$ is a path connecting $E_j$ to $E$ define $\rho'(\beta)=g\cdot \rho(\gamma^{-1}\circ \beta)$. Finally if $\beta$ is an endomorphism of $E$ define $\rho'(\beta)=g\cdot \rho(\gamma\circ \beta\circ \gamma^{-1})\cdot g^{-1}$. We leave to the reader to verify that this is indeed a functor with the required properties. 
\end{proof}
}

We shall be interested in two particular groupoids:  $\pi_1(\fS,\partial \fS)$ and $\pi_1(U\fS, \widetilde{\partial \fS})$. Note that $\pi:(U\fS,\widetilde{\partial \fS})\to (\fS,\partial \fS)$ induces a surjective morphism $\pi_*$ of groupoids. 

\begin{definition}[\red{Flat twisted $\mathrm{SL}_2(\mathbb{C})$-bundle}]\label{def:twbundle}
A \red{flat twisted $\mathrm{SL}_2(\mathbb{C})$-bundle} \red{(``twisted bundle'' in what follows to keep notation short)} on $\fS$ is a morphism $\rho:\pi_1(U\fS;\widetilde{\partial \fS})\to \mathrm{SL}_2(\mathbb{C})$ such that $\rho(\OO)=-Id$.
\end{definition}
By Lemma \ref{lem:extend} we extend $\rho$ to a morphism (with the same notation) $\rho:\pi_1(U\fS;\pm \widetilde{\partial \fS})\to \mathrm{SL}_2(\mathbb{C})$ such that for every boundary edge $e$,
\begin{equation}\label{eq:OO}
\rho(\OOe)=\left(\begin{array}{cc}
0 & -1\\
1 & 0\end{array}\right).
\end{equation}
Since $\fS$ is not a closed surface, its fundamental group $\pi_1(\fS)$ is a free group.
\begin{lemma} \label{lem:twistedbundles}
Suppose that $\partial \fS\neq \emptyset$. Then the set $tw(\fS)$ of twisted bundles on $\fS$ is the affine algebraic variety $\mathrm{SL}_2(\mathbb{C})^{n+k}$ where $$n=-1+\#\{e\subset \partial \fS\}\ {\rm and}\ k=\mathrm{rank}(\pi_1(\fS)).$$ 
In particular the algebra $\chi(\fS)$ of its regular functions is generated by the matrix entries of each of the copies of $\mathrm{SL}_2(\mathbb{C})$.  
\end{lemma}
\begin{proof}
\red{Since $U\fS$ is trivial, the fundamental groupoids $\pi_1(\fS;\partial \fS)$ and $\pi_1(U\fS,\widetilde{\partial \fS})$ are isomorphic. More explicitly,} we claim that there are non canonical injective morphisms of fundamental groupoids $s_*:\pi_1(\fS;\partial \fS)\to \pi_1(U\fS,\widetilde{\partial \fS})$.
To build one, pick any \red{non-zero} vector field on $\fS$ which is positively tangent to the edges of $\partial \fS$: it exists because we are not prescribing its behavior near the (non compact) cusps. This trivializes $U\fS$ as $\fS\times S^{1}$; let $s:\fS\to \fS\times\{1\}$ be a section of $\pi:U\fS\to \fS$. 

\red{The above isomorphism allows to provide an isomorphism from the set of twisted bundles to morphisms from $\pi_1(\fS;\partial \fS) \to \mathrm{SL}_2(\BC)$. Indeed} to each twisted bundle $\rho:\pi_1(U\fS;\widetilde{\partial \fS})\to \mathrm{SL}_2(\BC)$ we associate $\rho': \pi_1(\fS;\partial \fS) \to \mathrm{SL}_2(\BC)$ defined as $\rho'=\rho\circ s_*$ . Reciprocally given $\rho':\pi_1(\fS;\partial \fS) \to \mathrm{SL}_2(\BC)$ we extend it to $\rho:\pi_1(U\fS;\widetilde{\partial \fS})\to \mathrm{SL}_2(\BC)$ by setting $\rho(\OO)=-Id$ and $\rho|_{\pi_1(\fS\times\{1\})}=\rho'.$

\red{To conclude we now argue that the set of morphisms $\rho': \pi_1(\fS;\partial \fS) \to \mathrm{SL}_2(\BC)$ is in bijection with $\mathrm{SL}_2(\mathbb{C})^{n+k}$. Indeed} fix a set of immersed smooth paths $\alpha_1,\ldots ,\alpha_n\subset \fS$ connecting a fixed edge $e_0\subset \partial \fS$ to each other edge of $\partial \fS$ as well as a set of paths whose endpoints are in $e_0$ representing generators $g_1,\ldots , g_k$ of $\pi_1(\fS;e_0)$ (which is free because $\partial \fS\neq \emptyset)$. \red{Since the fundamental group of $\fS$ is free, the list of values $(\rho'(\alpha_1),\ldots, \rho'(\alpha_n),\rho'(g_1)\ldots, \rho'(g_k))\in \mathrm{SL}_2(\mathbb{C})^{n+k}$} provides the sought non canonical \red{bijection.}
\end{proof}
\begin{example}\label{ex:twpolygon}
Let $\P_n$ be the $n$-polygon with vertices numbered in the orientation sense from $0$ to $n-1$; then $tw(\P_n)=\mathrm{SL}_2(\mathbb{C})^{n-1}$ where the $n-1$ matrices are given by the holonomies of the diagonals connecting the edge $v_0v_1$ to each other edge. Then $\chi(\P_n)= O(SL_2)^{\otimes n-1}$ and in particular $\chi(\B)=O(SL_2)$ and $\chi(\P_3)=O(SL_2)\otimes O(SL_2)$ (where by ``equal'' we mean ``non-canonically isomorphic to'').
\end{example}
\begin{remark} The notion of \red{flat twisted $\mathrm{SL}_2(\mathbb{C})$-bundle} is closely related to the one considered in \cite{Thurston}.
\end{remark}
\subsection{Trace functions for non oriented curves} We will identify the states of a stated tangles with \red{vectors} in $\BC^2$ as follows:
 $$+:=\left(\begin{array}{c}1\\0\end{array}\right), -:=\left(\begin{array}{c}0\\1\end{array}\right).$$
 If $\vec{x},\vec{y}\in \mathbb{C}^2$ let $\det(\vec{x}|\vec{y})$ denote the determinant of the matrix whose first column is $\vec{x}$ and second is $\vec{y}$. 

We will say an immersion  $a:[0,1]\to \fS$ is in {\em good position} if $a(0), a(1) \in\partial \fS$ 
and the tangent vectors $\dot a(0), \dot a(1)$ \red{are positively tangent to $\partial \fS$.}

An immersion  $\al:[0,1]\to \fS$ is {\em transversal} if $\al(0), \al(1) \in \pfS$ and $\al$ is transversal to $\pfS$ at $0$ and $1$. One can bring such a transversal 
  $\al$ to an arc  $a$ in good position by an isotopy (relative $0$ and $1$) in a small neighborhood of $\al(0)$ and $\al(1)$.  The canonical lift of $a$ will be denoted by $\hat \al$ and is called the {\em good lift} of $\al$. Note that the homotopy class of $\hat \al$ is uniquely determined by $\al$, and we will consider $\hat \al$ as an element of $\pi_1(U\fS;\widetilde{\partial \fS})$.
Note that the good lift of the inverse path  $\alpha^{-1}$, defined by $\al^{-1}(t)=\alpha(1-t)$,  is not the inverse of $\hat \al$, since before    lifting one has to isotope $\alpha^{-1}$ to \red{a} good position. 
\def\sign{\mathrm{sign}}

A {\em stated transversal immersion} is a transversal immersion whose end points are stated~$\{\pm\}$.
\begin{definition}[Trace]\label{def:traces} Let $\rho$ be a twisted bundle on $\fS$. 

Assume $\alpha:[0,1]\to \fS$ is a stated transversal immersion with state  $\ve$ at $\al(0)$ and $\eta$ at $\al(1)$.  Define the trace of $\alpha$ by
 $$\tr(\alpha):=\det(\eta| \rho(\hat {\alpha})\cdot \epsilon).$$
 
Assume $\beta:[0,1]\to \fS$ is an immersed closed curve (i.e. $\beta(0)=\beta(1)$ and $\beta$ has the same tangent at $0$ and $1$). Define the trace  of $\beta$ by
$$   \tr(\beta)=\tr(\rho(\widetilde {\beta'})),$$
where $\beta'$ is any smooth closed curve isotopic to $\beta$ such that $\beta'(0)\in \pfS$.
\end{definition}
In the first case if $\alpha'$ is homotopic to $\alpha$ through stated transversal immersions and has the same states as $\alpha$ then $\tr(\alpha')=\tr(\alpha)$ (indeed the homotopy lifts to a homotopy of $\widehat{\alpha}$ and $\widehat{\alpha'}$). 
In the second case it is easy to see that $\tr(\beta)$ does not depend on the choice of $\beta'$, as the images under $\rho$ of any two such $\beta'$ are conjugate in $\mathrm{SL}_2(\BC)$ and hence have the same trace.
\begin{example}\label{ex:traces}
Let $\alpha:[0,1]\to \fS$ be a stated transversal immersion with state  $\ve$ at $\al(0)$ and $\eta$ at $\al(1)$.
$${\rm If}\ \rho(\hat{\alpha})=\left(\begin{array}{cc}a & b \\ c & d\end{array}\right), {\rm then} \tr(\al)=\det(\eta,\rho(\alpha)\cdot\epsilon) =
\begin{tabular}{c|cc}
$\eta\backslash \epsilon $& + & -\\
\hline
+ & c & d\\
- & -a & -b\\
\end{tabular}.$$
\red{We} remark that the matrix on the right, expressing the values of the traces for an immersed transverse stated arc $\alpha$ is $\rho(\sqrt{\OO}^{-1})\rho(\hat{\alpha})$ (see Equation \eqref{eq:OO}). 
\end{example}
\begin{remark}  The notion of trace here is similar to the one introduced in \cite{MS}, where trace is defined only for oriented arcs. The novelty here is the good lift, which is used to define traces for unoriented arcs and the use of twisted bundles as representations of fundamental groupoids.
\end{remark}

When $\al$ is stated, we provide $\al^{-1}$ with states so that the state of  $\al^{-1}(t)$ is equal to the state of $\al(1-t)$ for $t=0,1$.
\begin{lemma}\label{lem:lifts} Suppose $\rho$ is a twisted bundle on $\fS$.
(a) Let $\al$ be a stated transversal immersion.  One has
$$\rho (\widehat{\alpha^{-1}})=- \rho (\hat{\alpha})^{-1}.$$ 
 As a consequence, $\tr(\alpha)=\tr(\alpha^{-1})$.
 
 (b) Let $\beta:[0,1]\to \fS$ be an immersed closed curve such that $\beta(0)\in \pfS$. Then
  $$\rho (\widetilde{\beta})^{-1}=\rho (\widetilde{\beta^{-1}}).$$ 
  
 As a consequence, if $\gamma$ is any immersed closed curve then $\tr(\gamma) = \tr(\gamma^{-1})$.
\end{lemma}
\begin{proof} (a). A direct inspection shows that the homotopy class of the closed simple loop in $U\fS$ given by the concatenation $\widehat{\alpha^{-1}}\circ \hat{\alpha}$ is $\OO$ (see the left hand side of Figure \ref{fig:lift}). 
The first equality follows as by definition $\rho$ is a functor such that $\rho(\OO)=-Id$. 

To prove that $\tr(\alpha)=\tr(\alpha^{-1})$, 

we compute the traces using the notation and content of Example \ref{ex:traces}:
$${\rm If}\ \rho(\widehat{\alpha^{-1}})=-\rho(\hat{\al})^{-1}=\left(\begin{array}{cc}-d & b \\ c & -a\end{array}\right), {\rm then} \tr(\al)=\det(\eta',\rho(\alpha)\cdot\epsilon') =
\begin{tabular}{c|cc}
$\eta'\backslash \epsilon' $& + & -\\
\hline
+ & c & -a\\
- & d & -b\\
\end{tabular}.$$
But since the state at $\alpha^{-1}(0)$ is $\eta$ and that at $\alpha^{-1}(1)$ is $\epsilon$, we get the claim in this case by directly comparing with the transpose of the matrix of values provided in Example \ref{ex:traces}. 

(b). 
Observe that if $\beta$ is the black curve depicted in the r.h.s. of Figure \ref{fig:lift}, then $\beta^{-1}$ is regularly homotopic to the dotted curve $\beta'$ in the same picture. By construction $\dot{\beta'}(0)=\dot{\beta}(0)$ and $\beta'\circ \beta$ is regularly homotopic to an eight-shaped immersed curve in a disc. 
Then $\widetilde{\beta^{-1}}$ is homotopic to $\tilde{\beta'}$ in $U\fS$ and it holds $\rho(\tilde{\beta'}\circ \tilde{\beta})=Id$, thus $\rho(\widetilde{\beta^{-1}})=\rho(\tilde{\beta})^{-1}$.
The last statement now follows because $\rho(\tilde{\beta})\in \mathrm{SL}_2(\BC)$ so that $\tr(\rho(\tilde{\beta}))=\tr(\rho(\tilde{\beta})^{-1})$. 

\end{proof}
\begin{figure}
\includegraphics[width=4.2cm]{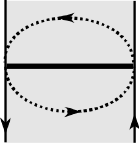}\qquad \includegraphics[width=4cm]{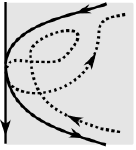}
\caption{On the left the unoriented horizontal arc can be lifted to $U\fS$ in two different ways (dotted), depending on the choice of an orientation; their composition is homotopic in $U\fS$ to $\OO$.  On the right we exhibit a smooth oriented curve $\beta$ (solid) and a curve $\beta'$ (dotted) which is regularly homotopic to $\beta^{-1}$ . The composition of the two is a nullhomotopic $8$-shaped loop in $U\fS$.}\label{fig:lift}
\end{figure}
Suppose $\alpha=\cup \alpha_i$, where each $\al_i$ is either a stated transversal non-oriented arc or a non-oriented immersed closed curved. Define
$$\tr(\alpha):=\prod_{i} \tr(\alpha_i).$$

\begin{lemma}\label{lem:classicalbigon}
If $q^{\frac{1}{2}}=1$ the map $\tr:\cSs(\B)\to \chi(\B)$ sending a stated skein to its trace is an isomorphism of algebras. The same  holds for $\tr:\cSs(\P_n)\to \chi(\P_n)$ for every $n\geq 2$. 
\end{lemma}
\begin{proof}
Applying Theorem \ref{teo:SL2} at $q^{\frac{1}{2}}=1$ we get an explicit algebra isomorphism $\phi:\cSs(\B)\to O(SL_2)$; by Example \ref{ex:twpolygon}, we know that $\chi(\B)$ is isomorphic to $O(SL_2)$. Furthermore, by  Example \ref{ex:traces}, the map $\tr$ is an algebra isomorphism. 
One argues similarly for $\P_n$: applying Corollary \ref{cor:polygons} at $q^{\frac{1}{2}}=1$ we get an explicit algebra isomorphism $\phi:\cSs(\P_n)\to O(SL_2)^{\otimes n-1}$; by the proof of Corollary \ref{cor:polygons} a system of algebra generators of $\cSs(\P_n)$ is easily seen to be the arcs connecting the edge $e_0$ to each other edge and stated arbitrarily. By Example \ref{ex:traces}, the map $\tr$ on these generators provides a system of generators of $O(SL_2)^{\otimes n-1}$ which by Example \ref{ex:twpolygon} is isomorphic to $\chi(\P_n)$. 
\end{proof}
\subsection{Splitting theorem for trace functions}

In all this subsection, let $c\subset \fS$ be an ideal arc oriented arbitrarily, let $\fS'$ be the result of cutting $\fS$ along $c$ and let $\pr:\fS'\to \fS$ be the projection and $\tpr:U\fS'\to U\fS$ the projection induced on the unit tangent bundles.  Let $\pr^{-1}(c)=c_1\cup c_2\subset \partial \fS'$ so that $c_1$ has the positive orientation and $c_2$ the negative one with respect to the orientation induced by that of $\fS'$ on the boundary. 
For each $c_i$ let $\widetilde{c_i}\subset U\fS'$ be its canonical lift and $(-c_i)^{\sim}$ the canonical lift of $-c_i$. Similarly let $\widetilde{c}$ be the canonical lift of $c$ in $U\fS$ and $(-c)^{\sim}$ be the canonical lift of $-c$.

\begin{lemma}\label{lem:decomposition}
Each $[\alpha]\in \pi_1(U\fS;\widetilde{\partial \fS})$ can be written as a composition $[\alpha_k]\circ [\alpha_{k-1}]\circ \cdots \circ [\alpha_1]$ of homotopy classes of immersed paths $\alpha_i:[0,1]\to U\fS$ such that $\alpha_i(\{0,1\})\subset \widetilde{\partial \fS}\cup \tilde{c}$ and $\alpha_i\cap \pi^{-1}(c)=\partial \alpha_i\cap \widetilde{c}$. 
Such a decomposition is unique up to insertion/deletions of compositions  $[\alpha']\circ [\alpha'^{-1}]$ for $[\alpha']\in \pi_1(U\fS;\widetilde{c})$ and replacement of $[\alpha_{j}]$ by $[\alpha'''_{j}]\circ [\alpha''_j]\circ[\alpha'_j]$ for some $[\alpha'_j],[\alpha'''_j]\in \pi_1(U\fS;\widetilde{\partial \fS}\cup \widetilde{c})$ and $[\alpha_j'']\in  \pi_1(U\fS;\widetilde{c})$ such that $[\alpha_j]=[\alpha'''_{j}]\circ [\alpha''_j]\circ[\alpha'_j]$ (or reciprocally).
\end{lemma}
\begin{proof}
Observe that $\pi^{-1}(c)\subset U\fS$ is homeomorphic to an annulus $A=\mathbb{R}\times S^1$ so that $\widetilde{c}=\mathbb{R}\times\{1\}.$ Represent the class $[\alpha]$ by a smooth curve $\alpha:[0,1]\to U\fS$ so that it is transverse to $A$; then homotope it so that it intersects $A$ exactly along $\widetilde{c}$: this provides an instance of the claimed splitting. If $\alpha':([0,1],\{0,1\})\to (U\fS,\widetilde{\partial \fS})$ is another smooth representative of the same class intersecting $A$ exactly along $\widetilde{c}$, let $h(t,s):[0,1]\times [0,1]\to U\fS$ be a smooth homotopy between $\alpha$ and $\alpha'$ which is transverse to $A$. Then $h^{-1}(A)$ is a disjoint union of arcs and circles embedded in $[0,1]\times [0,1]$ with boundary in $[0,1]\times \{0,1\}$ containing a finite number of maxima and minima with respect to the height function given by the second coordinate $s$. Pick a finite number of heights $s_0=0<s_1<\ldots <s_n=1$ so that each strip $[0,1]\times [s_i,s_{i+1}]$ contains at most one maximum or minimum of the diagram of $h^{-1}(A)$.  
Each immersed path $\alpha_{s_i}(t):=h(t,s_i):[0,1]\to U\fS$ intersects $A$ transversally a finite number of times and we can then modify $h$ locally around $h^{-1}(A)\cap ([0,1]\times \{s_i\})$ without inserting new maxima and minima so that $\alpha_{s_i}(t)$ intersects $A$ only along $\widetilde{c}$. 
Then the homotopies $h|_{[s_i,s_{i+1}]}$ transform the immersed path $\alpha_{s_i}$ into $\alpha_{s_{i+1}}$ by the moves described in the thesis: passing through a minimum replaces a smooth curve $\alpha$ with a composition $\alpha'\circ \alpha''\circ \alpha'''$ were all of $\alpha,\alpha',\alpha'',\alpha'''$ intersect $A$ only along $\widetilde{c}$ and in their boundary; passing through a maximum has the converse effect. Finally a strip containing no maxima and minima corresponds to a finite number of moves consisting in rewriting $\alpha\circ \alpha'$ with $\alpha\circ \beta\circ \beta^{-1}\circ \alpha'$ where $\beta\in \pi_1(U\fS;\widetilde{c})$ is the homotopy class represented by the restriction of $h$ to a ``vertical arc'' of $h^{-1}(A)\cap [0,1]\times [s_i,s_{i+1}]$ (i.e. an arc joining $[0,1]\times \{s_i\}$ and $[0,1]\times \{s_{i+1}\}$). 
\end{proof}

If $\rho':\pi_1(U\fS';\widetilde{\partial \fS'})\to \mathrm{SL}_2(\BC)$ is a twisted bundle, then by Lemma \ref{lem:extend} we can extend it to a twisted bundle $\rho'':\pi_1(U\fS';\widetilde{\partial \fS'}\cup (-c_2)^{\sim})\to \mathrm{SL}_2(\BC)$ by setting $\rho''(\sqrt{\OO_{c_2}})=\left(\begin{array}{cc} 0 & -1\\ 1 & 0\end{array}\right)$ where $\sqrt{\OO_{c_2}}$ is the path connecting $\widetilde{c_2}$ and $(-c_2)^{\sim}$ by following in the positive direction the fiber $\pi^{-1}(x)$ for some $x\in c_2$. 
\begin{proposition} \label{prop:algebrainjection}
There is a surjective map $i^*:tw(\fS')\to tw(\fS)$ defined as follows. 
Given $\alpha\in \pi_1(U\fS;\widetilde{\partial \fS})$, decompose it as $\alpha=\alpha_k\circ \alpha_{k-1}\circ \cdots \circ \alpha_1$ where each $\alpha_i\in \pi_1(U\fS;\widetilde{\partial \fS}\cup \widetilde{c})$ intersects $\pi^{-1}(c)$ at most in its endpoints and exactly along $\widetilde{c}$ (such a decomposition exists by Lemma \ref{lem:decomposition}). 
Then for each $\rho'\in tw(\fS')$ let 
$$i^*(\rho')(\alpha)=\rho''(\alpha'_k)\rho''(\alpha'_{k-1})\cdots \rho''(\alpha'_1)$$ where $\alpha'_i=\tpr^{-1}(\alpha_i)$ is the lift of $\alpha_i$ to $\pi_1(U\fS';\widetilde{\partial \fS'}\cup (-c_2)^{\sim})$.
Passing to the algebras $\chi(\fS)$ and $\chi(\fS')$ of regular functions on the algebraic varieties $tw(\fS)$ and $tw(\fS')$, $i^*$ induces an injective algebra morphism $i:\chi(\fS)\hookrightarrow \chi(\fS')$ which we will call the ``cutting morphism'' associated to $c$.
\end{proposition}
\begin{proof}
By Lemma \ref{lem:decomposition} to check that $i^*$ is well defined it is sufficient to check that for each $\alpha$ the choice of the decomposition does not affect the result of $i^*(\rho')(\alpha)$. But this is evident if we make an exchange $\alpha_2\circ \alpha_1\leftrightarrow \alpha_2\circ \alpha'\circ \alpha'^{-1}\circ \alpha_1$ or $\alpha \leftrightarrow \alpha'\circ \alpha''\circ \alpha'''$ as in the statement of Lemma \ref{lem:decomposition} because $\rho'$ is a functor.   

To prove surjectivity observe that by Lemma \ref{lem:extend} we can extend any morphism $\rho:\pi_1(U\fS;\widetilde{\partial \fS})\to \mathrm{SL}_2(\mathbb{C})$ to $\rho:\pi_1(U\fS;\widetilde{\partial \fS}\cup \widetilde{c}\cup (-c)^{\sim})\to \mathrm{SL}_2(\mathbb{C})$ setting in particular $$\rho(\sqrt{\OO_c})=\left(\begin{array}{cc} 0 & -1 \\ 1 & 0\end{array}\right).$$
Then if $\alpha'\in \pi_1(U\fS';\widetilde{\partial \fS'}\cup (-c_2)^{\sim})$ define $\rho': \pi_1(U\fS';\widetilde{\partial \fS'}\cup (-c_2)^{\sim})\to \mathrm{SL}_2(\BC)$ by $\rho'(\alpha')=\rho(\tpr_*(\alpha'))$ where  $\tpr_*:\pi_1(U\fS';\widetilde{\partial \fS'}\cup (-c_2)^{\sim})\to \pi_1(U\fS;\widetilde{\partial \fS}\cup (-c)^{\sim})$ is the morphism induced by the continuous map $\tpr:U\fS'\to U\fS$. Then letting $\overline{\rho}'$ be the restriction of $\rho'$ to $\pi_1(U\fS';\widetilde{\partial \fS'})$ we have that by construction $\rho=i^*(\overline{\rho}')$. 
Surjectivity of $i^*$ implies the injectivity of $i$.
\end{proof}

The following proposition tells us that the trace functions behave exactly as the skeins under cutting along an ideal arc (see Theorem \ref{thm.1a}).
 Suppose $\al$ is a stated transverse smooth simple curve intersecting transversally $c$. Then $\al':=\pr^{-1}(\al)$ is a transverse smooth simple curve 
 which is stated at every boundary point except for newly created boundary points, which are points in $\pr^{-1} (c)\cap \al'=(c_1\cup c_2)\cap \al'$.
  {\em A lift in $\fS'$} of $\al$ is a stated transverse smooth simple curve $\beta$ in $\fS'$ which is $\al'$ equipped with states on $\pr^{-1} (c\cap \al)$ such that if $x,y\in \pr^{-1} (c\cap \al)$ with $\pr(x)=\pr(y)$ then $x$ and $y$ have the same state. If $|c\cap \al|=k$, then  $\al$ has $2^k$ lifts in $\fS'$.

\begin{proposition}[Cutting trace functions]\label{prop:cuttrace}
Let $\alpha$ be a stated transverse smooth simple curve intersecting transversally $c$. Then
\begin{equation}\label{eq:cuttrace}
i(\tr(\alpha))=\sum \tr(\beta)
\end{equation}
 where the sum is taken on all the lifts in $\fS'$ of $\alpha$ (i.e. as in Theorem \ref{thm.1a}). Furthermore if $c'\subset \fS$ is another ideal arc disjoint from $c$ and $i'$ is the associated cutting morphism, it holds $i'\circ i=i\circ i'$.
\end{proposition}
\begin{proof}
Since by Proposition \ref{prop:algebrainjection} we already know that $i$ is a well defined injective algebra morphism, it is sufficient to check the statement for a system of stated transverse smooth curves $\{\gamma_i\in I\}$ which generate $\chi(\fS)$ as an algebra. By the proof of Lemma \ref{lem:twistedbundles}, we can choose a finite system of such $\gamma_i$ such that $|\gamma_i\cap c|\leq 2, \forall i$. 
Let $\alpha\in \{\gamma_i, i\in I\}$ be represented by a smooth immersion $\alpha:[0,1]\to \fS$ intersecting transversally $c$ with states $st(\alpha(0))=\epsilon, st(\alpha(1))=\eta$; if $|\alpha\cap c|=0$ the statement is true. 
If $|\alpha\cap c|=1$ then $\alpha=\alpha_2\circ \alpha_1$ where $\alpha_i$ are transverse smooth simple curves with $\alpha_1(1)=\alpha_2(0)\in c$ and are partially stated by $st(\alpha_1(0))=\epsilon, st(\alpha_2(1))=\eta$. Furthermore, up to switching $\alpha$ to $\alpha^{-1}$ we can suppose $(\dot{\alpha}_1(1),\dot{c})$ form a positive basis of $\fS$.   
\begin{figure}
\includegraphics[width=5cm]{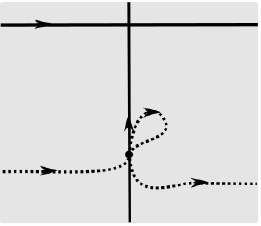}\qquad \includegraphics[width=5cm]{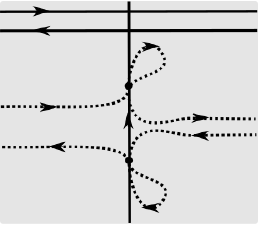}
\caption{On the l.h.s. the curve $\alpha$ (horizontal, solid) intersects once $c$ (vertical, solid). The dotted curve is regularly homotopic to $\alpha$ and is cut by $c$ into $\alpha_2\circ  \sqrt{\OO}^{-1}$ and $\alpha_1$ where $\alpha_i$ are in good position in $\fS'$.   On the r.h.s. $\alpha$ intersects twice $c$ with opposite orientations. The dotted curve is regularly homotopic to $\alpha$ and is cut by $c$ in $\alpha_3$, $\sqrt{\OO}^{-1}\circ \alpha_2\circ \sqrt{\OO}^{-1}$ and $\alpha_1$, where $\alpha_i$ are in good position in $U\fS'$.  }\label{fig:cuttrace}
\end{figure}

Let then $A_i$ (resp. $A$) be the $2\times 2$ matrix expressing the values of $\tr(\alpha_i)$ (resp. $\tr(\alpha)$) with states in $\{\pm\}$ as in Example \ref{ex:traces}; then, as remarked in the example $A_i=\rho(\sqrt{\OO_c})^{-1}\rho(\widehat{\alpha_i})$ (resp. $A=\rho(\sqrt{\OO_c})^{-1}\rho(\widehat{\alpha})$) so that Equation \eqref{eq:cuttrace} rewrites in this case as $A=A_2\cdot A_1$.

Now since the orientation induced by $\pr^{-1}(c)$ is negative on $c_2$ then in $U\fS$ the good lift $\hat{\alpha}$ of $\alpha$ is homotopic to $\tpr(\widehat{\alpha_2})\circ \sqrt{\OO_c}^{-1}\circ \tpr(\widehat{\alpha_1})$ where $\alpha_i$ are depicted in the left hand side of Figure \ref{fig:cuttrace}, therefore $$A=\rho(\sqrt{\OO_c})^{-1}\rho(\widehat{\alpha})=\rho(\sqrt{\OO_c})^{-1}\rho(\widehat{\alpha_2})\circ \rho(\sqrt{\OO_c})^{-1}\circ \rho(\widehat{\alpha_1})=A_2\cdot A_1$$
and the claim is proved. 

Suppose now that $|\alpha\cap c|=2$ where $\alpha$ is a stated smooth immersion transverse to $c$; by the proof of Lemma \ref{lem:twistedbundles} we can suppose that the sign of the intersections of $\alpha$ and $c$ is opposite and we can split $\alpha$ as $\alpha_3\circ \alpha_2\circ \alpha_1$ where $\alpha_i$ are transverse smooth immersions with $\alpha_1(1)=\alpha_2(0),\alpha_2(1)=\alpha_3(0)$ and partially stated so that $st(\alpha_1(0))=\epsilon$ and $st(\alpha_3(1))=\eta$. Furthermore, up to switching $\alpha$ and $\alpha^{-1}$ we can suppose that $(\dot{\alpha}_1(1),\dot{c})$ form a positive basis of $\fS$.

As above let $A_i=\rho(\OO_c)^{-1}\cdot \rho(\widehat{\alpha_i})$ and $A=\rho(\OO_c)^{-1}\cdot \rho(\widehat{\alpha})$ and Equation \eqref{eq:cuttrace} is equivalent to $A=A_3\cdot A_2\cdot A_1$. 
Then again, as shown in the right hand side of Figure \ref{fig:cuttrace}, $\alpha$ is regularly homotopic in $U\fS$ to $\tpr(\widehat{\alpha_3})\circ \sqrt{\OO_c}^{-1}\circ \tpr(\widehat{\alpha_2})\circ \sqrt{\OO_c}^{-1}\circ \tpr(\widehat{\alpha_1})$. Therefore we have:
$$A=\rho(\OO_c)^{-1}\cdot \rho(\widehat{\alpha})=\rho(\OO_c)^{-1}\cdot \rho(\widehat{\alpha_3})\cdot \rho(\OO_c)^{-1}\cdot \rho(\widehat{\alpha_2})\cdot \rho(\OO_c)^{-1}\cdot \rho(\widehat{\alpha_1})=A_3\cdot A_2\cdot A_1$$
and the thesis follows. 
\end{proof}

\subsection{The classical limit of stated skein algebras}

\begin{theorem}\label{teo:classicaliso}  Suppose $q^{1/2}=1$. The map sending a skein to its trace induces an algebra isomorphism $$\tr:\cSs(\fS)\to \chi(\fS).$$
\end{theorem}
\begin{proof} We first claim that the relations $\eqref{eq.skein},\eqref{eq.loop},\eqref{eq.arcs},\eqref{eq.order}$ with $q^{\frac{1}{2}}=1$ are satisfied by the trace functions.
By Lemma \ref{lem:classicalbigon} the claim is true for bigons. 
But by Proposition \ref{prop:algebrainjection} cutting induces an injective algebra map, thus to verify local relations we can verify them in a bigon containing the disc where the relations are depicted: this proves the claim in general. 

The algebra isomorphism is proved as follows: 
pick an ideal triangulation of $\fS$ and apply to each edge of the triangulation Proposition \ref{prop:cuttrace} on the side of $\chi(\fS)$ and Theorem \ref{thm.1a} on $\cSs(\fS)$. We get the following diagram of algebra morphisms of which the horizontal lines are injective and which is commutative by Proposition \ref{prop:cuttrace} and Theorem \ref{thm.1a}: 
\be
\label{eq.dia3}
 \begin{CD} \cSs(\fS)  @> \hookrightarrow  >>  \bigotimes_{i}\cSs(T_i)\\
 @V
 \tr    VV  @V V \tr   V    \\
  \chi(\fS)    @> \hookrightarrow  >>   \bigotimes_i \chi(T_i) .
\end{CD}
\ee

Since by Lemma \ref{lem:classicalbigon} the right vertical arrow is an isomorphism we conclude. 

\end{proof}

\end{document}